\documentclass[11pt,a4paper,reqno]{amsart}
\usepackage[utf8]{inputenc}
\usepackage[british]{babel}

\usepackage[svgnames]{xcolor}

\usepackage[colorlinks=true,linkcolor=Blue,citecolor=Green,urlcolor=Purple]{hyperref}
\hypersetup{pdftitle={Gevrey Regularity of weak solutions of the homogeneous Boltzmann equation without cutoff for Maxwellian molecules}, pdfauthor={Barbaroux, Hundertmark, Ried, Vugalter}, pdfsubject={35D10 (Primary); 35B65, 35Q20, 82B40 (Secondary)}, pdfkeywords={Gevrey regularity, Non-cutoff homogeneous Boltzmann equation, Maxwellian molecules}} 

\usepackage{amsmath,amssymb,amsthm,amsfonts}
\usepackage{dsfont}
\usepackage[margin={2.9cm}]{geometry}
\usepackage[full]{textcomp}
\usepackage[T1]{fontenc}
\usepackage{txfonts}
\usepackage[cal=boondoxo]{mathalfa}
\usepackage{microtype}
\usepackage[inline]{enumitem}
\usepackage{tikz}
\usetikzlibrary{arrows}
\newcommand{\R}{\mathbb{R}}

\newcommand{\N}{\mathbb{N}}

\newcommand{\1}{\mathds{1}}
\newcommand{\I}{\mathrm{i}}
\renewcommand{\S}{\mathbb{S}}

\theoremstyle{plain}
\newtheorem{proposition}{Proposition}[section]
\newtheorem{corollary}[proposition]{Corollary}
\newtheorem{theorem}[proposition]{Theorem}
\newtheorem*{theorem*}{Theorem}
\newtheorem{lemma}[proposition]{Lemma}
\newtheorem*{conjecture*}{Conjecture}

\theoremstyle{definition}
\newtheorem{definition}[proposition]{Definition}
\newtheorem{remark}[proposition]{Remark}
\newtheorem{remarks}[proposition]{Remarks}
\newtheorem*{remarks*}{Remarks}

\begin{document}
\title[Gevrey smoothing for the homogeneous non-cutoff Boltzmann equation]{Gevrey smoothing for weak solutions of the fully nonlinear homogeneous Boltzmann and Kac equations without cutoff for Maxwellian molecules}
\author{Jean-Marie Barbaroux}
\author{Dirk Hundertmark}
\author{Tobias Ried}
\author{Semjon Vugalter}
\date{21st September 2015}
\thanks{\textcopyright~2015 by the authors. Faithful reproduction of this article,
        in its entirety, by any means is permitted for non-commercial purposes}
	\begin{abstract}
	  It has long been suspected that the non-cutoff Boltzmann operator has similar coercivity properties as a fractional Laplacian. This has led to the hope that the homogenous Boltzmann equation enjoys similar regularity properties as the heat equation with a fractional Laplacian.
	   In particular, the weak solution of the fully nonlinear non-cutoff homogenous Boltzmann equation with initial datum in $L^1_2(\mathbb{R}^d)\cap L\log L(\mathbb{R}^d)$, i.e., finite mass, energy and entropy, should immediately become Gevrey regular for strictly positive times. We prove this conjecture for Maxwellian molecules. 
	\end{abstract}
\maketitle
{\noindent
}
{\hypersetup{linkcolor=black}
\tableofcontents}
\section{Introduction}\label{sec:introduction}

   It has long been suspected that the non-cutoff Boltzmann operator with a singular cross section kernel has similar coercivity properties as a fractional Laplacian $(-\Delta)^\nu$, for suitable $0<\nu<1$. This has been made precise by \textsc{Alexandre, Desvillettes, Villani}, and \textsc{Wennberg} \cite{ADVW00}, see also the reviews by \textsc{Alexandre} \cite{Ale09} and by \textsc{Villani} \cite{Vil02} for its history, and has led to the hope that the fully nonlinear  homogenous Boltzmann equation enjoys similar regularity properties as the heat equation with a fractional Laplacian given by
		\begin{align*}
			\left\{\begin{array}{rl}
					\partial_t u + (-\Delta)^{\nu}u &= 0 \\
					\left.u\right|_{t=0} &= u_0 \,\in L^1(\R^d) .
			\end{array}	\right.
		\end{align*}
 Using the Fourier transform one immediately sees that
	\begin{align*}
	\widehat{u}(t,\xi) = \mathrm{e}^{-t (2\pi |\xi|)^{2\nu}} \widehat{u_0}(\xi) \quad \text{with}  \quad \widehat{u_0} \in L^{\infty}(\R^d),
	\end{align*}
	 so
	\begin{align*}
		\sup_{t>0} \sup_{\xi\in\R^d} \mathrm{e}^{t |\xi|^{2\nu}}|\widehat{u}(t,\xi)| \leq \|u_0\|_{L^1(\R^d)} <\infty,
	\end{align*}
	that is, the Fourier transform of the solution is extremely fast decaying for strictly positive times.

  Introducing the Gevrey spaces as in Definition \ref{def:Gevrey}, it is natural to expect,   see, for example, \textsc{Desvillettes} and \textsc{Wennberg} \cite{DW10}:
   \begin{conjecture*}[Gevrey smoothing]
    Any weak	 solution of the non-cutoff homogenous Boltzmann equation with a singular cross section kernel of order $\nu$ and with initial datum in $L^1_2(\mathbb{R}^d)\cap L\log L(\mathbb{R}^d)$, i.e., finite mass, energy and entropy, belongs to the Gevrey class $G^{\frac{1}{2\nu}}(\R^d)$ for strictly positive times.
   \end{conjecture*}

   The central results of our work is a proof of this conjecture for Maxwellian molecules. In particular, we prove
   \begin{theorem*} Assume that the non-cutoff Boltzman cross section has a singularity $1+2\nu$ with $0<\nu<1$ and obeys some further technical conditions, which are true in all physically relevant cases, for details see \eqref{eq:cross-section} and \eqref{eq:cross-section-bdd}. Then, for initial conditions $f_0\in L\log L\cap L^1_m$ with an integer
   		\begin{align*}
   			m\ge \max\left(2,\frac{2^\nu-1}{2(2-2^{\nu})}\right)
   		\end{align*}
   		any weak solution of the fully non-linear homogenous Boltzmann equation for Maxwellian molecules belongs to the Gevrey class $G^{\frac{1}{2\nu}}$ for strictly positive times.

   		In particular, for $\nu \le \log(9/5)/\log(2) \simeq 0,847996$ we have
   		$m=2$ and the theorem does not require anything except the physically reasonable assumptions of finite mass, energy, and entropy. If $\log(9/5)/\log(2)<\nu<1$ and we assume only that $f_0\in L\log L\cap L^1_2$, then we prove that the solution is in $G^{\frac{\log2}{2\log(9/5)}}$, in particular, it is ultra-analytic.
   \end{theorem*}

 \begin{enumerate}
 	\item For a more precise formulation of our results, see Theorems \ref{thm:gevrey-main1}, \ref{thm:gevrey-main2}, and \ref{thm:gevrey-main3} for the case $m=2$ and Theorems \ref{thm:gevrey-main1-m},  \ref{thm:gevrey-main2-m}, and \ref{thm:gevrey-main3-m} below.
 	\item We would like to stress that our results cover both the weak and strong singularity regimes, where $0<\nu<1/2$, respectively $1/2\le \nu<1$.
 	\item The theorem above applies to all dimensions $d\geq 1$. The physical case for Maxwellian molecules in dimension $d=3$ is $\nu=1/4$.
 \end{enumerate}

   The main problem for establishing Gevrey regularity is that, in order to use the coercivity results of \textsc{Alexandre, Desvillettes, Villani} and \textsc{Wennberg} \cite{ADVW00}, one has to bound a non-linear and non-local commutator of the Boltzmann kernel with certain sub-Gaussian Fourier multipliers.
   The main ingredient in our proof is a new way of estimating this non-local and nonlinear commutator.

\subsection{The non-cutoff Boltzmann and Kac models}
We study the regularity of weak solutions of the Cauchy problem
\begin{align}\label{eq:cauchyproblem}	
\begin{cases}
\partial_t f = Q(f,f) & \\
f|_{t=0} = f_0 &
\end{cases}
\end{align}
for the fully nonlinear homogeneous Boltzmann and Kac equation in $d\geq 1$ dimensions \cite{Bol72,Kac59}.

For $d\geq 2$ the bilinear operator $Q$ is given by
\begin{align}\label{eq:Boltzmann-kernel}
	Q(g,f) = \int_{\R^d} \int_{\S^{d-1}} b(\cos\theta) \left( g(v'_*) f(v') - g(v_*) f(v) \right) \, \mathrm{d}\sigma \mathrm{d} v_*,
\end{align}
that is, the Boltzmann collision operator for Maxwellian molecules with angular collision kernel $b$ depending only on the deviation angle $\cos\theta = \sigma \cdot \frac{v-v_*}{|v-v_*|}$ for $\sigma\in\mathbb{S}^{d-1}$. Here we use the $\sigma$-representation of the collision process, in which
\begin{align*}
	v' = \frac{v+v_*}{2} + \frac{|v-v_*|}{2}\sigma, \quad v'_* = \frac{v+v_*}{2} - \frac{|v-v_*|}{2}\sigma, \quad \text{for } \sigma \in \S^{d-1}.
\end{align*}

By symmetry properties of the Boltzmann collision operator $Q(f,f)$, the function $b$ can be assumed to be supported on angles $\theta \in [0, \frac{\pi}{2}]$, for otherwise, see \cite{Vil02}, it can be replaced by
	\begin{align*}
	\widetilde{b}(\cos\theta) = \left( b(\cos\theta) + b(\cos(\pi-\theta) \right) \1_{\{0\leq \theta\leq \frac{\pi}{2}\}}.
	\end{align*}

We will assume that the angular collision kernel $b$ has the non-integrable singularity
\begin{align}
	\label{eq:cross-section} \sin^{d-2}\theta \, b(\cos\theta) \sim \frac{\kappa}{\theta^{1+2\nu}},\quad \text{as } \theta\to 0^+
\end{align}
for some $\kappa>0$ and $0 < \nu <1$, and satisfies
\begin{align}\label{eq:cross-section2}
	\int_0^{\pi/2} \sin^d\theta \, b(\cos\theta)\,\mathrm{d}\theta <\infty.
\end{align}

For inverse $s$-power forces (in three spatial dimensions), described by the potential $U(r) = r^{1-s}$, $s>2$, the collision kernel is of the more general form
\begin{align*}
	B(|v-v_*|, \cos\theta) = b(\cos\theta) |v-v_*|^{\gamma}, \quad \gamma = \frac{s-5}{s-1},
\end{align*}
where the angular collision kernel $b$ is locally smooth with a non-integrable singularity
\begin{align*}
	\sin\theta \, b(\cos\theta) \sim K \theta^{-1-2\nu}, \quad \nu = \frac{1}{s-1}.
\end{align*}
The case of \emph{(physical) Maxwellian molecules} corresponds to the values $\gamma=0$, $s=5$, $\nu=\frac{1}{4}$.

For $d=1$ we set
\begin{align}
	Q(g,f)\,\, = K(g,f) = \int_{\R} \int_{-\tfrac{\pi}{2}}^{\tfrac{\pi}{2}} b_1(\theta) \left( f(w'_*) g(w') - f(w_*) g(w) \right) \, \mathrm{d}\theta \mathrm{d} w_*,
\end{align}
which is the Kac operator for Maxwellian molecules, and angular collision kernel $b_1\geq 0$.
The pre- and post-collisional velocities are related by
\begin{align*}
	\begin{pmatrix} w' \\ w'_* \end{pmatrix} = \begin{pmatrix} \cos \theta & - \sin\theta \\  \sin\theta & \cos\theta \end{pmatrix} \begin{pmatrix} w \\ w_* \end{pmatrix}, \quad \text{for } \theta \in [-\tfrac{\pi}{2}, \tfrac{\pi}{2}].
\end{align*}

In the original Kac model $b_1$ was chosen to be constant, whereas we will assume, as in \cite{Des03}, that $b_1$ is an even function and has the non-integrable singularity
\begin{align}\label{eq:cross-section-kac}
	b_1(\theta) \sim \frac{\kappa}{|\theta|^{1+2\nu}}, \quad \text{for } \theta \to 0,
\end{align}
with $0<\nu<1$ and some $\kappa>0$, and further satisfies
\begin{align}\label{eq:cross-section-kac2}
	\int_{-\tfrac{\pi}{2}}^{\tfrac{\pi}{2}} b_1(\theta) \sin^2\theta \,\mathrm{d}\theta < \infty.
\end{align}

Making use of symmetry properties of the collision operator $K(f,f)$, we can assume $b_1$ to be supported on angles $\theta\in[ -\frac{\pi}{4}, \frac{\pi}{4}]$, for otherwise it can be replaced by its symmetrised version
\begin{align*}
	\widetilde{b_1}(\theta) = \left( b_1(\theta) + b_1(\tfrac{\pi}{2}-\theta) \right) \1_{\{0\leq \theta \leq \frac{\pi}{4}\}} + \left( b_1(\theta) + b_1(-\tfrac{\pi}{2}-\theta) \right) \1_{\{-\frac{\pi}{4}\leq \theta \leq 0\}}.
\end{align*}
This simple observation will be very convenient for our analysis.

We will mainly work with the weighted $L^p$ spaces, defined as
\begin{align*}
	L^p_{\alpha}(\R^d) := \left\{ f\in L^p(\R^d): \langle \cdot \rangle^{\alpha} f \in L^p(\R^d) \right\},
\end{align*}
$p\geq 1$, $\alpha \in \R$, with norm
\begin{align*}
	\| f\|_{L^p_{\alpha}(\R^d)} = \left( \int_{\R^d} |f(v)|^p \langle v \rangle^{\alpha p} \,\mathrm{d}v \right)^{1/p}, \quad \langle v \rangle := (1+ |v|^2)^{1/2}.
\end{align*}

We will also use the weighted ($L^2$ based) Sobolev spaces
\begin{align*}
	H^k_{\ell}(\R^d) = \left\{ f\in \mathcal{S}'(\R^d): \langle \cdot \rangle^{\ell} f \in H^k(\R^d) \right\}, \quad k,\ell \in \R,
\end{align*}
where $H^k(\R^d)$ are the usual Sobolev spaces
  given by 	$H^k (\R^d) = \left\{ f\in \mathcal{S}'(\R^d): \langle \cdot \rangle^{k} \hat{f} \in L^2(\R^d) \right\}$, for $k  \in \R$.
The inner product on $L^2(\R^d)$ is given by $\langle f,g\rangle = \int_{\R^d} \overline{f(v)} g(v) \,\mathrm{d}v$.


It will be assumed that the initial datum $f_0\not\equiv 0$ is a non-negative density with finite mass, energy and entropy, which is equivalent to
\begin{align}\label{eq:initialdata}
	f_0\geq 0, \quad f_0\in L^1_2(\R^d)\cap L\log L(\R^d),
\end{align}
where
\begin{align*}
	L\log L(\R^d) = \left\{f:\R^d\to\R \text{ measurable}: \|f\|_{L\log L} = \int_{\R^d} |f(v)| \log\left(1+|f(v)|\right)\,\mathrm{d}v < \infty\right\},
\end{align*}
and the negative of the entropy is given by $H(f):=\int_{\R^d} f\log f\, \mathrm{d}v$.

The space $L^1_2(\R^d)\cap L\log L(\R^d)$ is very natural, since
\begin{lemma}\label{lem:entropy}
	Let $f\geq 0$. Then
	\begin{align*}
		f\in L^1_2(\R^d)\cap L\log L(\R^d) \quad  \Leftrightarrow  \quad f\in L^1_2(\R^d) \text{ and } H(f) \text{ is finite}.
	\end{align*}
\end{lemma}
We suspect that this lemma is well-known, at least to the experts, but we could not find a reference in the literature. For the reader's convenience we will give the proof in appendix \ref{sec:appendix-LlogL}. Following is the precise definition of weak solutions which we use.

\begin{definition}[Weak Solutions of the Cauchy Problem \eqref{eq:cauchyproblem} \cite{Ark81,Vil98,Des95}]\label{def:weaksolution}
	Assume that the initial datum $f_0$ is in $L^1_2(\R^d)\cap L\log L(\R^d)$. $f: \R_+ \times \R^d \to \R$ is called a weak solution to the Cauchy problem \eqref{eq:cauchyproblem}, if it satisfies the following conditions\footnote{Throughout the text, whenever not explicitly mentioned, we will drop the dependence on $t$ of a function, i.e. $f(v):= f(t,v)$ etc}:
	\begin{enumerate}[label=(\roman*)]
		\item $f\geq 0$,\,\,\,  $f\in\mathcal{C}(\R_+; \mathcal{D}'(\R^d))\cap  L^{\infty}(\R_+;L^1_2(\R^d)\cap L \log L(\R^d))$  
		\item $f(0,\cdot) = f_0$
		\item For all $t\geq 0$, mass is conserved, $\int_{\R^d} f(t,v) \, \mathrm{d} v = \int_{\R^d} f_0(v) \, \mathrm{d}v$, kinetic energy is decreasing, $\int_{\R^d} f(t,v)\, v^2 \, \mathrm{d} v \leq \int_{\R^d} f_0(v) \, v^2 \, \mathrm{d}v$, and the entropy is increasing, $H(f(t,\cdot))\leq H(f_0)$.
		\item For all $\varphi\in\mathcal{C}^1(\R_+; \mathcal{C}_0^{\infty}(\R^d))$ one has
		\begin{align}\label{eq:weakformulation}
		\begin{split}
			&\langle f(t,\cdot), \varphi(t,v) \rangle - \langle f_0 , \varphi(0,\cdot)\rangle - \int_0^t \langle f(\tau, \cdot) \partial_{\tau}\varphi(\tau,\cdot)\rangle \,  \mathrm{d}\tau \\
			&\quad = \int_0^t \langle Q(f,f)(\tau, \cdot), \varphi(\tau,\cdot) \rangle \, \mathrm{d}\tau, \quad \text{for all } t\geq 0,
			\end{split}
		\end{align}
		where the latter expression involving $Q$ is defined by
		\begin{align*}
			&\langle Q(f,f), \varphi \rangle \\
			&\quad = \frac{1}{2}\int_{\R^{2d}}\int_{\S^{d-1}} b\left(\frac{v-v_*}{|v-v_*|}\cdot \sigma\right) f(v_*) f(v) \left( \varphi(v')+\varphi(v'_*) - \varphi(v) - \varphi(v_*)\right) \, \mathrm{d}\sigma \mathrm{d}v\mathrm{d}v_*,
		\end{align*}
		for test functions $\varphi \in W^{2,\infty}(\R^d)$ in dimension $d\geq 2$, and in one dimension
		\begin{align*}
		\langle Q(f,f), \varphi\rangle = \langle K(f,f), \varphi \rangle =\int_{\R^2} \int_{-\tfrac{\pi}{4}}^{\tfrac{\pi}{4}} b_1(\theta) \, g(w_*) g(w) \left(\phi(w') - \phi(w)\right) \, \mathrm{d}\theta\mathrm{d}w  \mathrm{d}w_*
		\end{align*}
		for test functions $\varphi \in W^{2,\infty}(\R)$, making use of symmetry properties of the Boltzmann and Kac collision operators and cancellation effects.
	\end{enumerate}
\end{definition}

Collecting results from the literature, the following is known regarding the existence, uniqueness and further properties of weak solutions.

\begin{theorem}[Arkeryd, Desvillettes, Mischler, Goudon, Villani, Wennberg]\label{thm:existence}
There exists a weak solution of the Cauchy problem \eqref{eq:cauchyproblem} in the sense of Definition \ref{def:weaksolution}.
For $d\geq 2$ momentum and energy are conserved,
\begin{align}
	\int_{\R^d} f(t,v) \, v \,\mathrm{d}v = \int_{\R^d} f_0(v)\, v \,\mathrm{d}v, \quad 	\int_{\R^d} f(t,v) \,v^2 \,\mathrm{d}v = \int_{\R^d} f_0(v)\, v^2 \,\mathrm{d}v.
\end{align}
In the one dimensional case (Kac equation), momentum is not conserved and energy can only decrease and is conserved under the additional moment assumption $f_0\in L^1_{2p}$ for some $p\geq2$.
\end{theorem}

\begin{remark}
$d\geq 2$: The existence of weak solutions of the Cauchy problem \eqref{eq:cauchyproblem} with initial conditions satisfying \eqref{eq:initialdata} for the homogeneous Boltzmann equation was first proved by \textsc{Arkeryd} \cite{Ark72,Ark81} (see also the articles by \textsc{Goudon} \cite{Gou97}, \textsc{Villani} \cite{Vil98}, and \textsc{Desvillettes} \cite{Des01,Des03}). Uniqueness in this case was shown by \textsc{Toscani} and \textsc{Villani} \cite{TV99}, see  also the review articles by \textsc{Mischler} and \textsc{Wennberg} \cite{MW99} (for the cut-off case) and \textsc{Desvillettes} \cite{Des01}.

$d=1$: For the homogeneous non-cutoff Kac equation for Maxwellian molecules existence of weak solutions was established by \textsc{Desvillettes} \cite{Des95}.
\end{remark}

\subsection{Higher regularity of weak solutions}

 It has been pointed out by several authors \cite{Ale09,DW10,Vil02} that, for singular cross-sections, the Boltzmann operator essentially behaves like a singular integral operator with a leading term similar to a fractional Laplace operator $(-\Delta)^{\nu}$.  In terms of compactness properties this has been noticed for the linearised Boltzmann kernel as early as in \cite{Pao74} and for the nonlinear Boltzmann kernel in \cite{Lio94}. Since the solutions of the heat equation with a fractional Laplacian gain a high amount of regularity for arbitrary positive times, it is natural to believe, as conjectured in \cite{DW10}, that weak solutions to the non-cutoff Boltzmann equation gain a certain amount of smoothness, and even analyticity, for any $t>0$.
 	This is in sharp contrast to the fact that in the Grad's cutoff case there cannot be any smoothing effect. Instead, regularity and singularities of the initial datum get propagated in this case, see, for example,
	\cite{MV04}.

The discussion about solution of the heat equation with a fractional Laplacian motivates the following definition of Gevrey spaces, which give
    a convenient framework to describe this smoothing by interpolating between smooth  and (ultra-)analytic functions.

\begin{definition}\label{def:Gevrey}

	Let $s>0$. A function $f$ belongs to the Gevrey class $G^s(\R^d)$, 	
	if there exists an $\epsilon_0>0$ such that
	\begin{align*}
		e^{\epsilon_0 \langle D_v \rangle^{1/s}} f \in L^2(\R^d)\,, \quad  \text{where} \quad \langle D_v \rangle = \left(1+ |D_v|^2 \right)^{1/2}.
	\end{align*}
	and we use the notation $D_v=-\frac{\I}{2\pi} \nabla_v$. Thus, $G^1(\R^d)$ is the space of real analytic functions, and $G^s(\R^d)$ for $s\in(0,1)$ the space of ultra-analytic functions.
	
	Equivalently\footnote{see, for example, Theorem 4 in \cite{LO97}.}, $f\in G^s(\R^d)$ if
	$f\in \mathcal{C}^{\infty}(\R^d)$ and there exists a constant $C>0$ such that for all $k\in\N_0$ one has
	\begin{align*}
		\| D^k f\|_{L^2(\R^d)} \leq C^{k+1} (k!)^s,
	\end{align*}
	where $\| D^k f\|_{L^2}^2 = \sup_{|\beta|=k} \| \partial^{\beta} f\|_{L^2}^2$.
\end{definition}

The first regularisation results in this direction were due to \textsc{Desvillettes} for the spatially homogeneous non-cutoff Kac equation \cite{Des95} and the homogeneous non-cutoff Boltzmann equation for Maxwellian molecules in two dimensions \cite{Des97}, where $\mathcal{C}^{\infty}$ regularisation is proved. Later,
\textsc{Desvillettes} and \textsc{Wennberg} \cite{DW10} proved, under rather general assumptions on the collision cross-section (excluding Maxwellian molecules, though), regularity in Schwartz space of weak solutions to the non-cutoff homogeneous Boltzmann equation.
By quite different methods, using Littlewood-Paley decompositions, \textsc{Alexandre} and \textsc{El Safadi} \cite{AE05} showed that the assumptions on the cross-section \eqref{eq:cross-section}-\eqref{eq:cross-section2} imply that the solutions are in $H^{\infty}$ for any positive time $t>0$. By moment propagation results for Maxwellian molecules (see \textsc{Truesdell} \cite{Tru56}) this cannot be improved to regularity in Schwartz space.

 For collision cross-sections corresponding to Debye-Yukawa-type interaction potentials,
 \begin{align*}
 	\sin\theta\,b(\cos\theta) \sim K \theta^{-1}(\log\theta^{-1})^{\ell} \quad  \text{for } \theta\to0 \quad \text{(with some } K>0,\, \ell>0\text{)},
 	\end{align*}
 \textsc{Morimoto, Ukai, Xu} and \textsc{Yang} \cite{MUXY09} proved the same $H^{\infty}$ regularising effect using suitable test functions in the weak formulation of the problem.

\bigskip

The question of the local existence of solutions in Gevrey spaces for Gevrey regular initial data with additional strong decay at infinity was first addressed in 1984 by \textsc{Ukai} \cite{Uka84}, both in the spatially homogeneous and inhomogeneous setting.

We are interested in the Gevrey \emph{smoothing effect}, namely that under the (physical) assumptions of finite mass, energy and entropy of the initial data, weak solutions of the homogeneous Boltzmann equation without cutoff are Gevrey functions for any strictly positive time. This question was treated in the case of the \emph{linearised} Boltzmann equation in the homogeneous setting by \textsc{Morimoto} \textit{et al.} \cite{MUXY09}, where they proved that, given $0<\nu<1$,  weak solutions of the linearized Boltzmann equation belong to the space $G^{\frac{1}{\nu}}(\R^3)$ for any  positive times.
Still in a linearised setting, \textsc{Lerner, Morimoto, Pravda-Starov} and \textsc{Xu} \cite{LMPX14} proved a Gelfand-Shilov smoothing effect, which includes Gevrey regularity, on radially symmetric solutions of the homogeneous non-cutoff Boltzmann equation for Maxwellian molecules. For the non-Maxwellian  Boltzmann operator, Gevrey regularity was proved under very strong unphysical decay assumptions on the initial datum in \cite{Lin14}.

For radially symmetric solutions, the homogeneous non-cutoff Boltzmann equation for Maxwellian molecules is related to the homogeneous non-cutoff Kac equation. The non-cutoff Kac equation was introduced by \textsc{Desvillettes} in \cite{Des95}, where first regularity results were established, see also \textsc{Desvillettes}' review \cite{Des03}. For this equation, the best available results so far are due to \textsc{Lekrine} and \textsc{Xu} \cite{LX09} and \textsc{Glangetas} and \textsc{Najeme} \cite{GN13}:  \textsc{Lekrine} and \textsc{Xu} \cite{LX09} proved Gevrey regularisation of order $\frac{1}{2\alpha}$ for mild singularities $0<\nu<\frac{1}{2}$ and all $0<\alpha<\nu$. Strong singularities $\tfrac{1}{2} \leq \nu<1$ were treated by \textsc{Glangetas} and \textsc{Najeme} \cite{GN13}, where they prove that for $\nu=\frac{1}{2}$ the solution becomes Gevrey regular of order $\frac{1}{2\alpha}$ for any $0<\alpha<\frac{1}{2}$ and Gevrey regular of order $1$, that is, analytic, when $\frac{1}{2}<\nu<1$. Thus, in the critical case $\nu=\frac{1}{2}$, the result of \cite{GN13} misses analyticity of weak solutions and they do not prove ultra-analyticity in the range $0<\nu<1$. Moreover, both results are obtained under the \emph{additional} moment assumption $f_0\in L^1_{2+2\nu}(\R)$.

Ultra-analyticity results have previously been obtained by \textsc{Morimoto} and \textsc{Xu} \cite{MX09} for the homogeneous Landau equation in the Maxwellian molecules case and related simplified models in kinetic theory.
The analysis of smoothing properties of Landau equation is quite different from the Boltzmann and Kac equations. The Landau equation explicitly contains a second order elliptic term, which yields coercivity,  and, more importantly, certain commutators with weights in Fourier space are identically zero, which simplifies the analysis tremendously, see Proposition 2.2 in \cite{MX09}.

For the nonlinear non-cutoff homogeneous Boltzmann equation some partial results regarding Gevrey regularisation were obtained by \textsc{Morimoto} and \textsc{Ukai} \cite{MU10} including the non-Maxwellian molecules case, but under the strong additional assumptions of Maxwellian decay and smoothness of the solution. Still with these strong decay assumptions, \textsc{Yin} and \textsc{Zhang} \cite{YZ12,YZ14} extended this result to a larger class of kinetic cross-sections.

We stress that for the main result of our paper the initial datum is \emph{only} assumed to obey the \emph{natural} assumptions coming from physics, i.e., finiteness of mass, energy and entropy.

\bigskip

Given $\beta>0$ and $\alpha\in(0,1)$ we define the Gevrey multiplier $G: \R_+ \times \R^d \to \R$ by
\begin{align*}
G(t, \eta) := e^{\beta t \langle \eta\rangle^{2\alpha}}
\end{align*}
and for $\Lambda>0$ the cut-off Gevrey multiplier $G_{\Lambda}: \R_+ \times \R^d \to \R$ by
\begin{align*}
G_{\Lambda}(t,\eta):= G(t, \eta) \1_{\Lambda}(|\eta|),
\end{align*}
where $\1_{\Lambda}$ is the characteristic function of the interval $[0, \Lambda]$. The associated Fourier multiplication operator is denoted by $G_{\Lambda}(t,D_v)$,
\begin{align*}
	(G_{\Lambda}(t, D_v)f)(t,v) := \int_{\R^d} G_{\Lambda}(t,\eta)\hat{f}(t, \eta) \, e^{2\pi\I \eta\cdot v} \,\mathrm{d}\eta = \mathcal{F}^{-1}\left[ G_{\Lambda}(t,\cdot) \hat{f}(t, \cdot) \right].
\end{align*}
We use the following convention regarding the Fourier transform of a function $f$ in this article,
\begin{align*}
	(\mathcal{F} f) (\eta) = \hat{f}(\eta) = \int_{\R^d} f(v) \, e^{-2\pi\I v\cdot \eta} \,\mathrm{d}v.
\end{align*}

The Fourier transform of the Boltzmann operator for Maxwellian molecules has the form (Bobylev identity, \cite{Bob84})
\begin{align}\label{eq:bobylev}
	\widehat{Q(g,f)}(\eta) = \int_{\S^{d-1}} b\left(\frac{\eta}{|\eta|}\cdot\sigma\right) \left[ \hat{g}(\eta^-) \hat{f}(\eta^+) - \hat{g}(0) \hat{f}(\eta) \right] \, \mathrm{d}\sigma, \quad \eta^{\pm} = \frac{\eta\pm |\eta| \sigma}{2},
\end{align}
for $d\geq 2$. There is a similar Bobylev identity for the Kac operator \cite{Des95},
\begin{align}
	\widehat{K(g,f)}(\eta) = \int_{-\frac{\pi}{4}}^{\frac{\pi}{4}} b_1(\theta) \left[ \hat{g}(\eta^-)\hat{f}(\eta^+) - \hat{g}(0) \hat{f}(\eta)\right] \,\mathrm{d}\theta, \quad \eta^+ = \eta \cos\theta, \eta^- = \eta \sin\theta.
\end{align}
A simple, but in a sense important, consequence of Bobylev's identity is that, for all $d\geq 1$,
\begin{equation}\label{eq:quasilocal}
	P_\Lambda Q(g,f) = P_\Lambda Q(P_\Lambda g, P_\Lambda f)
\end{equation}
where, for convenience, we put $P_\Lambda:= \1_\Lambda(D_v)$ for the orthogonal projection onto Fourier 'modes' $|\eta|\le \Lambda$.

Note also that, since $G_{\Lambda}(t,\cdot)$ has compact support in $\R^d_{\eta}$ for any $t>0$, one has
\[G_{\Lambda}f, G_{\Lambda}^2 f \in L^{\infty}([0,T_0]; H^{\infty}(\R^d))\]
for any finite $T_0>0$ and $\Lambda>0$, if $f\in L^{\infty}([0,T_0];L^1(\R^d))$. This holds, since
\begin{align*}
\|G_{\Lambda} f \|_{H^s(\R^d_v)}^2 \leq \| \hat{f} \|_{L^{\infty}(\R^d_{\eta})}^2 \| \langle \cdot \rangle^s G_{\Lambda}(t, \cdot)\|_{L^2(\R^d_{\eta})}^2 \leq \| f \|_{L^1(\R^d_v)}^2 \| \langle\cdot\rangle^{s} G_{\Lambda}(T_0, \cdot) \|_{L^2(\R^d_{\eta})}^2, \quad \text{ for all } s\geq 0.
\end{align*}
These functions, due to the cut-off in Fourier space, are even analytic in a strip containing $\R^d_v$.

\begin{theorem}[Gevrey smoothing I]\label{thm:gevrey-main1}
		Assume that the cross-section $b$ satisfies the \emph{singularity condition} \eqref{eq:cross-section} and the \emph{integrability condition} \eqref{eq:cross-section2} for $d\geq 2$, and for $d=1$, $b_1$ satisfies the \emph{singularity condition} \eqref{eq:cross-section-kac} and the \emph{integrability condition} \eqref{eq:cross-section-kac2} for some $0<\nu<1$. Let $f$ be a weak solution of the Cauchy problem \eqref{eq:cauchyproblem} with initial datum satisfying conditions \eqref{eq:initialdata}. Then, for all $0<\alpha\leq \min\left\{\alpha_{2,d}, \nu\right\}$,
	\begin{align}
		f(t,\cdot)\in G^{\tfrac{1}{2\alpha}}(\R^d)
	\end{align}
	for all $t>0$, where $\alpha_{2,d} = \frac{\log[(8+d)/(4+d)]}{\log 2}$. 
\end{theorem}

\begin{remarks}
\begin{enumerate}[label=(\roman*)]

	\item In numbers, $\alpha_{2,1} \simeq 0.847997$, $\alpha_{2,2} \simeq 0.736966$, and $\alpha_{2,3} \simeq 0.652077$. This means, that under \emph{only} physically reasonable assumptions of finite mass, energy, and entropy, weak solutions are analytic for $\nu\geq\tfrac{1}{2}$ and even ultra-analytic if $\nu>\frac{1}{2}$. It is easy to see that $\alpha_{2,d}$ is decreasing in $d$ and for $d=6$, $\alpha_{2,6} \simeq 0.485427$, hence, for $d\geq 6$, analyticity (respectively ultra-analyticity) does not follow from this theorem.
	\item For the proof of Theorem \ref{thm:gevrey-main1} (and also \ref{thm:gevrey-main2} and \ref{thm:gevrey-main3} below) it is important that the energy of $f$ is bounded, which enters in the technical Lemma \ref{lem:1dlemma} and its Corollary \ref{cor:ddlemma}. A considerably simpler proof could be given using only that $f\in L^1_1(\R^d)$. In this case, $\alpha_{2,d}$ is replaced by $\alpha_{1,d} = \frac{\log[(4+d)/(2+d)]}{\log 2}$ (see also Remark \ref{rem:general-m} below). However, $\alpha_{1,3} < 0.4855$ in three dimensions,  thus we would not be able to conclude (ultra-)analytic smoothing of weak solutions for strong singularities $\frac{1}{2}\leq \nu <1$.
    \item As our theorem above shows, weak solutions of the homogenous Kac equation become Gevrey regular for strictly positive times for moderately singular collision kernels with singularity $\nu\in (0,\frac{1}{2})$, see \eqref{eq:cross-section-kac} for the precise description of the singularity, for $\nu = \frac{1}{2}$ they become analytic, which improves the result of \textsc{Glangetas} and \textsc{Najeme} \cite{GN13} in this critical case, and even ultra-anaytic for $\nu\in (\frac{1}{2}, 1)$.
	\item Rotationally symmetric solutions $f$ corresponding to rotationally symmetric initial conditions $f_0$ are Gevrey regular for strictly positive times under the same conditions as in the one-dimensional case $d=1$. The proof is exactly as the proof of Theorem \ref{thm:gevrey-main1-m} with some small changes in the proof of Lemma \ref{lem:induction2} where the independence of the solution $f$ on the angular coordinates can be explicitly used with the $n=1$ version of Corollary \ref{cor:ddlemma}.
\end{enumerate}
\end{remarks}

As already remarked, the result of Theorem \ref{thm:gevrey-main1} deteriorate in the dimension. Under the same assumptions, but using quite a bit more structure of the Boltzmann operator, we can prove a dimension independent version. Its proof is considerably more involved than the proof of Theorem \ref{thm:gevrey-main1}.

\begin{theorem}[Gevrey smoothing II]\label{thm:gevrey-main2}
	Let $d\ge 2$. Assume that the cross-section $b$ satisfies the conditions of Theorem \ref{thm:gevrey-main1}. Let $f$ be a weak solution of the Cauchy problem \eqref{eq:cauchyproblem} with initial datum satisfying conditions \eqref{eq:initialdata}. Then, for all $0<\alpha\leq \min\left\{\alpha_{2,2}, \nu\right\}$,
	\begin{align}
		f(t,\cdot)\in G^{\tfrac{1}{2\alpha}}(\R^d)
	\end{align}
	for all $t>0$, where $\alpha_{2,2} = \frac{\log(5/3)}{\log 2} \simeq 0.736966$. In particular, in contrast to Theorem~\ref{thm:gevrey-main1}, the weak solution is real analytic if $\nu=\frac{1}{2}$ and ultra-analytic if $\nu>\frac{1}{2}$ in \emph{any dimension}.
\end{theorem}

If the integrability conditions \eqref{eq:cross-section2} is replaced by the slightly stronger condition that $b(\cos\theta)$ is bounded away from $\theta=0$, that is,
\begin{align}\label{eq:cross-section-bdd}
	\text{for any } 0<\theta_0<\tfrac{\pi}{2} \text{ there exists } C_{\theta_0}<\infty
	\text{ such that } 0\le b(\cos\theta) \le C_{\theta_0}   \text{ for all } \theta_0\le \theta\le \tfrac{\pi}{2},
\end{align}
which is true in all physically relevant cases,
we can prove an even stronger result.

\begin{theorem}[Gevrey smoothing III]\label{thm:gevrey-main3}
	Let $d\ge 2$. Assume that the cross-section $b$ satisfies the conditions of Theorem \ref{thm:gevrey-main1} and the condition \eqref{eq:cross-section-bdd},  that is, it is bounded away from the singularity.  Let $f$ be a weak solution of the Cauchy problem \eqref{eq:cauchyproblem} with initial datum satisfying conditions \eqref{eq:initialdata}. Then, for all $0<\alpha\leq \min\left\{\alpha_{2,1}, \nu\right\}$,
	\begin{align}
		f(t,\cdot)\in G^{\tfrac{1}{2\alpha}}(\R^d)
	\end{align}
	for all $t>0$, where $\alpha_{2,1} = \frac{\log(9/5)}{\log 2} \simeq 0.847997$.
\end{theorem}

\begin{remark}\label{rem:general-m}
\begin{enumerate}[label=(\roman*)]
	\item Since we do not rely on interpolation inequalities between Sobolev spaces, our results also include the \emph{limiting case} $\alpha=\nu$, at least if $\nu\leq \alpha_{2,n}$ ($n=d, 2, 1$). This is in contrast to all previous results on smoothing properties of the Boltzmann and Kac equations.
	\item If higher moments of the initial datum are bounded (and thus stay bounded eternally due to moment propagation results, see, for instance, \textsc{Villani}'s review \cite{Vil02}), the results in Theorem \ref{thm:gevrey-main2} and Theorem \ref{thm:gevrey-main3} can be improved in the high singularity case, where $\nu$ is close to one. Namely, let $f_0\in L\log L \cap L^1_m (\R^d)$ for some integer $m >2$, then the constants $\alpha_{2,d}$, $\alpha_{2,2}$, respectively $\alpha_{2,1}$ are replaced by $\alpha_{m,n} = \frac{\log[(4m+n)/(2m+n)]}{\log 2}$ ($n=d, 2, 1$), which are strictly increasing towards the limit $\alpha_{\infty,n} = 1$ as $m$ becomes large. See Theorems \ref{thm:gevrey-main1-m}, \ref{thm:gevrey-main2-m} and \ref{thm:gevrey-main3-m} below.
\end{enumerate}	
\end{remark}

Moreover, we prove that for very strong singularities $\nu$, we can prescribe precise conditions on the initial datum such that we have $f\in G^{\frac{1}{2\nu}}(\R^d)$.
\begin{theorem}
	Given $0<\nu<1$, there is $m(\nu)$ such that, if $m\in \N$ and $m\ge m(\nu)$ and $f_0\in L\log L \cap L^1_m$, the weak solution is in $G^{\frac{1}{2\nu}}(\R^d)$ for all $t>0$.
	
	More precisely, under the conditions of Theorem \ref{thm:gevrey-main1}  having  
	$m \ge  \max\left ( 2, \frac{2^\nu -1}{2-2^\nu} \right)$ yields Gevrey smoothing of order $\frac{1}{2\nu}$ and under the slightly stronger conditions of Theorem \ref{thm:gevrey-main3} having   
	  $m \ge  \max\left ( 2, \frac{2^\nu -1}{2(2-2^\nu)} \right)$ is enough.  
\end{theorem}

\begin{remark} 
The proof of this Theorem follows directly from the results of Theorems \ref{thm:gevrey-main1-m},  \ref{thm:gevrey-main2-m}, and \ref{thm:gevrey-main3-m}  in Section \ref{sec:great-results}, which extend Theorems \ref{thm:gevrey-main1}, \ref{thm:gevrey-main2}, and \ref{thm:gevrey-main3} to the case of finite moments $m\ge2$.
\end{remark}

The strategy of the proofs of our main results Theorems \ref{thm:gevrey-main1}, \ref{thm:gevrey-main2}, and \ref{thm:gevrey-main3} is as follows:
We start with the additional assumption $f_0\in L^2$ on the initial datum. We use the
known $H^\infty$ smoothing of the non-cutoff Boltzmann and Kac equation to allow this.
This yields an $L^2$ reformulation of the weak formulation of the Boltzmann and Kac equations which includes suitable growing Fourier multipliers.

The inclusion of sub-Gaussian Fourier multipliers leads to a nonlocal and nonlinear commutator of the Boltzmann and Kac kernels, which turns out to be a three-linear expression in the weighted solution $\hat{f}$ on the Fourier side. In order to bound this expression with $L^2$ norms, one of the three terms has to be controlled pointwise, \emph{including} a sub-Gaussian growing factor, see Proposition \ref{prop:ce}. The problem is that one has to control the pointwise bound with an $L^2$ norm, which is in general impossible. To overcome this obstacle there are several important technical steps:
\begin{enumerate}[label=(\arabic*)]
	\item When working on a ball of radius $\Lambda$, we need this uniform control only on a a ball of radius $\Lambda/\sqrt{2}$, which enables an inductive procedure.
	\item Using the additional a priori information that the kinetic energy is finite, or, depending on the initial condition, even higher moments are finite, we transform weighted $L^2$ bounds into pointwise bounds on slightly smaller balls with an additional loss of power in the weights in Fourier space. Here we rely on
	    Kolmogorov-Landau type inequalities, see Lemma \ref{lem:landau} and appendix \ref{sec:app-landau}.
	\item Use of strict concavity of the Fourier multipliers, see Lemma \ref{lem:abalpha}, in order to compensate for this loss of power.
	\item Averaging over a codimension 2 sphere, in the proof of Theorem \ref{thm:gevrey-main2}, which allows us to get, in any dimension, the same results as for the two dimensional Boltzmann equation.
	\item Averaging over a codimension 1  set constructed from a codimension 2 sphere and the collision angles $\theta$ away from the singularity, and using the fact that near the singularity, one of the three Fourier weights is not big due to Lemma \ref{lem:abalpha}, enables us to get,  in any dimension, the same results as for the one-dimensional Kac equation under the conditions of Theorems \ref{thm:gevrey-main3} and \ref{thm:gevrey-main3-m}.
\end{enumerate}

\section{Gevrey regularity and (ultra-)analyticity of weak solutions with $L^2$ initial data}\label{S2}
In this section, we will prove the Gevrey smoothing of weak solutions with
initial datum $f_0$ satisfying \eqref{eq:initialdata} and, \emph{additionally},  $f_0\in L^2(\R^d)$.

\subsection{$L^2$-Reformulation of the homogeneous Boltzmann equation for weak solutions and coercivity}\label{sec:reform}
The following is our starting point for the proof of the regularizing properties of the homogenous Boltzmann equation.
\begin{proposition}[] \label{prop:L2reform}
	Let $f$ be a weak solution of the Cauchy problem \eqref{eq:cauchyproblem} with initial datum $f_0$ satisfying \eqref{eq:initialdata}, and let $T_0>0$. Then for all $t\in(0, T_0]$, $\beta>0$, $\alpha\in(0,1)$, and $\Lambda>0$ we have
$G_{\Lambda}f \in \mathcal{C}\left([0, T_0]; L^2(\R^d)\right)$ and
	\begin{align}\label{eq:reformulation}
	\begin{split}
	&\frac{1}{2} \|G_{\Lambda}(t,D_v)f(t,\cdot)\|_{L^2}^2 - \frac{1}{2} \int_0^t \left\langle f(\tau, \cdot), \left( \partial_\tau G_{\Lambda}^2(\tau, D_v) \right) f(\tau,\cdot)\right\rangle \,\mathrm{d}\tau \\
	&= \frac{1}{2} \|\1_{\Lambda}(D_v)f_0\|_{L^2}^2 + \int_0^t \left\langle Q(f,f)(\tau, \cdot), G_{\Lambda}^2(\tau, D_v)f(\tau, \cdot)\right\rangle \, \mathrm{d}\tau.
	\end{split}
	\end{align}
\end{proposition}
Informally, equation \eqref{eq:reformulation} follows from using $\varphi(t,\cdot):= G_{\Lambda}^2(t,D_v)f(t,\cdot)$  in the weak formulation of the homogenous Boltzmann equation.
Recall that $G_{\Lambda}^2f \in L^{\infty}([0, T_0]; H^{\infty}(\R^d))$ for any finite $T_0>0$, so it misses the required regularity in time needed to be used as a test function. The proof of Proposition \ref{prop:L2reform} is analogous to \textsc{Morimoto} \textit{et al.} \cite{MUXY09}, for the sake of completeness and the convenience of the reader, we prove it in appendix \ref{sec:appendix-reformulation}.

The coercive properties of the non-cutoff Boltzmann bilinear operator which play the crucial role in the smoothing of solutions are made precise in the following sub-elliptic estimate by \textsc{Alexandre, Desvillettes, Villani} and \textsc{Wennberg} \cite{ADVW00}. We remark that, while the proof there is given for the Boltzmann equation, it equally applies to the Kac equation.

\begin{lemma}[Sub-elliptic Estimate, \cite{ADVW00}]\label{lem:subelliptic}
	Let $g\in L^1_2(\R^d) \cap L\log L(\R^d)$, $g\geq 0$ ($g\not\equiv 0$). Assume that the collision cross-section $b$ satisfies \eqref{eq:cross-section}-\eqref{eq:cross-section2} or \eqref{eq:cross-section-kac}-\eqref{eq:cross-section-kac2} respectively, with $0<\nu<1$. Then there exists a constant $C_g>0$ (depending only on the dimension $d$, the collision kernel $b$, $\|g\|_{L^1_2}$ and $\|g\|_{L\log L}$) and a constant $C>0$ (depending only on $d$ and $b$), such that for any $f\in H^1(\R^d)$ one has
	\begin{align*}
		-\langle Q(g,f), f\rangle \geq C_g \|f\|_{H^{\nu}}^2 - C \|g\|_{L^1_2} \|f\|_{L^2}^2.
	\end{align*}
\end{lemma}

\begin{remark} As explained for instance in \cite{AMUXY10}, the constant $C_g$ is an increasing function of $\|g\|_{L^1}$, $\|g\|_{L^1_2}^{-1}$ and $\|g\|_{L\log L}^{-1}$. In particular, if $g$ is a weak solution of the Cauchy problem \eqref{eq:cauchyproblem} with initial datum $g_0\in L^1_2(\R^d)\cap L\log L(\R^d)$, we have $\|g\|_{L^1}=\|g_0\|_{L^1}$, $\|g\|_{L^1_2} \leq \|g_0\|_{L^1_2}$ and $\|g\|_{L\log L} \leq \log2 \|g_0\|_{L^1} + H(g_0) + C_{\delta,d} \|g_0\|_{L^1_2}^{1-\delta}$, for small enough $\delta>0$ (see \eqref{eq:LlogL}). This implies $C_g \geq C_{g_0}$ and thus
			\begin{align*}
				-\langle Q(g,f), f\rangle \geq C_g \|f\|_{H^{\nu}}^2 - C \|g\|_{L^1_2} \|f\|_{L^2}^2 \geq C_{g_0} \|f\|_{H^{\nu}}^2 - C \|g_0\|_{L^1_2} \|f\|_{L^2}^2.
			\end{align*}
		uniformly in $t\ge 0$.
\end{remark}

Together with Proposition \ref{prop:L2reform} the coercivity estimate Lemma \ref{lem:subelliptic} implies

\begin{corollary}[A priori bound for weak solutions]  \label{cor:gronwallbound}
	Let $f$ be a weak solution of the Cauchy problem \eqref{eq:cauchyproblem} with initial datum $f_0$ satisfying \eqref{eq:initialdata}, and let $T_0>0$. Then there exist constants $\widetilde{C}_{f_0}, C_{f_0} >0$ (depending only on the dimension $d$, the collision kernel $b$, $\|f_0\|_{L^1_2}$ and $\|f_0\|_{L\log L}$) such that for all $t\in(0, T_0]$, $\beta>0$, $\alpha\in(0,1)$, and $\Lambda>0$ we have
	\begin{align}
	\begin{split}
		\|G_{\Lambda}f\|_{L^2}^2
		\leq
		  \|\1_{\Lambda}(D_v)f_0\|_{L^2}^2
		  &+ \int_0^t 2\left( - \widetilde{C}_{f_0} \|G_{\Lambda}f\|_{H^{\nu}}^2 + C_{f_0} \|G_{\Lambda}f\|_{L^2}^2 \right) \, \mathrm{d}\tau \\
		&+ \int_0^t 2\left| \left\langle Q(f, G_{\Lambda}f) - G_{\Lambda} Q(f,f), G_{\Lambda}f\right\rangle\right| \, \mathrm{d}\tau \\
		&+ \int_0^t 2\beta \|G_{\Lambda}f\|_{H^{\alpha}}^2 \,\mathrm{d}\tau.
		\end{split}
	\end{align}
\end{corollary}

\begin{proof}
We want to apply the coercivity result from Lemma \ref{lem:subelliptic} to the second integral on the right hand side of Proposition \ref{prop:L2reform}. Therefore, we write
\begin{align*}
	\langle Q(f,f), G_{\Lambda}^2f \rangle &= \langle G_{\Lambda} Q(f,f), G_{\Lambda}f\rangle = \langle Q(f, G_{\Lambda}f), G_{\Lambda}f\rangle + \langle G_{\Lambda}Q(f,f) - Q(f,G_{\Lambda}f), G_{\Lambda}f\rangle \\
		&\leq - \widetilde{C}_{f_0} \|G_{\Lambda}f\|_{H^{\nu}}^2 + \underbrace{C \|f_0\|_{L^1_2}}_{=: C_{f_0}} \|G_{\Lambda}f\|_{L^2}^2 + \langle G_{\Lambda}Q(f,f) - Q(f,G_{\Lambda}f), G_{\Lambda}f\rangle .
\end{align*}
Moreover,
\begin{align*}
	\partial_\tau G_{\Lambda}^2(\tau, \eta)= 2\beta \langle \eta\rangle^{2\alpha} G_\Lambda(t,\eta) .
\end{align*}
Inserting those two results into \eqref{eq:reformulation}, we obtain
\begin{align*}
	\|G_{\Lambda} f\|_{L^2}^2 \leq \|\1_{\Lambda}(D_v)f_0\|_{L^2}^2 &+ 2\beta \int_0^t \|G_{\Lambda}f(\tau, \cdot)\|_{H^{\alpha}}^2 \, \mathrm{d}\tau + 2 \int_0^t \left( -\widetilde{C}_{f_0} \|G_{\Lambda}f\|_{H^{\nu}}^2 + C_{f_0} \|G_{\Lambda}f\|_{L^2}^2 \right) \, \mathrm{d}\tau \\ &+ 2\int_0^t \langle G_{\Lambda}Q(f,f) - Q(f,G_{\Lambda}f), G_{\Lambda}f\rangle \,\mathrm{d}\tau. \qedhere
\end{align*}
\end{proof}

\begin{remark}
It is natural to call the term 	$\langle G_{\Lambda}Q(f,f) - Q(f,G_{\Lambda}f), G_{\Lambda}f\rangle$ the \emph{commutation error}.
\end{remark}

\subsection{Bound on the commutation error}\label{ssec:ce}
Next, we prove a new bound on the commutation error. An important ingredient is the following elementary observation:
\begin{lemma}[Strict concavity bound]\label{lem:abalpha}
	Let $\alpha \in(0,1]$ be fixed. The map $0\leq u\mapsto\epsilon(\alpha, u):=(1+u)^{\alpha} - u^{\alpha}$ has the following properties:
	\begin{enumerate}[label=(\roman*)]
		\item If $\alpha \in (0,1)$, then $\epsilon(\alpha,\cdot)$ is \emph{strictly} decreasing on $[0, \infty)$ with $\lim_{u\to\infty} \epsilon(\alpha, u) = 0$. \\
		In particular, for any $\gamma\geq 1$ and $0\leq \gamma s^- \leq s^+$ one has
		\begin{align}\label{eq:epsilon-inequ}
			 \epsilon\left(\alpha,\tfrac{s^+}{s^-}\right)\le \epsilon\left(\alpha,\gamma \right)\le \epsilon(\alpha,1)= 2^\alpha-1<1.
		\end{align}
		Moreover, for all $\alpha\in (0,1)$ and all $u>0$
			\begin{align*}
				\epsilon\left(\alpha, u \right) \le u^{\alpha-1} .
			\end{align*}
		\item If  $u>0$, then $\epsilon(\cdot,u)$ is \emph{strictly} increasing on $[0, 1]$.
		\item For all  $s^-, s^+ \ge 0$
			\begin{align*}
			(1+s^-+s^+)^\alpha \le  \epsilon\left(\alpha,\tfrac{s^+}{s^-}\right)(1+s^-)^\alpha + (1+s^+)^\alpha.
			\end{align*}
	\end{enumerate}
\end{lemma}

\begin{proof}
Since
\begin{align*}
	\frac{\partial}{\partial u} \epsilon(\alpha, u) = \alpha \left( (1+u)^{\alpha-1} - u^{\alpha-1}\right) <0 \quad \text{for } \alpha\in(0,1)
\end{align*}
$\epsilon(\alpha, \cdot)$ is strictly decreasing. Furthermore, for fixed $u>0$ we have
\begin{align*}
	\frac{\partial}{\partial\alpha} \epsilon(\alpha,u) = \log(1+u)\, (1+u)^{\alpha} - \log u \, u^{\alpha} >0,
\end{align*}
which shows that $\epsilon(\cdot, u)$ is strictly increasing.

For $\alpha\in(0,1)$ and $u\geq 0$ we estimate
\begin{align*}
	\epsilon(u,\alpha) = \alpha \int_u^{1+u} r^{\alpha-1}\,\mathrm{d}r \leq \alpha u^{\alpha-1} \leq u^{\alpha-1}.
\end{align*}
In particular, $\lim_{u\to\infty} \epsilon(\alpha,u) = 0$. By monotonicity, the chain of inequalities \eqref{eq:epsilon-inequ} follows.

Let $s^-,s^+\ge 0$. Then
\begin{align*}
	(1+s^-+s^+)^\alpha &= (s^-)^{\alpha} \left[\left(1+\tfrac{1+s^+}{s^-}\right)^{\alpha} - \left(\tfrac{1+s^+}{s^-}\right)^{\alpha}\right] + (1+s^+)^{\alpha} \\
	&\leq \epsilon\left(\alpha,\tfrac{1+s^+}{s^-}\right)(1+s^-)^\alpha + (1+s^+)^\alpha \leq \epsilon\left(\alpha,\tfrac{s^+}{s^-}\right)(1+s^-)^\alpha + (1+s^+)^\alpha
\end{align*}
where we made use of the monotonicity of $\epsilon(\alpha, \cdot)$ in the last inequality.
\end{proof}

\begin{remark}
 The proof of Lemma \ref{lem:abalpha} is so simple that one might wonder whether it could be of any use. In fact, it is crucial. It's usefulness is hidden in the fact that it enables us to gain a small exponent in the commutator estimates, see Proposition \ref{prop:ce} and Lemma \ref{lem:ce} below. Furthermore, $\epsilon(\alpha, \gamma)$ can be made as small as we like if $\gamma$ can be chosen large enough, which will be important in the proof of Theorem \ref{thm:gevrey-main3}.
\end{remark}

\begin{corollary}\label{cor:expdiff}
Let $\widetilde{G}(s):= e^{\beta t (1+s)^{\alpha}}$ for $s\geq 0$, $\alpha\in(0,1]$. Then, for all $s^- + s^+ = s$ with $ 0\leq s^- \leq s^+$,
\begin{align*}
	|\widetilde{G}(s) - \widetilde{G}(s^+)| \leq 2 \alpha\beta t (1+s^+)^{\alpha} \big(1-\tfrac{s^+}{s}\big) \,\widetilde{G}(s^-)^{\epsilon\big(\alpha,\tfrac{s^+}{s^-}\big)} \widetilde{G}(s^+)
\end{align*}
with $\epsilon(\alpha,u)$ from Lemma \ref{lem:abalpha}.
\end{corollary}

\begin{proof}
Since $s^+ \leq s$ and $\alpha \in (0,1]$,
\begin{align*}
	|\widetilde{G}(s) - \widetilde{G}(s^+)| \leq \int_{s^+}^s \left| \frac{\mathrm{d}}{\mathrm{d}r}\widetilde{G}(r) \right| \mathrm{d}r =\alpha \beta t \int_{s^+}^s  (1+r)^{\alpha-1} \widetilde{G}(r) \, \mathrm{d}r \leq \alpha \beta t  (1+s^+)^{\alpha-1} (s-s^+) \widetilde{G}(s).
\end{align*}	
In addition, since $s\leq 2s^+$,
\begin{align*}
\frac{s-s^+}{1+s^+} = \left(1-\frac{s^+}{s}\right) \frac{s}{1+s^+} \leq 2 \left(1-\frac{s^+}{s}\right).
\end{align*}
Moreover, since $s=s^+ + s^-$, the strict concavity Lemma \ref{lem:abalpha} gives
\begin{align*}
	\widetilde{G}(s) \leq \widetilde{G}(s^-)^{\epsilon\big(\alpha,\tfrac{s^+}{s^-}\big)} \widetilde{G}(s^+),
\end{align*}
which completes the proof.
\end{proof}

\begin{proposition}[Bound on Commutation Error]\label{prop:ce}
	Let $f$ be a weak solution of the Cauchy problem \eqref{eq:cauchyproblem} with initial datum $f_0$ satisfying \eqref{eq:initialdata}. Recall $\epsilon(\alpha,u) = (1+u)^{\alpha}-u^\alpha$. Then for all $t\in(0, T_0]$, $\beta>0$, $\alpha\in(0,1)$, and $\Lambda>0$ we have
	\begin{align}\label{eq:ce}
	\begin{split}
		&\left| \left\langle Q(f, G_{\Lambda}f) - G_{\Lambda} Q(f,f), G_{\Lambda}f\right\rangle\right| \\
		&\qquad \leq 2 \alpha \beta t \int_{\R^d} \int_{\S^{d-1}} b\left(\frac{\eta}{|\eta|}\cdot \sigma \right) \left(1- \frac{|\eta^+|^2}{|\eta|^2}\right) G(\eta^-)^{\epsilon(\alpha, |\eta^+|^2/ |\eta^-|^2)} |\hat{f}(\eta^-)| \\
		&\hskip5cm \times G_{\Lambda}(\eta^+) |\hat{f}(\eta^+)|\, G_{\Lambda}(\eta)|\hat{f}(\eta)| \, \langle \eta^+ \rangle^{2\alpha} \,  \mathrm{d}\sigma \mathrm{d}\eta,
		\end{split}
	\end{align}
for $d\geq 2$, and
	\begin{align}\label{eq:ce-kac}
	\begin{split}
		&\left| \left\langle Q(f, G_{\Lambda}f) - G_{\Lambda} Q(f,f), G_{\Lambda}f\right\rangle\right| \\
		&\qquad \leq 2 \alpha \beta t \int_{\R} \int_{-\frac{\pi}{4}}^{\frac{\pi}{4}} b_1\left(\theta\right) \sin^2\theta \, G(\eta^-)^{\epsilon(\alpha,|\eta^+|^2/ | \eta^- |^2)} |\hat{f}(\eta^-)| \\
		&\hskip5cm \times G_{\Lambda}(\eta^+) |\hat{f}(\eta^+)|\, G_{\Lambda}(\eta)|\hat{f}(\eta)| \, \langle \eta^+ \rangle^{2\alpha} \,  \mathrm{d}\theta \mathrm{d}\eta,
		\end{split}
	\end{align}
in the one-dimensional case.
\end{proposition}

\begin{remark} If the weight $G$ was growing \emph{polynomially}, the term $G(\eta^-)$ in the integral \eqref{eq:ce}, respectively \eqref{eq:ce-kac}, would be replaced by $1$. In this case, the ``bad terms" which contain $\eta^-$ can simply be bounded by
 $\|\hat{f}\|_{L^\infty}\le \|f\|_{L^1} = \|f_0\|_{L^1}$ and the rest can be bounded nicely in terms of $\|G_\Lambda\hat{f}\|_{L^2}$ and $\|G_\Lambda\hat{f}\|_{H^\alpha}$, see the discussion in appendix \ref{sec:appendix-hinfty}.

 If the weight  $G$ is exponential, the estimate of the terms containing $\eta^-$ in \eqref{eq:ce}, respectively \eqref{eq:ce-kac}, is an additional challenge and the methods we devised in order to control this term in the commutation error is probably the most important new contribution of this work.
 \end{remark}

\begin{proof}[Proof of Proposition \ref{prop:ce}]
	We start with $d\geq 2$. By Bobylev's identity, one has
	\begin{align*}
		&\left| \left\langle Q(f, G_{\Lambda}f) - G_{\Lambda} Q(f,f), G_{\Lambda}f\right\rangle\right| = \left| \left\langle \mathcal{F}\left[Q(f, G_{\Lambda}f) -G_{\Lambda} Q(f,f)\right], \mathcal{F}\left[G_{\Lambda}f\right]\right\rangle_{L^2}\right| \\
		&\quad \leq \int_{\R^d} \int_{\S^{d-1}} b\left(\frac{\eta}{|\eta|}\cdot\sigma\right) G_{\Lambda}(\eta) |\hat{f}(\eta)| \, |\hat{f}(\eta^-)| \, |\hat{f}(\eta^+)| |G_{\Lambda}(\eta^+) - G_{\Lambda}(\eta)| \,\mathrm{d}\sigma \,\mathrm{d}\eta\\
		&\quad = \int_{\R^d} \int_{\S^{d-1}} b\left(\frac{\eta}{|\eta|}\cdot\sigma\right) G_{\Lambda}(\eta) |\hat{f}(\eta)| \, |\hat{f}(\eta^-)| \, |\hat{f}(\eta^+)| |G(\eta^+) - G(\eta)| \,\mathrm{d}\sigma \,\mathrm{d}\eta,
	\end{align*}
where the latter equality follows from the fact that $G_{\Lambda}$ is supported on the ball $\{|\eta|\leq \Lambda\}$ and $|\eta^+| \leq |\eta|$.

To estimate $|G(\eta^+) - G(\eta)|$, we use Corollary \ref{cor:expdiff} with $s:=|\eta|^2$ and, accordingly, $s^{\pm} = |\eta^{\pm}|^2$. Notice that
\begin{align*}
	|\eta^{\pm}|^2 = \frac{|\eta|^2}{2} \left( 1 \pm \frac{\eta}{|\eta|}\cdot \sigma \right), \quad |\eta|^2 = |\eta^+|^2 + |\eta^-|^2,
\end{align*}
and, writing $\cos\theta = \frac{\eta\cdot \sigma}{|\eta|}$, we also have
\begin{align*}
	|\eta^+|^2 = |\eta|^2 \cos^2\tfrac{\theta}{2}, \quad 	|\eta^-|^2 = |\eta|^2 \sin^2\tfrac{\theta}{2}.
\end{align*}
Since $b$ is supported on angles in $[0, \pi/2]$, one sees $0\leq |\eta^-|^2 \leq \frac{1}{2} |\eta|^2$ and $\frac{1}{2}|\eta|^2 \leq |\eta^+|^2 \leq |\eta|^2$.
Therefore, $s^-\leq \frac{s}{2} \leq s^+ \leq s$ and $s = s^+ + s^-$.

It follows that for all $\eta\in\R^d$ with $|\eta|\leq\Lambda$, noting that $|\eta^+|\leq |\eta|\leq \Lambda$,
\begin{align}
		|G(\eta) - G(\eta^+)| \leq 2\alpha\beta t \langle \eta^+ \rangle^{2\alpha}\left(1-\tfrac{|\eta^+|^2}{|\eta|^2}\right) \,  G(\eta^-)^{\epsilon(\alpha,|\eta^+|^2/|\eta^-|^2 )} G_{\Lambda}(\eta^+),
\end{align}
which finishes the proof in dimension $d\ge 2$.

For the Kac model we remark that the above proof depends only on $|\eta^-|\le |\eta^+|\leq |\eta|$ and $|\eta^-|^2+ |\eta^+|^2= |\eta|^2$, hence $|\eta^-|^2\le |\eta|^2/2$, and the strict concavity Lemma \ref{lem:abalpha} and the Corollary \ref{cor:expdiff}.
Since, by symmetry, we assume that $b_1$ is supported in $[-\pi/4, \pi/4]$, the same bounds for $\eta^-$ and $\eta^+$ hold in dimension one and the above proof can be literally translated, with obvious changes in notation, to the Kac equation.
\end{proof}

The bound on the commutation error in Proposition \ref{prop:ce} is a trilinear expression in the weak solution $f$. In order to close the a priori bound from Corollary \ref{cor:gronwallbound} in $L^2$, one of the terms has to be controlled \emph{uniformly} in $\eta$. Seemingly impossible with the growing weights, it is exactly at this place where the gain of the small exponent $\epsilon(\alpha, |\eta^+|^2/ |\eta^-|^2)\le \epsilon(\alpha,1)<1$ in the $G(\eta^-)$ term in \eqref{eq:ce} and \eqref{eq:ce-kac} allows us to proceed with this strategy.
This gain of the small exponent is new and enabled by the strict concavity bound of Lemma \ref{lem:abalpha} and its Corollary \ref{cor:expdiff} and it is crucial for our inductive approach for controlling the commutation error.

\begin{lemma}\label{lem:ce} The inequality
	\begin{align*}
		\left| \left\langle Q(f, G_{\Lambda}f) - G_{\Lambda} Q(f,f), G_{\Lambda}f\right\rangle\right|
		\le I_{d,\Lambda} + I_{d,\Lambda}^+
	\end{align*}
	holds, where, for $d\ge 2$
	\begin{align}\label{eq:Id}
	\begin{split}
		I_{d,\Lambda} &= \alpha \beta t \int_{\R^d}\int_{0}^{\tfrac{\pi}{2}}  \int_{\S^{d-2}(\eta)} \sin^{d}\theta \, b(\cos\theta)\, G(\eta^-)^{\epsilon\left(\alpha, \cot^2\tfrac{\theta}{2}\right)}\,|\hat{f}(\eta^-)| \, \1_{\tfrac{\Lambda}{\sqrt{2}}}(|\eta^-|) \,\mathrm{d}\omega\,\mathrm{d}\theta\\
		&\qquad \qquad \qquad \qquad \qquad \qquad \times |G_{\Lambda}(\eta) \hat{f}(\eta)|^2 \,\langle \eta \rangle^{2\alpha}\,\mathrm{d}\eta.
	\end{split}
	\end{align}
	Here the vector $\eta^-$ is expressed as a function of $\eta$ and $\sigma$, that is,
	\begin{align}\label{eq:etapar}
		\eta^-= \eta^-(\eta,\sigma) = \frac{1}{2}(\eta - |\eta|\sigma)
			= |\eta|\sin^2(\tfrac{\theta}{2}) \frac{\eta}{|\eta|} - |\eta| \sin (\tfrac{\theta}{2}) \cos (\tfrac{\theta}{2}) \,\omega
	\end{align}
	and $\sigma$ is is a vector on the unit sphere given by
	\begin{align}\label{eq:sigmaeta}
		\sigma = \sigma(\theta,\omega)=\cos(\theta) \frac{\eta}{|\eta|} + \sin (\theta)\, \omega
	\end{align}
	with polar angle $\theta \in [0,\pi/2]$ with respect to the north pole in the $\eta$ direction, $\omega\in \S^{d-2}(\eta):= \{ \widetilde{\omega}\in \R^d: \widetilde{\omega} \perp \eta, |\widetilde{\omega}|=1  \}$, the $d-2$
	sphere in $\R^d$ orthogonal to the $\eta$ direction, and $\mathrm{d}\omega$ the canonical measure on $\mathbb{S}^{d-2}$. 
	\begin{align}\label{eq:Idplus}
	\begin{split}
			I_{d, \Lambda}^+ &= 2^{d} \alpha\beta t\int_{\R^d} \int_0^{\frac{\pi}{4}} \int_{\S^{d-2}(\eta^+)} \sin^d \vartheta \, b\left( \cos 2\vartheta\right)  \,  \, G(\eta^-)^{\epsilon\left(\alpha, \cot^2\vartheta\right)} |\hat{f}(\eta^-)|\, \1_{\tfrac{\Lambda}{\sqrt{2}}}(|\eta^-|)\\  & \qquad \qquad \qquad \qquad \qquad  \times |G_{\Lambda}(\eta^+) \hat{f}(\eta^+)|^2 \langle\eta^+\rangle^{2\alpha} \, \mathrm{d}\vartheta\,\mathrm{d}\omega \,\mathrm{d}\eta^+
	\end{split}
	\end{align}
	where now the vector $\eta^-$ is expressed as a function of $\eta^+$ and $\sigma$, that is,
	\begin{align}\label{eq:etapluspar}
		\eta^-= \eta^-(\eta^+,\sigma) = \eta^+ - |\eta^+|\left(\frac{\eta^+\cdot\sigma}{|\eta^+|}\right)^{-1}\sigma = -|\eta^+| \tan (\vartheta) \, \omega
	\end{align}
	where now $\sigma$ is is a vector on the unit sphere with north pole in the $\eta^+$ direction given by
	\begin{align}\label{eq:sigmaetaplus}
		\sigma =\sigma(\vartheta,\omega)= \cos(\vartheta) \frac{\eta^+}{|\eta^+|} + \sin (\vartheta)\, \omega
	\end{align}
	with polar angle $\vartheta \in [0,\pi/4]$  and $\omega\in \S^{d-2}(\eta^+)$, the $(d-2)$-sphere in $\R^d$ orthogonal to the $\eta^+$ direction.
	If $d=2$ we set $\S^0:=\emptyset$ in this context.
	
	For $d=1$ we have
	\begin{align*}
		I_{1, \Lambda} &=  \alpha \beta t \int_{\R}\int_{-\tfrac{\pi}{4}}^{\tfrac{\pi}{4}} \sin^{2}\theta b_1(\theta)\, G(\eta^-)^{\epsilon\left(\alpha, \cot^2\tfrac{\theta}{2}\right)}\,|\hat{f}(\eta^-)| \, \1_{\tfrac{\Lambda}{\sqrt{2}}}(|\eta^-|) \,\mathrm{d}\theta\\
		&\qquad \qquad \qquad \qquad \qquad \qquad \times |G_{\Lambda}(\eta) \hat{f}(\eta)|^2 \,\langle \eta \rangle^{2\alpha}\,\mathrm{d}\eta, \\
		I_{1, \Lambda}^+ &= \sqrt{2} \alpha \beta t \int_{\R}\int_{-\tfrac{\pi}{4}}^{\tfrac{\pi}{4}} \sin^{2}\theta b_1(\theta)\, G(\eta^-)^{\epsilon\left(\alpha, \cot^2\tfrac{\theta}{2}\right)} \,|\hat{f}(\eta^-)| \, \1_{\tfrac{\Lambda}{\sqrt{2}}}(|\eta^-|) \,\mathrm{d}\omega\,\mathrm{d}\theta\\
		&\qquad \qquad \qquad \qquad \qquad \qquad \times |G_{\Lambda}(\eta^+) \hat{f}(\eta^+)|^2 \,\langle \eta^+ \rangle^{2\alpha}\,\mathrm{d}\eta^+,
	\end{align*}
	where in the first case $\eta^-=\eta^-(\eta,\theta)= \eta \sin \theta$ and in the second case $\eta^-= \eta^-(\eta^+,\theta)= \eta^+\tan\theta$ and there is no need to distinguish between the $\theta$ and $\vartheta$ parametrization.
\end{lemma}
\begin{remark}
 In the $\eta$, respectively $\eta^+$, integrals above $\eta^-$  and $\sigma$ are always the same vectors expressed in different parametrizations. We therefore have the relation
  $\vartheta = \theta/2$, see Figure \ref{fig:geometry} for the geometry of the collision process in Fourier space.
\end{remark}

\begin{figure}[ht]\centering
\definecolor{ccqqqq}{rgb}{0.8,0.,0.}
\definecolor{ubqqys}{rgb}{0.29411764705882354,0.,0.5098039215686274}
\definecolor{qqwuqq}{rgb}{0.,0.39215686274509803,0.}
\definecolor{uuuuuu}{rgb}{0.26666666666666666,0.26666666666666666,0.26666666666666666}
\begin{tikzpicture}[line cap=round,line join=round,>=triangle 45,x=1.0cm,y=1.0cm]
\clip(-5,-.5) rectangle (6,5.5);
\draw [shift={(0.,2.5)},color=qqwuqq,fill=qqwuqq,fill opacity=0.1] (0,0) -- (90.:0.6) arc (90.:147.39540683375847:0.6) -- cycle;
\draw[color=ubqqys,fill=ubqqys,fill opacity=0.1] (0.2480998443537024,0.13581777214955604) -- (0.11228207220414639,0.3839176165032584) -- (-0.135817772149556,0.2480998443537024) -- (0.,0.) -- cycle;
\draw [shift={(0.,0.)},color=qqwuqq,fill=qqwuqq,fill opacity=0.1] (0,0) -- (90.:1.) arc (90.:118.69770341687922:1.) -- cycle;
\draw [shift={(2.1060230081732154,1.1529042019792943)},color=qqwuqq,fill=qqwuqq,fill opacity=0.1] (0,0) -- (180.:0.9) arc (180.:208.69770341687922:0.9) -- cycle;
\draw [shift={(-2.106023008173215,3.8470957980207054)},color=qqwuqq,fill=qqwuqq,fill opacity=0.1] (0,0) -- (-61.30229658312078:0.9) arc (-61.30229658312078:-32.60459316624155:0.9) -- cycle;
\draw [->] (0.,0.) -- (0.,5.);
\draw(0.,2.5) circle (2.5cm);
\draw [->] (0.,2.5) -- (-2.106023008173215,3.8470957980207054);
\draw [->] (0.,2.5) -- (2.1060230081732154,1.1529042019792943);
\draw [->] (0.,0.) -- (2.1060230081732154,1.1529042019792943);
\draw [->] (0.,0.) -- (-2.106023008173215,3.8470957980207054);
\draw [dash pattern=on 5pt off 5pt] (0.,0.) circle (5.cm);
\draw [dash pattern=on 5pt off 5pt] (0.,0.) circle (2.400941692315012cm);
\draw (0,5.) node[anchor=north west] {$\eta$};
\draw (-2.3,3.38) node[anchor=north west] {$\eta^+$};
\draw (0.6,1.1) node[anchor=north west] {$\eta^-$};
\draw [dash pattern=on 5pt off 5pt] (0.,1.1529042019792943)-- (2.1060230081732154,1.1529042019792943);
\draw [->,color=ccqqqq] (0.,2.5) -- (-0.9394274853127519,3.1008950581730215);
\begin{scriptsize}
\draw [fill=uuuuuu] (0.,0.) circle (1.5pt);
\draw [fill=uuuuuu] (0.,2.5) circle (1.5pt);
\draw [fill=uuuuuu] (2.1060230081732154,1.1529042019792943) circle (1.5pt);
\draw[color=qqwuqq] (-0.13,2.88) node {$\theta$};
\draw[color=qqwuqq] (-0.2,0.8) node {$\theta$/2};
\draw[color=qqwuqq] (1.5,1.03) node {$\theta$/2};
\draw[color=ccqqqq] (-0.5,2.55) node {$\sigma$};
\draw[color=qqwuqq] (-1.65,3.4) node {$\vartheta$};
\end{scriptsize}
\end{tikzpicture}
\caption{\label{fig:geometry}Geometry of the collision process in Fourier space.}
\end{figure}
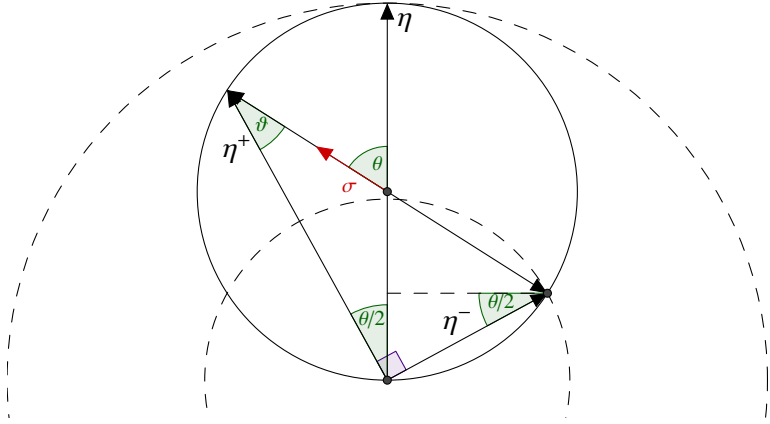

\begin{remark}
From the bounds given in Lemma \ref{lem:ce} one might already see that, in order to bound the commutation error by some multiple of $\|G_{\Lambda}f\|_{H^{\alpha}(\R^d)}^2$, one has to control integrals of the form
\begin{align*}
  \sup_{|\eta|\le \Lambda}\int_{0}^{\tfrac{\pi}{2}}  \int_{\S^{d-2}(\eta)} \sin^{d}\theta b(\cos\theta)\, G^{\epsilon\left(\alpha, \cot^2\tfrac{\theta}{2}\right)}(\eta^-)\,|\hat{f}(\eta^-)| \, \1_{\tfrac{\Lambda}{\sqrt{2}}}(|\eta^-|) \,\mathrm{d}\omega\,\mathrm{d}\theta,
\end{align*}
with the parametrisation \eqref{eq:etapar} for $\eta^-$, and similarly for \eqref{eq:Idplus} and the corresponding integrals in the one dimensional case. Due to characteristic function in $\eta^-$, this uniform control is not needed on the full ball of radius $\Lambda$, but only on a strictly smaller one, giving rise to an \emph{induction-over-length-scales} type of argument.
\end{remark}

\begin{proof}[Proof of Lemma \ref{lem:ce}] Let $d\ge 2$.
Using the elementary estimate
\begin{align*}
	|G_{\Lambda}(\eta) \hat{f}(\eta)|\, |G_{\Lambda}(\eta^+)\hat{f}(\eta^+)| \leq \frac{1}{2} \left(|G_{\Lambda}(\eta) \hat{f}(\eta)|^2 + |G_{\Lambda}(\eta^+) \hat{f}(\eta^+)|^2 \right)
\end{align*}
in the bound \eqref{eq:ce} gives
\begin{align*}
			\left| \left\langle Q(f, G_{\Lambda}f) - G_{\Lambda} Q(f,f), G_{\Lambda}f\right\rangle\right|
		\le \widetilde{I}_{d,\Lambda} + \widetilde{I}_{d,\Lambda}^+
\end{align*}
with
\begin{align*}
	\widetilde{I}_{d,\Lambda}
		&=  \alpha \beta t \int_{\R^d} \int_{\S^{d-1}} b\left(\frac{\eta}{|\eta|}\cdot \sigma \right) \left(1- \frac{|\eta^+|^2}{|\eta|^2}\right) G(\eta^-)^{\epsilon(\alpha, |\eta^+|^2/ |\eta^-|^2)} |\hat{f}(\eta^-)|
			\1_{\frac{\Lambda}{\sqrt{2}}}(|\eta^-|)\\
		&\hskip6cm \times  |G_{\Lambda}(\eta)\hat{f}(\eta)|^2 \, \langle \eta^+ \rangle^{2\alpha} \,  \mathrm{d}\sigma \mathrm{d}\eta,
		\intertext{and}
	\widetilde{I}_{d,\Lambda}^+
		&=  \alpha \beta t \int_{\R^d} \int_{\S^{d-1}} b\left(\frac{\eta}{|\eta|}\cdot \sigma \right) \left(1- \frac{|\eta^+|^2}{|\eta|^2}\right) G(\eta^-)^{\epsilon(\alpha, |\eta^+|^2/ |\eta^-|^2)} |\hat{f}(\eta^-)|
			\1_{\frac{\Lambda}{\sqrt{2}}}(|\eta^-|)\\
		&\hskip6cm \times |G_{\Lambda}(\eta^+) \hat{f}(\eta^+)|^2 \, \langle \eta^+ \rangle^{2\alpha} \,  \mathrm{d}\sigma \mathrm{d}\eta
\end{align*}
 First we consider $\widetilde{I}_{d,\Lambda}$: Writing $\sigma$ in a parametrization where the north pole is in the $\eta$ direction, one has
 \begin{align*}
 	\sigma = \cos\theta \frac{\eta}{|\eta|} + \sin \theta\, \omega
 \end{align*}
 where $\cos\theta = \tfrac{\eta\cdot\sigma}{|\eta|} \ge 0$ and $\omega$ is a unit vector orthogonal to $\eta$, that is, $\omega \in \S^{d-2}(\eta)$. Due to the support condition on $b$ one has $\cos\theta \ge 0$, that is, $\sigma$ is restricted to the northern hemisphere $\theta\in[0,\pi/2]$.
 In this parametization one has $\mathrm{d}\sigma = \sin^{d-2}\theta \, \mathrm{d}\theta\mathrm{d}\omega$.
 From the definition of $\eta^\pm$ one sees
 \begin{align*}
 	\eta^\pm
 		&= \frac{1}{2} (\eta\pm |\eta|\sigma)
 		 = \frac{|\eta|}{2}(1\pm \cos\theta) \frac{\eta}{|\eta|} \pm \frac{|\eta|}{2} \sin(\theta)\, \omega
 \end{align*}
  so
  \begin{align*}
  	\eta^+ = |\eta| \cos^2(\tfrac{\theta}{2}) \frac{\eta}{|\eta|} + |\eta| \sin(\tfrac{\theta}{2})\cos(\tfrac{\theta}{2})\, \omega .
  \end{align*}
  In particular,
  \begin{align*}
  	|\eta^+| = |\eta| \cos\frac{\theta}{2}, \quad \text{and} \quad 1-\frac{|\eta^+|^2}{|\eta|^2} = 1-\cos^2\frac{\theta}{2} = \sin^2\frac{\theta}{2}.
  \end{align*}
  Moreover,
  \begin{align*}
  		\eta^- = |\eta| \sin^2\frac{\theta}{2} \frac{\eta}{|\eta|} - |\eta| \sin\frac{\theta}{2}\cos\frac{\theta}{2}\, \omega, \quad \text{and} \quad 	|\eta^-| = |\eta| \sin\frac{\theta}{2} ,
  \end{align*}
  so
  \begin{align*}
  	\frac{|\eta^+|^2}{|\eta^-|^2} = \frac{\cos^2\frac{\theta}{2}}{\sin^2\frac{\theta}{2}}
  		= \cot^2\frac{\theta}{2} .
  \end{align*}
After this preparation, using also $\langle \eta^+\rangle^{2\alpha} \leq  \langle \eta\rangle^{2\alpha}$ and $\sin\frac{\theta}{2} \leq \sin \theta$ for $\theta\in[0, \frac{\pi}{2}]$, the inequality $\widetilde{I}_{d,\Lambda}\le I_{d,\Lambda}$ is immediate. The inclusion of the additional factor $\1_{\Lambda}(|\eta|) = \1_{\sin\frac{\theta}{2} \Lambda}(|\eta^-|) \leq \1_{\Lambda/\sqrt{2}}(|\eta^-|)$ seems artificial for the moment, but will be convenient to keep track of the fact that $\eta^-$ is always restricted to a ball of radius $\tfrac{\Lambda}{\sqrt{2}}$.

Concerning $\widetilde{I}_{d,\Lambda}$, we want to implement a change of variables from $\eta$ to $\eta^+$.
As a function of $\eta$ and $\sigma$, $\eta^+= \frac{1}{2}(\eta-|\eta|\sigma)$. Thus
\begin{align*}
	\left| \frac{\partial\eta^+}{\partial\eta} \right| = \left| \frac{1}{2} \left(\1 + \frac{\eta}{|\eta|}\otimes \sigma \right) \right| = \frac{1}{2^d} \left(1+ \frac{\eta}{|\eta|}\cdot\sigma\right) \geq \frac{1}{2^d},
\end{align*}
since $\eta\cdot \sigma \geq 0$ and the second equality is an application of Sylvester's determinant theorem. Therefore, the Jacobian of the transformation  from $\eta$ to $\eta^+$ can be bounded by
\begin{align*}
	\left| \frac{\partial\eta}{\partial\eta^+} \right|
	 = \left| \frac{\partial\eta^+}{\partial\eta} \right|^{-1} \le 2^d.
\end{align*}
In addition,
\begin{align*}
	|\eta^+|^2 = \frac{|\eta|^2}{2}\left(1+\frac{\eta\cdot\sigma}{|\eta|}\right) \quad \text{and} \quad \eta^+\cdot\sigma = \frac{|\eta|}{2} \left(1+\frac{\eta\cdot\sigma}{|\eta|}\right) = \frac{|\eta^+|^2}{|\eta|},
\end{align*}
 which implies
\begin{align*}
 	\frac{\eta^+\cdot\sigma}{|\eta^+|} = \frac{|\eta^+|}{|\eta|} \quad \text{and} \quad \frac{\eta\cdot\sigma}{|\eta|} = 2\frac{|\eta^+|^2}{|\eta|^2} - 1 = 2 \left(\frac{\eta^+\cdot\sigma}{|\eta^+|}\right)^2 - 1.
\end{align*}
Moreover, from the definition of $\eta^\pm$ one sees
\begin{align*}
	\eta= 2\eta^+ -|\eta|\sigma
\end{align*}
so
\begin{align*}
	\eta^- = \eta^+- |\eta|\sigma = \eta^+ - |\eta^+| \left(\frac{\eta^+\cdot\sigma}{|\eta^+|}\right)^{-1}\sigma .
\end{align*}
Therefore, taking care of the domain of integration,
\begin{align*}
	\widetilde{I}_d^+ \leq 2^{d} \int_{\R^d} \int_{\S^{d-1}} &b\left( 2 \left(\frac{\eta^+\cdot\sigma}{|\eta^+|}\right)^2 - 1\right)  \, \left(1-\left(\frac{\eta^+\cdot\sigma}{|\eta^+|}\right)^2 \right) \, \1_{\frac{\eta^+\cdot\sigma}{|\eta^+|}\Lambda}(|\eta^+|)\\
	&\qquad \qquad \times G^{\epsilon\left(\alpha, |\eta^+|^2/|\eta^-|^2\right)}(\eta^-) |\hat{f}(\eta^-)|\, |G_{\Lambda}(\eta^+) \hat{f}(\eta^+)|^2 \langle\eta^+\rangle^{2\alpha}  \, \mathrm{d}\sigma\,\mathrm{d}\eta^+.
\end{align*}
Introducing spherical coordinates with north pole in the $\eta^+$ direction, one has
\begin{align*}
	\sigma =\sigma(\vartheta,\omega)= \cos(\vartheta) \frac{\eta^+}{|\eta^+|} + \sin (\vartheta)\, \omega
\end{align*}
where now $\cos\vartheta= \tfrac{\eta^+\cdot\sigma}{|\eta^+|}$. From figure \ref{fig:geometry} one sees
$\vartheta =\tfrac{\theta}{2}\in [0,\pi/4]$. In this parametrization one  has
\begin{align*}
	\eta^- = \eta^+ - \frac{|\eta^+|}{\cos\vartheta}\sigma
		= - |\eta^+|\tan(\vartheta)\, \omega
\end{align*}
and again
$\mathrm{d}\sigma = \sin^{d-2}\vartheta \,\mathrm{d}\vartheta\mathrm{d}\omega$.
Thus
\begin{align*}
	\widetilde{I}_d^+ &\leq 2^{d} \int_{\R^d} \int_{\S^{d-2}} \int_0^{\frac{\pi}{4}} b\left( \cos 2\vartheta\right)  \, \sin^d \vartheta \, G^{\epsilon\left(\alpha, \cot^2\vartheta\right)}(\eta^-) |\hat{f}(\eta^-)|\, \1_{(\cos\vartheta) \Lambda}(|\eta^+|)\\  & \qquad \qquad \qquad \qquad \qquad  \times |G(\eta^+) \hat{f}(\eta^+)|^2 \langle\eta^+\rangle^{2\alpha} \, \mathrm{d}\vartheta\,\mathrm{d}\omega \,\mathrm{d}\eta^+.
\end{align*}
Since $|\eta^-|=  |\eta^+| \tan\vartheta$ we obtain $\1_{(\cos\vartheta)\Lambda}(|\eta^+|) = \1_{(\sin\vartheta)\Lambda}(|\eta^-|) \leq \1_{\Lambda/\sqrt{2}}(|\eta^-|)$ since $\vartheta\in [0,\pi/4]$. Hence $\widetilde{I}_{d,\Lambda}^+\le I_{d,\Lambda}^+$.

The proof in the $d=1$ case is completely analogous.
\end{proof}

\subsection{Extracting pointwise information from local $L^2$ bounds}

\begin{lemma}\label{lem:1dlemma}
	Let $m\geq 2$ and $h\in W^{m, \infty}(\R)$ and $q\geq \frac{1}{m}$. 
	Then there exists a constant $L_{m}<\infty$ depending only on $q, m, \|h\|_{L^{\infty}(\R)}$ and $\|h^{(m)}\|_{L^{\infty}(\R)}$ such that
	\begin{align*}
		|h(r)|^{q} \leq L_{m} \int_{\Omega_r} |h(\xi)|^{q-\tfrac{1}{m}}\,\mathrm{d}\xi \quad \text{for all } r\in\R,
	\end{align*}
	where $\Omega_r = [r, r+2]$ if $r\geq 0$ and $\Omega_r = [r-2, r]$ if $r<0$.
\end{lemma}

Looking into the proof of Lemma \ref{lem:1dlemma}, it is clear that its $m=1$ version also holds, even with a much simpler proof. 
Before actually going into the proof, we state an important consequence of it, which will enable us to get pointwise decay estimates on a function once suitable $L^2$ norms are bounded. 

For $m\in\N$ define $\|D^m f\|_{L^{\infty}(\R^d)}:= \sup_{\omega\in\S^{d-1}} \| (\omega\cdot\nabla)^m f\|_{L^{\infty}(\R^d)}$. Notice that this norm is invariant under rotations of the function $f$. 

\begin{corollary}\label{cor:ddlemma}
	Let $H\in \mathcal{C}^{m}(\R^n)$. Then there exists a constant $L_{m,n}<\infty$ (depending only on $ m, n, \|H\|_{L^{\infty}(\R^n)}$ and, $\|D^m H\|_{L^{\infty}(\R^n)}$) such that
	\begin{align*}
		|H(x)| \leq L_{m,n} \left( \int_{Q_x} |H(\xi)|^2 \,\mathrm{d}\xi \right)^{\frac{m}{2m+n}},
	\end{align*}
	where $Q_x$ is a cube in $\R^n$ of side length $2$, with $x$ being one of the corners, such that it is oriented away from $x$ in the sense that $x\cdot (\xi-x) \geq 0$ for all $\xi\in Q_x$.
\end{corollary}

\begin{remark}
The constant $L_{m,n}$ in Corollary \ref{cor:ddlemma} is invariant under rotations of the function $H$. This will be convenient for its application in Sections \ref{subsec:partii} and \ref{subsec:partiii}. 
\end{remark}

\begin{proof}
	We apply Lemma \ref{lem:1dlemma} iteratively in each coordinate direction to obtain
	\begin{align*}
		|H(x_1, x_2, \dots, x_n)|^{2+\frac{n}{m}} &\leq L_m^{(1)} \int_{\Omega_{x_1}} |H(\xi_1, x_2, \dots, x_d)|^{2+\frac{n-1}{m}}\,\mathrm{d}\xi_1 \\
		&\leq L_m^{(1)}L_m^{(2)} \int_{\Omega_{x_1}}\int_{\Omega_{x_2}} |H(\xi_1, \xi_2, x_3 \dots, x_d)|^{2+\frac{n-2}{m}}\,\mathrm{d}\xi_1\,\mathrm{d}\xi_2 \\
		&\leq L_m^{(1)}\cdots L_m^{(n)} \int_{\Omega_{x_1}}\cdots\int_{\Omega_{x_d}} |H(\xi_1, \dots, \xi_d)|^{2}\,\mathrm{d}\xi_1\cdots\,\mathrm{d}\xi_n.
	\end{align*}
	The constants $L_m^{(i)}$, $i=1, \dots, n$, only depend on $m$, \
	\begin{align*}
		\|H(x_1, \dots, x_{i-1}, \,\cdot\,, x_{i+1}, \dots, x_n)\|_{L^{\infty}(\R)} &\leq \|H\|_{L^{\infty}(\R^n)} \\
		\intertext{and }
		\|\partial_i^m H(x_1, \dots, x_{i-1}, \,\cdot\,, x_{i+1}, \dots, x_n)\|_{L^{\infty}(\R)} &\leq \|D^m H\|_{L^{\infty}(\R^n)}.
	\end{align*}
	Setting $L_{m,n} = \prod_{i=1}^n L_m^{(i)}$ yields the stated inequality with $Q_x = \Omega_{x_1}\times \cdots \times \Omega_{x_n}$.
\end{proof}
\begin{remark}
It is worth noticing that the exponent in Corollary \ref{cor:ddlemma} is decreasing in the dimension and increasing in $m$.
\end{remark}

For the proof of Lemma \ref{lem:1dlemma} we need the following interpolation result between $L^{\infty}$ norms of derivatives of a function.

\begin{lemma}[Kolmogorov-Landau inequality on the unit interval]\label{lem:landau}
	Let $m\geq 2$ be an integer. There exists a constant $C_m>0$ such that for all $w\in W^{m, \infty}([0,1])$,
	\begin{align*}
		\|w^{(k)}\|_{L^{\infty}([0,1])} \leq C_m \left( \frac{\|w\|_{L^{\infty}([0,1])}}{u^k} + u^{m-k} \|w^{(m)}\|_{L^{\infty}([0,1])} \right), \quad k=1, \dots, m-1,
	\end{align*}
	for all $0< u \leq 1$.
\end{lemma}

\begin{proof}
The result dates back to \textsc{E. Landau} and \textsc{A. N. Kolmogorov} who proved it on $\R$ and $\R^+$. A proof of the inequality on a finite interval can be found in the book by \textsc{R. A. DeVore} and \textsc{G. G. Lorentz} \cite{DL93} (pp.37--39), but for the reader's convenience we also give a short proof in Appendix \ref{sec:app-landau}.
\end{proof}

For us, the important consequence we are going to make use of is
\begin{corollary}\label{cor:landau}
	Let $C_m>0$ be the constant from Lemma \ref{lem:landau}. Then for all $w\in W^{m,\infty}([0,1])$,
	\begin{align}\label{eq:landau-multiplicative}
		\|w^{(k)}\|_{L^{\infty}([0,1])} \leq 2C_m \|w\|_{L^{\infty}([0,1])}^{1-k/m} \max\left\{\|w\|_{L^{\infty}([0,1])}^{k/m}, \|w^{(m)}\|_{L^{\infty}([0,1])}^{k/m} \right\}, \quad k=1, \dots, m-1.
	\end{align}
\end{corollary}
\begin{proof}
	If $\|w^{(m)}\|_{L^{\infty}([0,1])} \leq \|w\|_{L^{\infty}([0,1])}$, we choose $u=1$ in the bound from Lemma \ref{lem:landau}, which gives
	\begin{align*}
		\|w^{(k)}\|_{L^{\infty}([0,1])} \leq 2 C_m \|w\|_{L^{\infty}([0,1])}
	\end{align*}
	in this case, and  if $\|w^{(m)}\|_{L^{\infty}([0,1])} \geq \|w\|_{L^{\infty}([0,1])}$, we can choose $u = \|w\|_{L^{\infty}([0,1])}^{1/m} \|w^{(m)}\|_{L^{\infty}([0,1])}^{-1/m} \leq 1$ to obtain
	\begin{align*}
		\|w^{(k)}\|_{L^{\infty}([0,1])} \leq 2 C_m \|w\|_{L^{\infty}([0,1])}^{1-k/m} \|w^{(m)}\|_{L^{\infty}([0,1])}^{k/m}.
	\end{align*}
	Together this proves \eqref{eq:landau-multiplicative}.
\end{proof}

We can now turn to the
\begin{proof}[Proof of Lemma \ref{lem:1dlemma}]
Assume without loss of generality that $r\geq 0$, so that $\Omega_r = [r, r+2]$. By the Sobolev embedding theorem $h$ is continuous and we let $r^*$ be a point in $\Omega_r$ where $|h|$ attains its maximum. We can assume that $r^* \in [r, r+1]$ and set $\langle h\rangle_{r^*}:= \int_{r^*}^{r^*+1}h(\xi)\,\mathrm{d}\xi$ (otherwise we use $\langle h\rangle_{r^*}:= \int_{r^*-1}^{r^*}h(\xi)\,\mathrm{d}\xi$). Then for some $p\geq 1$ we have
\begin{align*}
	|h(r^*)|^p - \left| \langle h^p\rangle_{r^*} \right| \leq \int_{r^*}^{r^*+1} |h^p(r^*) - h^p(\xi)| \,\mathrm{d}\xi = \int_0^1 |h^p(r^*) - h^p(r^*+\zeta)|\,\mathrm{d}\zeta.
\end{align*}
Bt the fundamental theorem of calculus, for any $\zeta\in[0,1]$ the integrand can be bounded by
\begin{align*}
	|h^p(r^*) - h^p(r^*+\zeta)| &\leq p \int_0^1 |h(r^*+s \zeta)|^{p-1} |h'(r^*+s\zeta)| \zeta\,\mathrm{d}s \\
	&\leq p \sup_{s\in[0,1]}|h'(r^* + s \zeta)| \int_0^1 |h(r^*+s \zeta)|^{p-1} \zeta\,\mathrm{d}s
\end{align*}
We now use that
\begin{align*}
	\sup_{s\in[0,1]}|h'(r^* + s \zeta)| = \sup_{x\in[0,\zeta]} |h'(r^* + x)| \leq \sup_{x\in[0,1]} |h'(r^* + x)| = \|h'(r^*+\cdot)\|_{L^{\infty}([0,1])}
\end{align*}
and apply the Kolmogorov-Landau inequality for the first derivative in its multiplicative form from Corollary \ref{cor:landau} to the function $[0,1]\ni x \mapsto h(r^*+x)\in W^{m,\infty}([0,1])$ to obtain
\begin{align*}
	\|h'(r^*+\cdot)\|_{L^{\infty}([0,1])} &\leq 2 C_m \|h(r^*+\cdot)\|_{L^{\infty}([0,1])}^{1-1/m} \max\left\{ \|h(r^*+\cdot)\|_{L^{\infty}([0,1])}^{1/m}, \|h^{(m)}(r^*+\cdot)\|_{L^{\infty}([0,1])}^{1/m}\right\} \\
	&\leq 2 C_m |h(r^*)|^{1-1/m} \max\left\{\|h\|_{L^{\infty}(\R)}^{1/m}, \|h^{(m)}\|_{L^{\infty}(\R)}^{1/m}\right\}.
\end{align*}
It follows that
\begin{align*}
	|h(r^*)|^p - \left| \langle h^p\rangle_{r^*} \right| &\leq 2 p C_m |h(r^*)|^{1-1/m} \max\left\{\|h\|_{L^{\infty}(\R)}^{1/m}, \|h^{(m)}\|_{L^{\infty}(\R)}^{1/m}\right\}  \\
	&\qquad \times \int_0^1 \int_0^1 |h(r^* + s\zeta)|^{p-1} \zeta\,\mathrm{d}s\,\mathrm{d}\zeta.
\end{align*}
The latter integral can be further estimated by
\begin{align*}
	\int_0^1 \int_0^1 |h(r^* + s\zeta)|^{p-1} \zeta\,\mathrm{d}s\,\mathrm{d}\zeta &= \int_0^1 \int_0^{\zeta} |h(r^*+x)|^{p-1} \,\mathrm{d}x\,\mathrm{d}\zeta 
	\\
	&\leq \int_0^1 \int_0^1 |h(r^*+x)|^{p-1} \,\mathrm{d}\zeta \,\mathrm{d}x = \int_0^1 |h(r^*+x)|^{p-1} \,\mathrm{d}x \\
	&= \int_{r^*}^{r^*+1} |h(\xi)|^{p-1}\,\mathrm{d}\xi \leq \int_{\Omega_r} |h(\xi)|^{p-1}\,\mathrm{d}\xi.
\end{align*}
Using
\begin{align*}
	\left|\langle h^p\rangle_{r^*}\right| \leq \int_{r^*}^{r^*+1} |h(\xi)|^p\,\mathrm{d}\xi &\leq \|h\|_{L^{\infty}(\Omega_r)} \int_{\Omega_r} |h(\xi)|^{p-1}\,\mathrm{d}\xi \\
	&\leq |h(r^*)|^{1-1/m} \|h\|_{L^{\infty}(\R)}^{1/m} \int_{\Omega_r} |h(\xi)|^{p-1}\,\mathrm{d}\xi
\end{align*}
we get
\begin{align*}
	|h(r^*)|^p \leq L_m |h(r^*)|^{1-1/m} \int_{\Omega_r} |h(\xi)|^{p-1}\,\mathrm{d}\xi
\end{align*}
with $L_m = 2pC_m \max\left\{\|h\|_{L^{\infty}(\R)}^{1/m}, \|h^{(m)}\|_{L^{\infty}(\R)}^{1/m}\right\} +\|h\|_{L^{\infty}(\R)}^{1/m}$,
and therefore
\begin{align*}
	|h(r^*)|^{p-1+1/m} \leq L_m \int_{\Omega_r} |h(\xi)|^{p-1}\,\mathrm{d}\xi.
\end{align*}
Choosing $q:=p-1+1/m \geq 1/m$ then yields
\begin{align*}
	|h(r)|^q \leq |h(r^*)|^q \leq L_m \int_{\Omega_r} |h(\xi)|^{q-1/m}\,\mathrm{d}\xi,
\end{align*}
which is the claimed inequality.
\end{proof}

\subsection{Gevrey smoothing of weak solutions for $L^2$ initial data: Part I}

Equipped with Corollary \ref{cor:ddlemma} we can construct an inductive scheme based upon a uniform bound on $G(\eta^-)^{\epsilon(\alpha, 1)} |\hat{f}(\eta^-)|$. As already remarked, this result will depend on the dimension, and will actually deteriorate quickly as dimension increases. Nevertheless it leads to strong regularity properties of weak solutions in the physically relevant cases.

\begin{theorem}\label{thm:gevrey1}
	Assume that the initial datum $f_0$ satisfies $f_0\geq 0$, $f_0 \in L\log L(\R^d) \cap L^1_{m}(\R^d)$ for some $m\geq 2$, and, in addition, $f_0\in L^2(\R^d)$.
	Further assume that the cross-section $b$ satisfies the \emph{singularity condition} \eqref{eq:cross-section} and the \emph{integrability condition} \eqref{eq:cross-section2} for $d\geq 2$, and for $d=1$, $b_1$ satisfies the \emph{singularity condition} \eqref{eq:cross-section-kac} and the \emph{integrability condition} \eqref{eq:cross-section-kac2} for some $0<\nu<1$. Let $f$ be a weak solution of the Cauchy problem \eqref{eq:cauchyproblem} with initial datum $f_0$. 	Set $\alpha_{m,d}:=\log\left(\frac{4m+d}{2m+d}\right)/\log 2$.
	Then, for all $0<\alpha\leq \min\left\{\alpha_{m,d}, \nu\right\}$ and $T_0>0$, there exists $\beta>0$, such that for all $t\in[0, T_0]$
 	\begin{align}\label{eq:local-time-uniform-1}
		e^{\beta t \langle D_v\rangle^{2\alpha}} f(t,\cdot) \in L^2(\R^d),
	\end{align}
	that is, $f\in G^{\tfrac{1}{2\alpha}}(\R^d)$ for all $t\in(0,T_0]$.
\end{theorem}
By decreasing $\beta$, if necessary, one even has a uniform bound,
\begin{corollary}\label{cor:gevrey1}
		Let $T_0>0$. Under the same conditions as in Theorem \ref{thm:gevrey1} there exit  $\beta>0$ and $M_1<\infty$ such that
	\begin{align}\label{eq:cor-gevrey1}
		\sup_{0\le t\le T_0}\sup_{\eta\in\R^d} e^{ \beta t \langle \eta \rangle^{2\alpha}} \, |\hat{f}(t,\eta)| \le M_1 .
	\end{align}
\end{corollary}

\begin{remark}
	\begin{enumerate}[label=(\roman*)]
		\item For strong singularities, the restriction on the Gevrey class originates in the bound on the commutation error, with the best value in $d=1$ dimension. The aim of part II below will be to recover the two-dimensional result in \emph{any dimension} $d\geq 2$. Under slightly stronger assumptions on the angular cross-section, which still covers all physically relevant cases, we can get the one-dimensional result in any dimension $d\geq 1$, see part III.
		\item In dimensions $d=1, 2, 3$ and $m=2$, corresponding to initial data with finite energy, we have $\alpha_{2,d} = \log\left(\frac{8+d}{4+d}\right)/\log 2 \geq \log\left(\frac{11}{7}\right)/\log 2 \simeq 0.652077$. This means that for $\nu=\frac{1}{2}$ the weak solution gets analytic and even ultra-analytic for $\nu>\frac{1}{2}$.
		\item In the case of physical Maxwellian molecules, where $\nu=\frac{1}{4}$, in three dimensions and with initial datum having finite mass, energy and entropy, we obtain Gevrey $G^{2}(\R^3)$ regularity.
		\item Even though the range of $\alpha$ in Theorem \ref{thm:gevrey1} above  deteriorates as the dimension increases, it only fails to cover (ultra-)analyticity results in dimensions $d\geq 6$. Theorems \ref{thm:gevrey2} and \ref{thm:gevrey3} below yield results uniformly in the dimension.
	\end{enumerate}
\end{remark}

We will prove Theorem \ref{thm:gevrey1} inductively over suitable length scales $\Lambda_N\to\infty$ as $N\to\infty$ in Fourier space.
To prepare for this, we fix some $M<\infty$, $0<T_0<\infty$ and introduce

\begin{definition}[Hypothesis $\mathrm{Hyp1}_{\Lambda}(M)$] \label{def:Hyp_Lambda}
    Let $M\geq 0$. Then  for all $0\le t\le T_0$
    \begin{align}\label{eq:Hyp_Lambda}
    	\sup_{|\zeta|\leq \Lambda} G(t, \zeta)^{\epsilon(\alpha,1)} |\hat{f}(t,\zeta)| \leq M.
    \end{align}
\end{definition}
\begin{remark}
    Recall that $G(t,\zeta)= e^{\beta t \langle\zeta\rangle^{\alpha}}$, that is, it depends on $\alpha, \beta$, and $t$, and also $f$ is a time dependent function,
    even though we suppress this dependence in our notation.
     Thus $\mathrm{Hyp1}_{\Lambda}(M)$ also depends on the parameters in $G(t,\zeta)$ and on $M$ and $T_0$, which, for simplicity, we do not emphasise in our notation.
     We will later fix some $T_0>0$ and a suitable large enough $M$.
     The main reason why this is possible is that, since $\|\hat f\|_{L^\infty} \le \|f\|_{L^1} = \|f_0\|_{L^1}<\infty$,
    for any $\Lambda, \beta, T_0>0$ the hypothesis $\mathrm{Hyp1}_{\Lambda}(M)$ is true for large enough $M$ and even any $M> \|f_0\|_{L^1}$ is possible by choosing
    $\beta>0$ small enough.
\end{remark}

A first step into the inductive proof is the following
\begin{lemma}\label{lem:induction1}
	Let $\alpha \leq \nu$ and define $c_{b,d}:= |\S^{d-2}| \int_{0}^{\tfrac{\pi}{2}} \sin^d\theta\, b(\cos\theta)\,\mathrm{d}\theta$ for $d\geq 3$, $c_{b,2}:= \int_{0}^{\tfrac{\pi}{2}} \sin^2\theta\, b(\cos\theta)\,\mathrm{d}\theta$, $c_{b,1}:= \int_{-\tfrac{\pi}{4}}^{\tfrac{\pi}{4}} \sin^2\theta\, b_1(\theta)\,\mathrm{d}\theta$, which are finite by the integrability assumptions \eqref{eq:cross-section2} and \eqref{eq:cross-section-kac2}, and let
	$\beta\le \frac{\tilde{C}_{f_0}}{(1+2^{d-1})\, c_{b,d}\alpha T_0 M +1}$. Then, for any weak solution of the homogenous Boltzmann equation,
	\begin{equation} \label{eq:inductionbound1}
		\mathrm{Hyp1}_{\Lambda}(M) \quad \Rightarrow \quad \|G_{\sqrt{2}\Lambda} f \|_{L^2(\R^d)} \le \|\1_{\sqrt{2}\Lambda}(D_v) f_0\|_{L^2(\R^d)} \, e^{C_{f_0}T_0}
	\end{equation}
	for all $0\le t\le T_0$.
\end{lemma}

\begin{remark}
The main point of this lemma is that the right hand side of \eqref{eq:inductionbound1} does not depend on $M$. This is crucial for our analysis and might seem a bit surprising, at first. It is achieved by making $\beta$ small enough.
\end{remark}

\begin{proof}
  Let $d\ge 2$.
  Since $\cot^2\tfrac{\theta}{2} \geq 1$ for $\theta\in[0, \frac{\pi}{2}]$ and $\cot^2\vartheta \geq 1$ for $\vartheta\in [0, \frac{\pi}{4}]$, we can bound $\epsilon(\alpha, \cot^2\tfrac{\theta}{2})$ and $\epsilon(\alpha, \cot^2\vartheta)$ by $\epsilon(\alpha, 1)$ in the integrals $I_{d, \sqrt{2}\Lambda}$ and $I_{d, \sqrt{2}\Lambda}^+$ from Lemma \ref{lem:ce}.

  Assume $\mathrm{Hyp1}_{\Lambda}(M)$ holds. Then
  \begin{align*}
	G(t,\zeta)^{\epsilon(\alpha,1)} |\hat{f}(t,\zeta)| \leq M \quad \text{for all}\quad  |\zeta|\leq \Lambda .
  \end{align*}
  In particular, the terms containing $\eta^-$ in $I_{d, \sqrt{2}\Lambda}$ and $I_{d, \sqrt{2}\Lambda}^+$ can be bounded by $M$.  Thus, these integrals can now be further estimated by
 \begin{align*}
 	I_{d, \sqrt{2}\Lambda} &\leq \alpha \beta t \, M \,|\S^{d-2}| \int_{0}^{\tfrac{\pi}{2}}   \sin^{d}\theta \,b(\cos\theta)\,\mathrm{d}\theta \int_{\R^d} |G_{\sqrt{2}\Lambda}(\eta) \hat{f}(\eta)|^2 \,\langle \eta \rangle^{2\alpha}\,\mathrm{d}\eta \\
 	&=  \alpha \beta t \, M \, c_{b,d} \|G_{\sqrt{2}\Lambda}f\|_{H^{\alpha}(\R^d)}^2
 \intertext{and,}
	I_{d, \sqrt{2}\Lambda}^+ &\leq 2^{d} \alpha\beta t\, M \, |\S^{d-2}|  \int_0^{\frac{\pi}{4}}   \sin^d \vartheta  \, b\left( \cos 2\vartheta\right) \, \mathrm{d}\vartheta \int_{\R^d} |G(\eta^+) \hat{f}(\eta^+)|^2 \langle\eta^+\rangle^{2\alpha} \,\mathrm{d}\eta^+.
 \end{align*}
 In the $\vartheta$ integral, we bound $\sin\vartheta \leq \sin(2\vartheta)$ to obtain
  \begin{align*}
	I_{d, \sqrt{2}\Lambda}^+ &\leq 2^{d-1} \alpha\beta t\, M \, c_{b,d} \|G_{\sqrt{2}\Lambda}f\|_{H^{\alpha}(\R^d)}^2
 \end{align*}
By Lemma \ref{lem:ce}, the commutation error corresponding to the weight $G_{\sqrt{2}\Lambda}$ is thus bounded by
	\begin{align}\label{eq:ce-etaetaplus}
	\begin{split}
	\left| \left\langle Q(f, G_{\sqrt{2}\Lambda}f) - G_{\sqrt{2}\Lambda} Q(f,f), G_{\sqrt{2}\Lambda}f\right\rangle\right| &\leq I_{d,\sqrt{2}\Lambda} + I_{d, \sqrt{2}\Lambda}^+ \\
    &\leq (1 + 2^{d-1})\, \alpha \beta t \, M \, c_{b,d} \|G_{\sqrt{2}\Lambda}f\|_{H^{\alpha}(\R^d)}^2 .
	\end{split}
 	\end{align}
With Corollary \ref{cor:gronwallbound} we then have
    \begin{align*}
    	\|G_{\sqrt{2}\Lambda}f\|_{L^2(\R^d)}^2 \leq &\|\1_{\sqrt{2}\Lambda}(D_v)f_0\|_{L^2}^2 + \int_0^t 2 C_{f_0} \|G_{\sqrt{2}\Lambda}f\|_{L^2(\R^d)}^2 \, \mathrm{d}\tau \\
    	&\qquad+\int_0^t 2 \left(- \widetilde{C}_{f_0} \|G_{\sqrt{2}\Lambda}f\|_{H^{\nu}(\R^d)}^2
	  + \left((1+2^{d-1})\, \alpha \beta t \, M \, c_{b,d} + \beta\right) \|G_{\sqrt{2}\Lambda} f\|_{H^{\alpha}(\R^d)}^2  \right)\, \mathrm{d}\tau .
    \end{align*}
    Since $\alpha\leq\nu$ and $\beta \le \frac{\widetilde{C}_{f_0}}{(1+2^{d-1}) c_{b,d} \, \alpha T_0 M + 1}$, this implies
    \begin{align*}
    \|G_{\sqrt{2}\Lambda}f\|_{L^2(\R^d)}^2 \leq \|\1_{\sqrt{2}\Lambda}(D_v)f_0\|_{L^2(\R^d)}^2 + \int_0^t 2C_{f_0} \|G_{\sqrt{2}\Lambda} f\|_{L^2(\R^d)}^2 \, \mathrm{d}\tau
    \end{align*}
    and with Gronwall's inequality
    \begin{align}\label{eq:gronwall-final}
    	\|G_{\sqrt{2}\Lambda}f\|_{L^2(\R^d)}^2 \leq \|\1_{\sqrt{2}\Lambda}(D_v) f_0\|_{L^2(\R^d)}^2 \, e^{2C_{f_0} T_0} 
    \end{align}
  follows.

  For $d=1$, we note that, with the obvious change in notation, the above proof literally translates to the Kac equation.
\end{proof}

The second ingredient gives a uniform bound in terms of a weighted $L^2$ norm and some a priori uniform bound on some higher derivative of $\hat{f}$.

\begin{lemma} \label{lem:induction2}
    	Assume that there exist finite constants $A_m$ and $B$, such that
    	\begin{equation}\label{eq:ind2-assumption1}
    	\|f(t, \cdot)\|_{L^1_m} \leq A_m , \quad  \text{and } \quad 	\| (G_{\sqrt{2}\Lambda} f)(t, \cdot) \|_{L^2(\R^d)} \le B
    	\end{equation}
    	for some integer $m\geq 2$ and for all $0\le t\le T_0$. Set
    	\begin{equation}\label{eq:ind2-tildeLambda}
    		\widetilde{\Lambda}:= \frac{1+\sqrt{2}}{2} \Lambda
    	\end{equation}
    	and assume furthermore that
    	\begin{equation}\label{eq:ind2-assumption2}
    		\Lambda\ge \Lambda_0 := \frac{4 \sqrt{d}}{\sqrt{2}-1}.
    	\end{equation}
    	Then for all $ |\eta|\le \widetilde{\Lambda}$
    	\begin{equation}\label{eq:Gbound}
    		|\hat{f}(t, \eta)| \leq K_1 \, G(t,\eta)^{-\frac{2m}{2m+d}} \quad \text{for all} \quad 0\leq t \leq T_0
    	\end{equation}
		with a constant $K_1$ depending only on the dimension $d$, $m$, $A_m$, and  $B$.
\end{lemma}

\begin{remark}\label{rem:choice_alpha}
The exponent $\frac{2m}{2m+d}$ in equation \eqref{eq:Gbound} comes from Corollary \ref{cor:ddlemma}, choosing $n=d$. This is responsible for our definition of $\alpha_{m,d}$, since then $\epsilon\left(\alpha_{m,d},1\right) = \frac{2m}{2m+d}$.
\end{remark}

\begin{remark} \label{rem:ind2}
 The assumptions of Lemma \ref{lem:induction2} are quite natural: since the Boltzmann equation conserves mass and kinetic energy does not increase, we have the a priori estimate
\begin{align*}
	\|f(t,\cdot)\|_{L^1_2(\R^d)} \leq  \|f_0\|_{L^1_2(\R^d)} =: A_2,
\end{align*}
and due to the known results on moment propagation\footnote{see, for instance,  \textsc{Villani}'s review \cite{Vil02} pp. 73ff for references.} for the homogeneous Boltzmann equation in the Maxwellian molecules case, we have
\begin{align*}
	f_0\in L^1_m(\R^d) \quad \Longrightarrow \quad f(t,\cdot) \in L^1_m(\R^d) \text{ uniformly in } t\ge 0
 \end{align*}
for any $m>2$ in addition to assumptions \eqref{eq:initialdata}.
\end{remark}

The importance of Lemma \ref{lem:induction2} is that it effectively converts a local $L^2$ bound on suitable balls into a \emph{pointwise bound} on slightly smaller balls.

\begin{proof}[Proof of Lemma \ref{lem:induction2}]
By the Riemann-Lebesgue lemma $\hat{f}$ has continuous and bounded derivatives of order up to $m$. Since for any multi-index $\alpha\in\N_0^d$ one has $\partial^{\alpha}\hat{f} = (-2\pi\I)^{|\alpha|}\, \widehat{v^{\alpha} f}$, we obtain the bound
   \begin{align*}
   	\|D^m\hat{f}(t, \cdot)\|_{L^{\infty}(\R^d)} &= \sup_{\omega\in\S^{d-1}} \|(\omega\cdot\nabla)^m \hat{f}(t, \cdot)\|_{L^{\infty}(\R^d)} \leq \sup_{\omega\in\S^{d-1}} \sup_{\eta\in\R^d} \sum_{|\alpha|=m} \binom{m}{\alpha} \, |\omega^{\alpha}|\, |\partial^{\alpha}\hat{f}(\eta)| \\
   	&\leq  (2\pi)^m \sup_{\omega\in\S^{d-1}} \int_{\R^d} \sum_{|\alpha|=m} \binom{m}{\alpha} \, |\omega^{\alpha} v^{\alpha}| \, f(v)\, \mathrm{d}v \leq  (2\pi)^m \sup_{\omega\in\S^{d-1}}\int_{\R^d} (\omega\cdot v)^m \, f(v)\,\mathrm{d}v \\
   	&\leq  (2\pi)^m \int_{\R^d} |v|^m \, f(v)\,\mathrm{d}v \leq(2\pi)^m \|f(t, \cdot)\|_{L^1_m(\R^d)} \leq (2\pi)^m A_m
   \end{align*}
Of course, also $\|\hat{f}\|_{L^{\infty}(\R^d)} \leq \|f\|_{L^1(\R^d)} \leq A_m$.

	Let $\eta\in \R^d$ such that $|\eta| \leq \widetilde{\Lambda}$. By Corollary \ref{cor:ddlemma} applied to the function $\hat{f}$, there is a constant $L_{m,d}$ that depends only on $d, m$, and $A_m$ such that
	\begin{align*}
		|\hat{f}(\eta)| \leq L_{m,d} \left(\int_{Q_\eta} |\hat{f}(\zeta)|^2\,\mathrm{d}\zeta \right)^{\frac{m}{2m+d}}
	\end{align*}
	where $Q_{\eta}$ is the cube of side length 2 at $\eta$, such that all sides are oriented away from the origin. The definitions of $\widetilde{\Lambda}$ and $\Lambda_0$ guarantee by Pythagoras' theorem, that, for $|\eta|\leq \widetilde{\Lambda}$, $Q_{\eta}$ always stays inside the ball around the origin with radius $\sqrt{2}\Lambda$. Since the orientation of $Q_{\eta}$ is such that $\eta$ is the point closest to the origin and the weight $G$ is radial and increasing, we have
	\begin{align*}
		|\hat{f}(\eta)| &\leq L_{m,d} \left(G(\eta)^{-2} \int_{Q_\eta} G(\zeta)^2|\hat{f}(\zeta)|^2\,\mathrm{d}\zeta \right)^{\frac{m}{2m+d}} \\
		&\leq L_{m,d}\,  G(\eta)^{-\frac{2m}{2m+d}} \left(\int_{\{|\eta|\leq \sqrt{2}\Lambda\}} G(\zeta)^2|\hat{f}(\zeta)|^2\,\mathrm{d}\zeta \right)^{\frac{m}{2m+d}} \\
		&\leq L_{m,d} B^{\frac{2m}{2m+d}}\, G(\eta)^{-\frac{2m}{2m+d}}.
	\end{align*}
	Setting $K_1:= L_{m,d} B^{\frac{2m}{2m+d}}$ yields the claimed inequality.
\end{proof}

\begin{proof}[Proof of Theorem \ref{thm:gevrey1}]
	By Lemma \ref{lem:induction1}, \ref{lem:induction2}, and Remark \ref{rem:ind2}, a suitable choice for $A_m$, $B$, and the length scales $\Lambda_N$ is
 \begin{align*}
    B&:= \|f_0\|_{L^2(\R^d)} e^{C_{f_0}T_0}, \\
    A_m&:= \sup_{t\geq 0} \|f(t, \cdot)\|_{L^1_m(\R^d)} < \infty ,
   \end{align*}
   and
  \begin{equation*}
    \Lambda_N := \frac{\Lambda_{N-1}+\sqrt{2}\Lambda_{N-1}}{2} = \frac{1+\sqrt{2}}{2}\Lambda_{N-1} =
    \left( \frac{1+\sqrt{2}}{2} \right)^N \Lambda_0
  \end{equation*}
  with  $\Lambda_0$ from \eqref{eq:ind2-assumption2}.

  Furthermore, we set
  \begin{equation*}
    M_1:= \max \left\{ 2 A_m+1, K_1 \right\}
  \end{equation*}	
  with the constant $K_1$ from equation \eqref{eq:Gbound}.

 For the start of the induction, we need $\mathrm{Hyp1}_{\Lambda_0}(M_1)$ to be true.
 Since
 \begin{equation*}
 	\sup_{0\le t\le T_0} \sup_{|\eta|\le \Lambda_0} G(\eta)^{\epsilon(\alpha,1)} |\hat{f}(\eta)|
 	\le e^{\epsilon(\alpha,1) \beta T_0 (1+\Lambda_0^2)^{\alpha}} A_m
  \end{equation*}
  and from our choice of $M_1$ there exists $\beta_0 >0$ such that $\mathrm{Hyp1}_{\Lambda_0}(M_1)$ is true for all $0\le\beta\le \beta_0$.

  Now, we choose
  \begin{equation*}
  	\beta=\min \left( \beta_0, \frac{\tilde{C}_{f_0}}{(1+2^{d-1}) c_{b,d} \, \alpha T_0 M_1 +1} \right).
  \end{equation*}
  With this choice, the conditions of Lemma \ref{lem:induction1} and \ref{lem:induction2} are fulfilled and $\mathrm{Hyp1}_{\Lambda_0}(M_1)$ is true.

  For the induction step assume that $\mathrm{Hyp1}_{\Lambda_N}(M_1)$ is true. Then Lemma \ref{lem:induction1} gives
  \begin{equation*}
  	\| G_{\sqrt{2}\Lambda_N} f\|_{L^2(\R^d)} \le \|\1_{\sqrt{2}\Lambda}(D_v) f_0\|_{L^2(\R^d)} \, e^{C_{f_0}T_0} \leq B.
  \end{equation*}
  Note that $\epsilon(\alpha,1)\le \frac{2m}{2m+d}$, since $\alpha\leq \min\left\{\alpha_{m,d},\nu\right\}$, see Remark \ref{rem:choice_alpha}. In addition,
  $\Lambda_{N+1}= \widetilde{\Lambda}_N$, so Lemma \ref{lem:induction2} shows
  \begin{equation*}
  \sup_{|\eta|\le \Lambda_{N+1}} G(\eta)^{\epsilon(\alpha,1)} |\hat{f}(\eta)|
  \le K_1\leq M_1,
  \end{equation*}
  that is, $\mathrm{Hyp1}_{\Lambda_{N+1}}(M_1)$ is true. By induction, it is true for all $N\in\N$.
  Invoking Lemma \ref{lem:induction1} again, we also have
  \begin{equation*}
  	\|G_{\sqrt{2}\Lambda_N}f\|_{L^2(\R^d)} \leq B
  \end{equation*}
  for all $N\in\N$ and passing to the limit $N\to\infty$, we see
  $\|Gf\|_{L^2(\R^d)} \leq B $, which concludes the proof of the theorem.
\end{proof}

\begin{proof}[Proof of Corollary \ref{cor:gevrey1}]
The proof of Theorem \ref{thm:gevrey1} showed that given $T_0>0$ there exists $M_1>0$ and $\beta>0$ such that Hyp$1_{\Lambda_N}(M_1)$ is true for all $N\in\N$. This clearly implies \eqref{eq:cor-gevrey1}.
\end{proof}

\subsection{Gevrey smoothing of weak solutions for $L^2$ initial data: Part II}\label{subsec:partii}
The results of Part I are best in one dimension and give the correct smoothing in terms of the Gevrey class for $\nu$ not too close to one, more precisely $\nu\le \alpha_{m,d}$. In order to improve this  in higher dimensions $d\geq 2$ and for a larger range of singularities $0<\nu<1$, the commutator estimates have to be refined. We have

\begin{theorem}\label{thm:gevrey2} Let $d\ge 3$.
	Assume that the initial datum $f_0$ satisfies $f_0\geq 0$, $f_0 \in L\log L(\R^d) \cap L^1_{m}(\R^d)$ for some $m\geq 2$, and, in addition, $f_0\in L^2(\R^d)$.
	Further assume that the cross-section $b$ satisfies the \emph{singularity condition} \eqref{eq:cross-section} and the \emph{integrability condition} \eqref{eq:cross-section2} for some $0<\nu<1$. Let $f$ be a weak solution of the Cauchy problem \eqref{eq:cauchyproblem} with initial datum $f_0$, 	
	then for all $0<\alpha\leq \min\left\{\alpha_{m,2}, \nu\right\}$ and $T_0>0$, there exists $\beta>0$, such that for all $t\in[0, T_0]$
 	\begin{align}\label{eq:local-time-uniform-2}
		e^{\beta t \langle D_v\rangle^{2\alpha}} f(t,\cdot) \in L^2(\R^d),
	\end{align}
	that is, $f\in G^{\tfrac{1}{2\alpha}}(\R^d)$ for all $t\in(0,T_0]$.

	In particular, the weak solution is real analytic if $\nu=\frac{1}{2}$ and ultra-analytic if $\nu>\frac{1}{2}$.
\end{theorem}
The beauty of this theorem is that, in contrast to Theorem \ref{thm:gevrey1}, its result does not deteriorate as dimension increases. We also have a corollary similar to Corollary \ref{cor:gevrey1}, however with a weaker conclusion. Moreover, it is \emph{not} uniform in the time $t\ge 0$ but only holds on finite, but arbitrary, time intervals $[0,T_0]$.
\begin{corollary}\label{cor:gevrey2}
	Under the same assumptions as in Theorem \ref{thm:gevrey2}, for any weak solution
	$f$ of the Cauchy problem \eqref{eq:cauchyproblem} and any $0<T_0<\infty$ there exists $\widetilde{\beta}>0$ and $M<\infty$ such that
	\begin{align}\label{eq:cor-gevrey2}
		\sup_{0\le t\le T_0}\sup_{\eta\in\R^d} e^{ \widetilde{\beta} t \langle \eta \rangle^{2\alpha}} \, |\hat{f}(t,\eta)| \le M .
	\end{align}
\end{corollary}
The proof of Theorem \ref{thm:gevrey2} is again based on an induction over length scales in Fourier space. Having a close look at the integrals $I_{d,\Lambda}$ and $I_{d,\Lambda}^+$ from Lemma \ref{lem:ce} and using that $\epsilon(\alpha,\gamma)$ is decreasing in $\gamma $, one sees that it should be enough to bound expressions of the form
\begin{align*}
	& \int_{\S^{d-2}(\eta)} G(\eta^-)^{\epsilon(\alpha,1)} |\hat{f}(\eta^-)|
		\1_{\tfrac{\Lambda}{\sqrt{2}}}(|\eta^-|)\, \mathrm{d}\omega \\
	\intertext{and}
	& \int_{\S^{d-2}(\eta^+)} G(\eta^-)^{\epsilon(\alpha,1)} |\hat{f}(\eta^-)|
		\1_{\tfrac{\Lambda}{\sqrt{2}}}(|\eta^-|)\, \mathrm{d}\omega
\end{align*}
uniformly in $\eta$ and $\theta$, respectively $\eta^+$ and $\vartheta$, with the parametrization \eqref{eq:etapar}, respectively \eqref{eq:etapluspar}, that is, instead of having to use the purely pointwise estimates  expressed in the hypothesis $\mathrm{Hyp}1_\Lambda$ from the previous section, one can take advantage of averaging over codimension $2$ spheres first. This motivates

\begin{definition}[Hypothesis $\mathrm{Hyp2}_{\Lambda}$(M)] \label{def:Hyp_Lambda2}
    Let $M\geq 0$ be finite. Then for all $0\le t\le T_0$,
    \begin{align}\label{eq:Hyp_Lambda2}
\sup_{\zeta\in\R^d\setminus\{0\}} \sup_{(z,\rho)\in A_{\Lambda}} \int_{\S^{d-2}(\zeta)} G\left(t,z\tfrac{\zeta}{|\zeta|} - \rho \omega \right)^{\epsilon(\alpha, 1)} \left|\hat{f}\left(t,z\tfrac{\zeta}{|\zeta|} - \rho \omega \right)\right|\,\mathrm{d}\omega \leq M,
    \end{align}
where $A_{\Lambda} = \{(z, \rho)\in \R^2: 0\leq z \leq \rho, z^2 + \rho^2 \leq \Lambda^2\}$ and $\S^{d-2}(\zeta) = \{ \omega\in \R^d: \omega \perp \zeta, |\omega| = 1\}$.
\end{definition}

Again, we have
\begin{lemma}\label{lem:induction1-2}
	Let $\alpha \leq \nu$, define $c_{b,d,2} = \int_{0}^{\frac{\pi}{2}} \sin^d \theta b(\cos\theta)\,\mathrm{d}\theta$ (which is finite by the integrability assumption \eqref{eq:cross-section2}), and
	let $\beta\le \frac{\tilde{C}_{f_0}}{(1+2^{d-1}) c_{b,d,2}\alpha T_0 M +1}$. Then, for any weak solution of the homogenous Boltzmann equation,
	\begin{equation} \label{eq:inductionbound1-2}
		(\mathrm{Hyp2}_{\Lambda}) \quad \Rightarrow \quad \|G_{\sqrt{2}\Lambda} f \|_{L^2(\R^d)} \le \|\1_{\sqrt{2}\Lambda}(D_v) f_0\|_{L^2(\R^d)} \, e^{C_{f_0}T_0}
	\end{equation}
	for all $0\le t\le T_0$.
\end{lemma}

\begin{proof} Using the monotonicity of $\epsilon(\alpha,\gamma)$ in $\gamma$ and \eqref{eq:Id} one sees
\begin{align*}
		I_{d,\sqrt{2}\Lambda} \le \alpha \beta t
			\int_{\R^d}&\left(\int_{0}^{\tfrac{\pi}{2}} \sin^{d}\theta b(\cos\theta)
			\left(\int_{\S^{d-2}(\eta)}  G(\eta^-)^{\epsilon\left(\alpha, 1\right)}\,|\hat{f}(\eta^-)| \, \1_{\Lambda}(|\eta^-|)\,\mathrm{d}\omega \right) \,\mathrm{d}\theta \right) \\
			&\qquad \times |G_{\sqrt{2}\Lambda}(\eta) \hat{f}(\eta)|^2 \,\langle \eta \rangle^{2\alpha}  \, \mathrm{d}\eta
\end{align*}
where $\eta^-= \eta^-(\eta,\theta,\omega)$ is expressed via the parametrization \eqref{eq:etapar}. For $\sigma=(\theta, \omega)\in [0, \frac{\pi}{2}]\times \S^{d-2}$, one has $\eta^-  = |\eta|\sin^2\frac{\theta}{2} \frac{\eta}{|\eta|}+ |\eta| \sin\frac{\theta}{2}\cos\frac{\theta}{2} \, \omega$ and if $|\eta|\le \sqrt{2}\Lambda$, then $|\eta^-|\le \Lambda$. Identifying $z= |\eta|\sin^2\tfrac{\theta}{2}$ and $\rho= |\eta|\sin\frac{\theta}{2}\cos\frac{\theta}{2} $, and the direction of $\zeta$ with the direction of $\eta$,
 hypothesis $(\mathrm{Hyp2}_{\Lambda})$ clearly implies
 \begin{align*}
 	\sup_{|\eta|\le \sqrt{2}\Lambda}\sup_{\theta\in [0,\pi/2]}
 	\int_{\S^{d-2}(\eta)}  G(\eta^-)^{\epsilon\left(\alpha, 1\right)}\,|\hat{f}(\eta^-)| \, \1_{\Lambda}(|\eta^-|)\,\mathrm{d}\omega
 	\le M
 \end{align*}
	It follows that
	\begin{align*}
		I_{d,\sqrt{2}\Lambda} &\leq  \alpha \beta t \, M \int_{\R^d} \int_0^{\frac{\pi}{2}} \sin^d\theta b(\cos\theta)\, \mathrm{d}\theta \,|G_{\sqrt{2}\Lambda}(\eta) \hat{f}(\eta)|^2 \, \langle \eta \rangle\,\mathrm{d}\eta \\
		&= \alpha \beta t \, M \, c_{b,d,2} \|G_{\sqrt{2}\Lambda}f\|_{H^{\alpha}(\R^d)}^2.
	\end{align*}
Similarly one has
\begin{align*}
		I_{d,\sqrt{2}\Lambda}^+ \le 2^{d} \alpha \beta t
			\int_{\R^d}&\left(\int_{0}^{\tfrac{\pi}{4}} \sin^{d}\vartheta b(\cos2\vartheta)
			\left(\int_{\S^{d-2}(\eta^+)}  G(\eta^-)^{\epsilon\left(\alpha, 1\right)}\,|\hat{f}(\eta^-)| \, \1_{\Lambda}(|\eta^-|)\,\mathrm{d}\omega \right) \,\mathrm{d}\vartheta \right) \\
			&\qquad \times |G_{\sqrt{2}\Lambda}(\eta^+) \hat{f}(\eta^+)|^2 \,\langle \eta^+ \rangle^{2\alpha}  \, \mathrm{d}\eta^+
\end{align*}
where $\eta^-= \eta^-(\eta,\vartheta,\omega)$ is expressed via the parametrization \eqref{eq:etapluspar}. The vectors $\eta^-$ and $\eta^+$ are orthogonal and we have $\eta^-  = -|\eta^+| \tan\vartheta\,\omega$ for $(\vartheta, \omega)\in[0, \frac{\pi}{4}] \times \S^{d-2}(\eta^+)$.
	
	Setting $z=0$ and $\rho= |\eta^+|\tan\vartheta$ we have $\rho=|\eta^-|\le \Lambda$ in the $\vartheta$ and $\eta^+$ integrals above. Thus
	  $(\mathrm{Hyp2}_{\Lambda})$ again implies
  \begin{align*}
 	\sup_{|\eta^+|\le \sqrt{2}\Lambda}\sup_{\vartheta\in [0,\pi/4]}
 	\int_{\S^{d-2}(\eta^+)}  G(\eta^-)^{\epsilon\left(\alpha, 1\right)}\,|\hat{f}(\eta^-)| \, \1_{\Lambda}(|\eta^-|)\,\mathrm{d}\omega
 	\le M
 \end{align*}
Hence,
	\begin{align*}
		I_{d,\sqrt{2}\Lambda}^+ &\leq 2^{d} \alpha \beta t \, M \int_0^{\frac{\pi}{2}} \sin^d\theta b(\cos\theta)\, \mathrm{d}\theta  \int_{\R^d} \,|G_{\sqrt{2}\Lambda}(\eta^+) \hat{f}(\eta^+)|^2 \, \langle \eta^+ \rangle\,\mathrm{d}\eta^+ \\
		&\leq 2^{d-1} \alpha \beta t \, M \, c_{b,d,2} \|G_{\sqrt{2}\Lambda}f\|_{H^{\alpha}(\R^d)}^2.
	\end{align*}
	The rest of the proof is the same as in the proof of Lemma \ref{lem:induction1}.
\end{proof}

To close the induction process, we next show
\begin{lemma}\label{lem:induction2-2}
	    	Let $\beta\leq \frac{1}{T_0}$.     	Assume that there exist finite constants $A_m$ and $B$, such that
    	\begin{equation}\label{eq:ind2-assumption1-2}
    	\|f(t, \cdot)\|_{L^1_m} \leq A_m , \quad  \text{and } \quad 	\| (G_{\sqrt{2}\Lambda} f)(t, \cdot) \|_{L^2(\R^d)} \le B
    	\end{equation}
    	for some integer $m\geq 2$ and for all $0\le t\le T_0$.
    	
	Set $\widetilde{\Lambda}:= \frac{1+\sqrt{2}}{2} \Lambda$
		and assume that
	\begin{align}\label{eq:ind2-2-assumption2}
		\Lambda\ge \Lambda_0 := \frac{4 \sqrt{2}}{\sqrt{2}-1} .
	\end{align}
Then for all $\zeta\in\R^d\setminus\{0\}$ and $0\leq z \leq \rho$ with $\rho^2 + z^2 \leq \widetilde{\Lambda}^2$ one has
    	\begin{equation*}
    		\int_{\S^{d-2}(\zeta)} \left|\hat{f}\left(t,z \tfrac{\zeta}{|\zeta|} + \rho \omega\right)\right|\,\mathrm{d}\omega \leq K_2 \, \widetilde{G}(t,z^2+ \rho^2 )^{-\frac{2m}{2m+2}} \quad \text{for all } 0\le t\le T_0
    	\end{equation*}
		 with a constant $K_2$ depending only on $d,m,A_m$, and $B$. Recall that
		 $\widetilde{G}(t, s) = e^{\beta t (1+s)^\alpha}$.
\end{lemma}

\begin{proof} Fix $0<t\le T_0$, $\zeta\in \R^d\setminus\{0\}$, and
	set $F(\rho, z):= \hat{f}(t, z \frac{\zeta}{|\zeta|}+\rho\omega)$, where we drop, for simplicity, the dependence on the time $t$ in our notation for $F$.
	Then, since $\|f(t,\cdot)\|_{L^1_m}\le A_m$ one has $\hat{f}(t,\cdot)\in \mathcal{C}^m(\R^d)$  and thus also  $F\in \mathcal{C}^m(\R^2)$ with $\|F\|_{L^{\infty}} \leq A_m$ $\|\partial^m_\rho F\|_{L^{\infty}} \leq (2\pi)^mA_m$, and $\|\partial^m_z F\|_{L^{\infty}} \leq (2\pi)^mA_m$ and Corollary \ref{cor:ddlemma} applied to $F$ yields
	\begin{align}\label{eq:2dim}
		\left|\hat{f}\left( z \tfrac{\zeta}{|\zeta|}+\rho\omega\right)\right| &\leq L_{m,2} \left( \int_{\rho}^{\rho+2} \int_z^{z+2}
		\left|\hat{f}\left( x \tfrac{\zeta}{|\zeta|} + y\omega \right)\right|^2\,\mathrm{d}x \mathrm{d}y \right)^{\frac{m}{2m+2}}.
	\end{align}
	where we also dropped the dependence of $\hat{f}$ on the time variable $t$. Furthermore, we will drop the time dependence of $G$ and $\widetilde{G}$ in the following, that is, $G(\xi)$  and $\widetilde{G}(s)$ will stand for $G(t,\xi)$, respectively $\widetilde{G}(t,s)$.
	
	To recover the $L^2$ norm of $G_{\sqrt{2}\Lambda}f$ in the right hand side of \eqref{eq:2dim} we now need to take care of three things:
	\begin{enumerate}[label=(\roman*)]
	\item Multiply with a suitable power of the radially increasing weight $G$.
		\item Integrate over the missing $d-2$ directions, which will be taken care of by integrating over $\S^{d-2}(\zeta)$ and taking into account additional factors to get the $d$-dimensional Lebesgue measure.
		\item Ensure that the region of integration $[\rho,\rho+2]\times[z, z+2]\times \S^{d-2}(\zeta)$ stays inside a ball of radius $\sqrt{2}\Lambda$ uniformly in the direction of $\zeta$. This we control by choosing $\Lambda_0$ large enough (a simple geometric consideration shows that $\Lambda_0$ from the statement of Lemma \ref{lem:induction2-2} works) and restricting $\rho$ and $z$ by $\rho^2 + z^2 \leq \widetilde{\Lambda}^2$.
	\end{enumerate}

	Let $z,\rho\ge 0$. In  the region of integration in \eqref{eq:2dim}, the point $\rho\omega + z\tfrac{\eta}{|\eta|}$ is closest to the origin in $\R^d$, and since the weight $G$ is radially increasing, we get
	\begin{align}\label{eq:2dbound}
	\begin{split}
		\left|\hat{f}\left(z \tfrac{\zeta}{|\zeta|}+\rho\omega\right)\right| &\leq L_{m,2} \widetilde{G}\left(z^2+\rho^2\right)^{-\frac{2m}{2m+2}} \\
		&\qquad \left( \int_{\rho}^{\rho+2} \int_z^{z+2} G\left(x \tfrac{\zeta}{|\zeta|}+y\omega\right)^2  \left|\hat{f}\left(x \tfrac{\zeta}{|\zeta|}+y\omega\right)\right|^2\,\mathrm{d}x\mathrm{d}y \right)^{\frac{m}{2m+2}}.
		\end{split}
	\end{align}
	
	 Assume that $z^2+\rho^2\le \widetilde{\Lambda}^2$. Then the integration of inequality \eqref{eq:2dbound} over $\S^{d-2}(\zeta)$ yields with an application of Jensen's inequality ($t\mapsto t^{\frac{m}{2m+2}}$ is concave!)
	\begin{align*}
		\int_{\S^{d-2}(\zeta)} &\left|\hat{f}\left(z \tfrac{\zeta}{|\zeta|}+\rho\omega\right)\right|\, \mathrm{d}\omega \leq L_{m,2}  |\S^{d-2}|^{\frac{m+2}{2m+2}}\, \widetilde{G}\left(z^2+ \rho^2 \right)^{-\frac{2m}{2m+2}}\\
		& \quad \times \left( \int_{\S^{d-2}(\zeta)} \int_{\rho}^{\rho+2} \int_z^{z+2} G_{\sqrt{2}\Lambda}\left(x \tfrac{\zeta}{|\zeta|}+y\omega\right)^2  \left|\hat{f}\left(x\tfrac{\eta}{|\eta|}+y\omega\right)\right|^2\,\mathrm{d}x\,\mathrm{d}y \,\mathrm{d}\omega \right)^{\frac{m}{2m+2}}.
	\end{align*}
	Now assume additionally $0\leq z \leq \rho$ and $\Lambda_0^2 \leq \rho^2 + z^2 \leq \widetilde{\Lambda}^2$. Since $0\leq z \leq \rho$ we have $\Lambda_0^2 \leq z^2 + \rho^2 \leq 2 \rho^2$ and therefore
	\begin{align*}
		&\int_{\S^{d-2}(\zeta)} \int_{\rho}^{\rho+2} \int_z^{z+2} G_{\sqrt{2}\Lambda}\left(x \tfrac{\zeta}{|\zeta|}+y\omega\right)^2  \left|\hat{f}\left(x \tfrac{\zeta}{|\zeta|}+y\omega\right)\right|^2\,\mathrm{d}x\,\mathrm{d}y \,\mathrm{d}\omega \\
		&\leq 2^{\frac{d-2}{2}} \Lambda_0^{2-d} \int_{\S^{d-2}(\zeta)} \int_{\rho}^{\rho+2} \int_z^{z+2} G_{\sqrt{2}\Lambda}\left(x \tfrac{\zeta}{|\zeta|}+y\omega\right)^2  \left|\hat{f}\left(x \tfrac{\zeta}{|\zeta|}+y\omega\right)\right|^2\,y^{d-2} \,\mathrm{d}x\,\mathrm{d}y \,\mathrm{d}\omega \\
		&\leq 2^{\frac{d-2}{2}} \Lambda_0^{2-d} \|G_{\sqrt{2}\Lambda}f\|_{L^2(\R^d)}^2,
	\end{align*}
	since $y^{d-2} \,\mathrm{d}x\,\mathrm{d}y \,\mathrm{d}\omega$ is the $d$-dimensional Lebesgue measure in the cylindrical coordinates $(x,y\omega)$ with $x\in \R $, $y>0$, $\omega\in \S^{d-2}(\zeta)$ along the cylinder with axis $\zeta $.
	So with the assumption $\|G_{\sqrt{2}\Lambda}f\|_{L^2(\R^d)} \leq B$ we obtain
	\begin{align*}
		\int_{\S^{d-2}(\zeta)} \left|\hat{f}\left(t,z \tfrac{\zeta}{|\zeta|}+\rho\omega\right)\right|\,\mathrm{d}\omega &\leq L_{m,2}  |\S^{d-2}|^{\frac{m+2}{2m+2}}\, \left( 2^{\frac{d-2}{2}} \Lambda_0^{2-d} B^2 \right)^{\frac{m}{2m+2}} \widetilde{G}\left(t,z^2+ \rho^2 \right)^{-\frac{2m}{2m+2}}.
	\end{align*}
	In the case $z^2+ \rho^2 \leq \Lambda_0^2$ we have $\widetilde{G}(t, z^2+\rho^2)^{-1}\, e^{\beta t (1+\Lambda_0^2)^{\alpha}} \geq 1$ and we can simply bound
	\begin{align*}
		\int_{\S^{d-2}(\zeta)} \left|\hat{f}\left(t,z \tfrac{\zeta}{|\zeta|}+\rho\omega\right)\right|\,\mathrm{d}\omega
		&\le
			\widetilde{G}\left(t,z^2+\rho^2\right)^{-\frac{2m}{2m+2}}
			e^{\tfrac{2m}{2m+2}\beta t(1+\Lambda_0^2)^\alpha}
			|\S^{d-2}| \, \|\hat{f}(t,\cdot)\|_{L^\infty(\R^d)} \\
		&\le A_m |\S^{d-2}| e^{1+\Lambda_0^2} \widetilde{G}\left(t,z^2+\rho^2\right)^{-\frac{2m}{2m+2}}
	\end{align*}
	since $\beta\le 1/T_0$, by assumption. So choosing
	\begin{align*}
		K_2:= \max\left( L_{m,2}  |\S^{d-2}|^{\frac{m+2}{2m+2}}\, \left( 2^{\frac{d-2}{2}} \Lambda_0^{2-d} B^2 \right)^{\frac{m}{2m+2}},  A_m |\S^{d-2}| e^{1+\Lambda_0^2}\right)
	\end{align*}
	finishes the proof of the lemma.
\end{proof}

Now we have all the ingredients for the inductive
\begin{proof}[Proof of Theorem \ref{thm:gevrey2}]
		By Lemmata \ref{lem:induction1-2} and \ref{lem:induction2-2} a suitable choice for $A_m$ and  $B$ is
 \begin{align*}
    B&:= \|f_0\|_{L^2(\R^d)} e^{C_{f_0}T_0}, \\
    A_m&:= \sup_{t\geq 0} \|f(t, \cdot)\|_{L^1_m(\R^d)} < \infty.
   \end{align*}
   Note that the finiteness of $A_m$ is guaranteed since $f_0 \in L^1_m(\R^d)$, see Remark \ref{rem:ind2}. We further choose the length scales $\Lambda_N$ to be
  \begin{equation*}
    \Lambda_N := \frac{\Lambda_{N-1}+\sqrt{2}\Lambda_{N-1}}{2} = \frac{1+\sqrt{2}}{2}\Lambda_{N-1} =
    \left( \frac{1+\sqrt{2}}{2} \right)^N \Lambda_0
  \end{equation*}
  with  $\Lambda_0$ now from \eqref{eq:ind2-2-assumption2}, and we set
  \begin{equation*}
    M_2:= \max \left\{ 2 |\S^{d-2}|A_m+1, K_2 \right\}
  \end{equation*}	
  with the constant $K_2$ from Lemma \ref{lem:induction2-2}.

 For the start of the induction, we need $\mathrm{Hyp2}_{\Lambda_0}(M_2)$ to be true.
 Since
 \begin{align*}
 	\sup_{0\le t\le T_0}
 	\sup_{\zeta\in\R^d\setminus\{0\}} \sup_{(z,\rho)\in A_{\Lambda_0}}  &\int_{\S^{d-2}(\zeta)} G\left(t,z\tfrac{\zeta}{|\zeta|} - \rho \omega \right)^{\epsilon(\alpha, 1)} \left|\hat{f}\left(t,z\tfrac{\zeta}{|\zeta|} - \rho \omega \right)\right|\,\mathrm{d}\omega \\
 	&\le |\S^{d-2}|e^{\beta T_0 (1+\Lambda_0^2)} A_m
  \end{align*}
  and from our choice of $M_2$ there exists $\beta_0 >0$ such that $\mathrm{Hyp2}_{\Lambda_0}(M_2)$ is true for all $0\le\beta\le \beta_0$.

  Now, we choose
  \begin{equation*}
  	\beta=\min \left( \beta_0, T_0^{-1},\frac{\tilde{C}_{f_0}}{(1+2^{d-1}) c_{b,d,2} \, \alpha T_0 M_2 +1} \right).
  \end{equation*}
  With this choice, the conditions of Lemma \ref{lem:induction1-2} and \ref{lem:induction2-2} are fulfilled and $\mathrm{Hyp2}_{\Lambda_0}(M_2)$ is true.

  For the induction step assume that $\mathrm{Hyp2}_{\Lambda_N}(M_2)$ is true. Then Lemma \ref{lem:induction1-2} gives
  \begin{equation*}
  	\| G_{\sqrt{2}\Lambda_N} f\|_{L^2(\R^d)} \le \|\1_{\sqrt{2}\Lambda}(D_v) f_0\|_{L^2(\R^d)} \, e^{C_{f_0}T_0} = B
  \end{equation*}
  and then, since $\epsilon(\alpha,1)\le \frac{2m}{2m+2}$ by our choice of $\alpha$, and
  $\Lambda_{N+1}= \widetilde{\Lambda}_N$, Lemma \ref{lem:induction2-2} shows that  $\mathrm{Hyp2}_{\Lambda_{N+1}}(M_2)$ is true, so by induction, it is true for all $N\in\N$.
  Invoking Lemma \ref{lem:induction1-2} again, we also have
  \begin{equation*}
  	\|G_{\sqrt{2}\Lambda_N}f\|_{L^2(\R^d)} \leq B
  \end{equation*}
  for all $N\in\N$ and letting $N\to\infty$, we see
  $\|Gf\|_{L^2(\R^d)} \leq B $, which concludes the proof of Theorem \ref{thm:gevrey2}.
\end{proof}

\begin{proof}[Proof of Corollary \ref{cor:gevrey2}]
Theorem \ref{thm:gevrey2} shows that $Gf\in L^{2}(\R^d)$ for all $0\leq t\leq T_0$. applying Corollary \ref{cor:ddlemma} with $n=d$ to $\hat{f}$ yields
\begin{align*}
	|\hat{f}(\eta)| \leq L_{m,d} G(\eta)^{-\frac{2m}{2m+d}} \left( \int_{Q_{\eta}} G(\zeta)^2 |\hat{f}(\zeta)|^2\,\mathrm{d}\zeta \right)^{\frac{m}{2m+d}} \leq L_{m,d} \|Gf\|_{L^2(\R^d)}^{\frac{2m}{2m+d}}\, G(\eta)^{-\frac{2m}{2m+d}},
\end{align*}
where we also used that the Fourier multiplier is radially increasing. This proves the uniform bound \eqref{eq:cor-gevrey2} with $\widetilde{\beta} = \beta \frac{2m}{2m+d}$.
\end{proof}

\subsection{Gevrey smoothing of weak solutions for $L^2$ initial data: Part III}\label{subsec:partiii}

Under the slightly stronger assumption on the angular collision cross-section $b$, namely that $b$ is bounded away from the singularity, we can state out theorem about Gevrey regularisation in its strongest form.

\begin{theorem}\label{thm:gevrey3}
	Assume that the initial datum $f_0$ satisfies $f_0\geq 0$, $f_0 \in L\log L(\R^d) \cap L^1_{m}(\R^d)$ for some $m\geq 2$, and, in addition, $f_0\in L^2(\R^d)$.
	Further assume that the cross-section $b$ in dimensions $d\geq 2$ satisfies the \emph{singularity condition} \eqref{eq:cross-section} for some $0<\nu<1$ and the \emph{boundedness condition} \eqref{eq:cross-section-bdd}. Let $f$ be a weak solution of the Cauchy problem \eqref{eq:cauchyproblem} with initial datum $f_0$, then for all $0<\alpha\leq \min\left\{\alpha^{m,1}, \nu\right\}$ and all $T_0>0$, there exists $\beta>0$, such that for all $t\in[0, T_0]$ 	
 	\begin{align}\label{eq:local-time-uniform-3}
		e^{\beta t \langle D_v\rangle^{2\alpha}} f(t,\cdot) \in L^2(\R^d),
	\end{align}
	that is, $f\in G^{\tfrac{1}{2\alpha}}(\R^d)$ for all $t\in(0,T_0]$.

	In particular, the weak solution is real analytic if $\nu=\frac{1}{2}$ and ultra-analytic if $\nu>\frac{1}{2}$.
\end{theorem}

\begin{remark} Thus, under slightly stronger assumption on $b$ than in Theorem \ref{thm:gevrey1}, which we stress are nevertheless fulfilled in any physically reasonable cases, we can prove the same regularity in \emph{any dimension} as can be obtained for radially symmetric  solutions of the homogenous Boltzmann equation.
\end{remark}
\begin{corollary}\label{cor:gevrey3}
	Under the same assumptions as in Theorem \ref{thm:gevrey3}, for any weak solution
	$f$ of the Cauchy problem \eqref{eq:cauchyproblem} and any $0<T_0<\infty$ there exists $\beta>0$ and $M<\infty$ such that
	\begin{align}\label{eq:cor-gevrey3}
		\sup_{0\le t\le T_0}\sup_{\eta\in\R^d} e^{ \beta t \langle \eta \rangle^{2\alpha}} \, |\hat{f}(t,\eta)| \le M .
	\end{align}
\end{corollary}
\begin{proof}
	Given Theorem \ref{thm:gevrey3}, the proof of Corollary \ref{cor:gevrey3} is the same as the proof of Corollary \ref{cor:gevrey2}.
\end{proof}

The proof of Theorem \ref{thm:gevrey3} shows the delicate interplay between the angular singularity of the collision kernel, the strict concavity of the Gevrey weights, and the use of averages of the weak solution in Fourier space, together with our inductive procedure, which has proved to be successful in Theorems \ref{thm:gevrey1} and \ref{thm:gevrey2}.
Again, the main work is to bound the expressions $I_{d,\Lambda}$ and $I_{d,\Lambda}^+$ from Lemma \ref{lem:ce}.
Before we start the proof of Theorem \ref{thm:gevrey3}, we start with some preparations. It is clear that we only have to prove Theorem \ref{thm:gevrey3} in dimension $d\ge 2$ and for singularities $\nu >\alpha_{2,m}$, since otherwise the result is already contained in Theorems \ref{thm:gevrey1} and \ref{thm:gevrey2}.

Looking at the integral $I_{d,\Lambda}$ from Lemma \ref{lem:ce}, one has
\begin{align*}
		I_{d,\Lambda} &=  \alpha \beta t \int_{\R^d}\left(
			\int_{0}^{\tfrac{\pi}{2}}  \int_{\S^{d-2}(\eta)} \sin^{d}\theta b(\cos\theta)\, G(\eta^-)^{\epsilon\left(\alpha, \cot^2\tfrac{\theta}{2}\right)}\,|\hat{f}(\eta^-)| \, \1_{\tfrac{\Lambda}{\sqrt{2}}}(|\eta^-|) \,\mathrm{d}\omega\,\mathrm{d}\theta
			\right)\\
		&\qquad \qquad \qquad \qquad \qquad \qquad \times |G_{\Lambda}(\eta) \hat{f}(\eta)|^2 \,\langle \eta \rangle^{2\alpha}\,\mathrm{d}\eta.
\end{align*}
where we use the parametrization \eqref{eq:etapar} for $\eta^-=\eta^-(\eta,\theta,\omega)$.
Splitting the $\theta$ integral above at a point $\theta_0\in (0,\tfrac{\pi}{2})$ and using the monotonicity of the cotangent on $[0,\tfrac{\pi}{2}]$ and of $\epsilon(\alpha,\gamma)$ in $\gamma$ one sees
\begin{align*}
	I_{d,\Lambda} \le I_{d,\Lambda,1}  + I_{d,\Lambda,2}
\end{align*}
whith
\begin{align}
	I_{d,\Lambda,1}  &:= \alpha \beta T_0
						\sup_{0<\theta\le \tfrac{\pi}{2}} \sup_{0<|\eta|\le \Lambda}
						\int_{\S^{d-2}(\eta)} G(\eta^-(\eta,\theta,\omega))^{\epsilon\left(\alpha, \cot^2\tfrac{\theta_0}{2}\right)}\,|\hat{f}(\eta^-(\eta,\theta,\omega))| \, \1_{\tfrac{\Lambda}{\sqrt{2}}}(|\eta^-(\eta,\theta,\omega)|) \,\mathrm{d}\omega \nonumber \\
					&  \qquad\qquad\times \int_0^{\theta_0} \sin^{d}\theta \, b(\cos\theta)\, \mathrm{d}\theta\, \|G_\Lambda f\|_{H^\alpha(\R^d)}^2
					\label{eq:Isplit-1}
	\intertext{and}
	I_{d,\Lambda,2}  &:= C_{\theta_0}\alpha \beta T_0
						\sup_{0<|\eta|\le \Lambda}
						\int_{\theta_0}^{\tfrac{\pi}{2}}\int_{\S^{d-2}(\eta)} G(\eta^-(\eta,\theta,\omega))^{\epsilon\left(\alpha, 1\right)}\,|\hat{f}(\eta^-(\eta,\theta,\omega))| \, \1_{\tfrac{\Lambda}{\sqrt{2}}}(|\eta^-(\eta,\theta,\omega)|) \,\mathrm{d}\omega \, \mathrm{d}\theta
						\nonumber\\
					&  \qquad\times \|G_\Lambda f\|_{H^\alpha(\R^d)}^2
					\label{eq:Isplit-2}
\end{align}
where $C_{\theta_0}$ is an upper bound for $b(\cos\theta)$ on $[\theta_0,\tfrac{\pi}{2}]$.
Now we choose $\theta_0>0$ so small that
\begin{align*}
	\epsilon(\alpha,\cot^2\frac{\theta_0}{2}) \le \epsilon(\alpha_{2,m},1) = \frac{2m}{2m+2}
\end{align*}
and note that from Corollary \ref{cor:gevrey2}, since $\nu>\alpha_{2,m}$, there exists a finite $M_2$ such that
\begin{align*}
	\sup_{0<\theta\le \tfrac{\pi}{2}} \sup_{0<|\eta|\le \Lambda}
						\int_{\S^{d-2}(\eta)} G(\eta^-(\eta,\theta,\omega))^{\epsilon\left(\alpha_{2,m}, 1\right)}\,|\hat{f}(\eta^-(\eta,\theta,\omega))| \, \1_{\tfrac{\Lambda}{\sqrt{2}}}(|\eta^-(\eta,\theta,\omega)|) \,\mathrm{d}\omega
						\le M_2<\infty .
\end{align*}
So from \eqref{eq:Isplit-1} we get the bound
\begin{align}\label{eq:Isplit-1-bound}
	I_{d,\Lambda,1} &\le \alpha \beta T_0   M_2 c_{b,d,2} \|G_\Lambda f\|_{H^\alpha(\R^d)}^2
\end{align}
where the finiteness of $c_{b,d,2}$ follows from the singularity condition and the boundedness of $b(\cos\theta)$ away from $\theta=0$.

For the integral $I_{d,\Lambda}^+$ from Lemma \ref{lem:ce}, a completely analogous reasoning as above shows for small enough $\vartheta_0$ such that $\epsilon(\alpha\cot^\vartheta)\le \epsilon(\alpha_{2,m},1)$ we also have
\begin{align*}
	I_{d,\Lambda}^+\le I_{d,\Lambda,1}^+ +I_{d,\Lambda,2}^+
\end{align*}
with
\begin{align}\label{eq:Iplus-split-1-bound}
	I_{d,\Lambda,1}^+ \le 2^{d-1} \alpha \beta T_0   M_2 c_{b,d,2} \|G_\Lambda f\|_{H^\alpha(\R^d)}^2
\end{align}
and
\begin{align}\label{eq:Iplus-split-2}
	I_{d,\Lambda,2}^+  &:= 2^d C_{\vartheta_0}\alpha \beta T_0
						\sup_{0<|\eta^+|\le \Lambda}
						\int_{\vartheta_0}^{\tfrac{\pi}{4}}\int_{\S^{d-2}(\eta^+)} G(\eta^-(\eta^+,\vartheta,\omega))^{\epsilon\left(\alpha, 1\right)}\,|\hat{f}(\eta^-(\eta^+,\vartheta,\omega))| \, \1_{\tfrac{\Lambda}{\sqrt{2}}}(|\eta^-(\eta^+,\vartheta,\omega)|) \,\mathrm{d}\omega \, \mathrm{d}\vartheta
						\nonumber\\
					&  \qquad\times \|G_\Lambda f\|_{H^\alpha(\R^d)}^2
\end{align}
where we use the parametrization \eqref{eq:etapluspar} for $\eta^-=\eta^-(\eta^+,\vartheta,\omega)$ and
where $C_{\vartheta_0}$ is an upper bound for $b(\cos(2\vartheta))$ on $[\vartheta_0,\tfrac{\pi}{4}]$.

Recall that we always assume $\alpha\le \alpha_{1,m}$, so $\epsilon(\alpha,1)\le \epsilon(\alpha_{1,m},1)= \frac{2m}{2m+1}$.
Thus we see that in order to set up our inductive procedure for controlling $I_{d\Lambda}$ and $I_{d,\Lambda}^+$ it is natural to introduce
\begin{definition}[Hypothesis $\mathbf{Hyp3}_{\Lambda}(M)$] \label{def:Hyp_Lambda3}
    Let $M\ge 0$ be finite, $0<\theta_0,\vartheta_0<\tfrac{\pi}{4}$, $T_0>0$, and $m\ge 2$ an integer. Then for all $0\le t\le T_0$ one has
    \begin{align}
   \sup_{|\eta|\le \sqrt{2}\Lambda}\int_{\theta_0}^{\frac{\pi}{2}} \int_{\S^{d-2}(\eta)} G\left(t,\eta^-(\eta,\theta,\omega)\right)^{\frac{2m}{2m+1}} \left|\hat{f}\left(\eta^-(\eta,\theta,\omega)\right)\right|\1_{\Lambda}(|\eta^-(\eta,\theta,\omega)|)\,\mathrm{d}\omega\,\mathrm{d}\theta &\leq M ,
   \label{eq:Hyp_Lambda3-1}\\
   \intertext{where we use the parametrization given in \eqref{eq:etapar} for $\eta^-$, and }
     \sup_{|\eta^+|\le \sqrt{2}\Lambda}\int_{\vartheta_0}^{\frac{\pi}{4}} \int_{\S^{d-2}(\eta^+)} G\left(t,\eta^-(\eta^+,\vartheta,\omega)\right)^{\frac{2m}{2m+1}} \left|\hat{f}\left(\eta^-(\eta^+,\vartheta,\omega)\right)\right|\1_{\Lambda}(|\eta^-(\eta^+,\vartheta,\omega)|)\,\mathrm{d}\omega\,\mathrm{d}\vartheta &\leq M
     \label{eq:Hyp_Lambda3-2}
     \intertext{where we use the parametrization given in \eqref{eq:etapluspar} for $\eta^-$.}
    \end{align}
\end{definition}

For the induction proof of Theorem \ref{thm:gevrey3}, we again start with

\begin{lemma}\label{lem:induction1-3} Let $M\ge 0$, $T_0>0$, $m\ge 2$ an integer, $\alpha_{m,2}<\nu<1$,
	$0<\alpha \leq \nu$ and recall $c_{b,d,2} = \int_{0}^{\frac{\pi}{2}} \sin^d \theta b(\cos\theta)\,\mathrm{d}\theta$ (which is finite by the singularity assumption \eqref{eq:cross-section2} and the boundedness assumption \eqref{eq:cross-section-bdd}).
	Let $M_2$ be from Corollary \ref{cor:gevrey2} and
 $\beta\le \frac{\tilde{C}_{f_0}}{\alpha T_0[(1+2^{d-1}) c_{b,d,2} M_2 + (C_{\theta_0}+2^d C_{\vartheta_0})M]  +1}$. Then for any weak solution of the homogenous Boltzmann equation,
	\begin{equation} \label{eq:inductionbound1-3}
		\mathbf{Hyp3}_{\Lambda}(M) \quad \Rightarrow \quad \|G_{\sqrt{2}\Lambda} f \|_{L^2(\R^d)} \le \|\1_{\sqrt{2}\Lambda}(D_v) f_0\|_{L^2(\R^d)} \, e^{C_{f_0}T_0}
	\end{equation}
	for all $0\le t\le T_0$.
\end{lemma}
\begin{proof}
	Given Lemma \ref{lem:ce} and the above discussion with the bounds in \eqref{eq:Isplit-1-bound}, \eqref{eq:Iplus-split-1-bound} and using the hypotheses $(\mathrm{Hyp}3_{\Lambda})$ for the terms in \eqref{eq:Isplit-2} and \eqref{eq:Iplus-split-2}, one sees that the commutation error on the level $\sqrt{2}\Lambda$ is bounded by
	\begin{align*}
		&\left| \left\langle Q(f, G_{\sqrt{2}\Lambda}f) - G_{\sqrt{2}\Lambda} Q(f,f), G_{\sqrt{2}\Lambda}f\right\rangle\right| \leq I_{d,\sqrt{2}\Lambda} + I_{d, \sqrt{2}\Lambda}^+ \\
		&\le  (1+2^{d-1})\alpha \beta T_0   M_2 c_{b,d,2} \|G_\Lambda f\|_{H^\alpha(\R^d)}^2
			+ (C_{\theta_0}+ 2^d C_{\vartheta_0}) \alpha \beta T_0  M \|G_\Lambda f\|_{H^\alpha(\R^d)}^2  .
	\end{align*}
	Given this bound on the commutation error, the rest of the proof is the same as in the proof of Lemma \ref{lem:induction1}.
\end{proof}

To close the induction step we also need a suitable version of Lemma \ref{lem:induction2-2} but before we prove this we need a preparatory Lemma.

\begin{lemma}\label{lem:change of variables}
Let $H:\R^d\to \R_+$ be a locally integrable function and let $\eta,\eta_+\in\R^d$ with $|\eta|,|\eta^+|\ge \Lambda_0>0$, $0<\theta_0 \le \tfrac{\pi}{2}$, and $0<\vartheta_0 \le \tfrac{\pi}{8}$.
Then with the parametrization $\eta^-=\eta^-(\eta,\theta,\omega)$ given in \eqref{eq:etapar} one has
\begin{align*}
	\int_{\theta_0}^{\tfrac{\pi}{2}} \int_0^2 H\left(\eta^-(\eta,\theta,\omega)+z\tfrac{\eta}{|\eta|}\right) \,\mathrm{d} z\, \mathrm{d}\theta
	\le \frac{2}{\Lambda_0\cos\theta_0}
		\int_{\Lambda_0\sin^2\tfrac{\theta_0}{2}}^{\tfrac{|\eta|}{2}+2}
		\int_{\Lambda_0\sin\theta_0}^{\tfrac{|\eta|}{2}}
			H\left(x\tfrac{\eta}{|\eta|} - y\omega\right)
			\,\mathrm{d} y\, \mathrm{d}x
\end{align*}
for any unit vector $\omega$ orthogonal to $\eta$.
Moreover, with the parametrization $\eta^-=\eta^-(\eta^+,\theta,\omega)$ given in \eqref{eq:etapluspar} one has, for any $\widetilde{\Lambda}\ge\frac{1+\sqrt{2}}{2}\Lambda_0$,
\begin{align*}
	\int_{\vartheta_0}^{\tfrac{\pi}{4}} \int_0^2 H\left(\eta^-(\eta^+,\vartheta,\omega)+z\tfrac{\eta}{|\eta|}\right) & \1_{\tfrac{\widetilde{\Lambda}}{\sqrt{2}}}(|\eta^-(\eta^+,\vartheta,\omega)|)\,\mathrm{d} z\, \mathrm{d}\vartheta \\
	&\le
	\frac{1}{2\Lambda_0}
		\int_{0}^{2}
		\int_{\Lambda_0\tan\vartheta_0}^{\tfrac{\widetilde{\Lambda}}{\sqrt{2}}}
			H\left(x\tfrac{\eta}{|\eta|} - y\omega\right)
			\,\mathrm{d} y\, \mathrm{d}x
\end{align*}
\begin{remark}
	The restriction $\vartheta_0\le \tfrac{\pi}{8}$ is only for convenience, to ensure that
	$\Lambda_0\tan\vartheta_0\le \tfrac{\widetilde{\Lambda}}{\sqrt{2}}$.
\end{remark}

\begin{proof}
	Fix $\eta$ as required and $\omega$ orthogonal to it. We want to have a map $\Phi_1:(\theta,z)\mapsto \Phi_1(\theta,z)=(x,y)$ such that
	\begin{align*}
		\eta^-(\eta,\theta,\omega)+z\tfrac{\eta}{|\eta|}
		=
		 x\tfrac{\eta}{|\eta|} - y\omega .
	\end{align*}
	From the parametrization \eqref{eq:etapar} we read off
	\begin{align*}
		x = |\eta|\sin^2\frac{\theta}{2} + z
		\quad\text{and }
		y = \frac{|\eta|}{2}\sin\theta
	\end{align*}
	and we can compute the Jacobian going from the $(\theta,z)$ variables to $(x,y)$ as
	\begin{align*}
		\left|\frac{\partial(x,y)}{\partial(\theta,z)}\right| = |\det D\Phi_1| = \frac{|\eta|}{2}\cos\theta\ge \frac{|\eta|}{2}\cos\theta_0 .
	\end{align*}
	Since $|\eta|\ge \Lambda_0$, $\theta\in [\theta_0,\tfrac{\pi}{2}]$, and $0\le z\le 2$, we have
	$\Lambda_0\sin^2\tfrac{\theta_0}{2}\le x\le |\eta|\sin^2\tfrac{\pi}{4} = \tfrac{\eta}{2} $
	and $\frac{\Lambda_0}{2}\sin\theta_0 \le y\le \tfrac{\eta}{2}$.
	So doing a change of variables $(\theta,z)= \Phi_1^{-1}(x,y)$ in the integral we can bound
	\begin{align*}
		\int_{\theta_0}^{\tfrac{\pi}{2}} \int_0^2 H\left(\eta^-(\eta,\theta,\omega)+z\tfrac{\eta}{|\eta|}\right) \,\mathrm{d} z\, \mathrm{d}\theta
	\le \frac{2}{\Lambda_0\cos\theta_0}
		\int_{\Lambda_0\sin^2\tfrac{\theta_0}{2}}^{\tfrac{|\eta|}{2}+2}
		\int_{\Lambda_0\sin\theta_0}^{\tfrac{|\eta|}{2}}
			H\left(x\tfrac{\eta}{|\eta|} + y\omega\right)
			\,\mathrm{d} y\, \mathrm{d}x
	\end{align*}
	since the map $\Phi_1$ is a nice diffeomorphism.
	
	For the second bound the calculation is, in fact, a bit easier, one just has to take care that $|\eta^-|$ cannot be too large, which is taken into account by the factor
	$\1_\Lambda(|\eta^-|)$.  We now want a map $\Phi_2:(\theta,z)\mapsto \Phi_2(\theta,z)=(x,y)$ such that
	\begin{align*}
		\eta^-(\eta^+,\vartheta,\omega)+z\tfrac{\eta^+}{|\eta^+|}
		=
		 x\tfrac{\eta^+}{|\eta^+|} - y\omega .
	\end{align*}
	From the parametrization \eqref{eq:etapar} we read off
	\begin{align*}
		x =  z
		\quad\text{and }
		y = |\eta^-|=|\eta^+|\tan\vartheta
	\end{align*}
 and the Jacobian going from the $(\vartheta,z)$ variables to $(x,y)$ is simply
	\begin{align*}
		\left|\frac{\partial(x,y)}{\partial(\vartheta,z)}\right| = |\det D\Phi_2|
		= 2|\eta^+| \ge 2\Lambda_0.
	\end{align*}
	We certainly have $0\le x\le 2$ and also $\Lambda_0\tan\vartheta_0\le y$. Since
	$y= |\eta^-|$, we also have the restriction $y\le \Lambda$. So the proof of the second inequality follows similar to the proof of first one.
\end{proof}

Finally, we can state and prove the second step in our inductive procedure.

\end{lemma}
\begin{lemma}\label{lem:induction2-3}
	    	Let $\beta\leq \frac{1}{T_0}$.  Asssume that there exist finite constants $A_m$ and $B$, such that
    	\begin{equation}\label{eq:ind2-assumption1-3}
    	\|f(t, \cdot)\|_{L^1_m} \leq A_m , \quad  \text{and } \quad 	\| (G_{\sqrt{2}\Lambda} f)(t, \cdot) \|_{L^2(\R^d)} \le B
    	\end{equation}
    	for some integer $m\geq 2$ and for all $0\le t\le T_0$.
    	
	Set $\widetilde{\Lambda}:= \frac{1+\sqrt{2}}{2} \Lambda$
		and assume that
	\begin{align}\label{eq:ind2-3-assumption2}
		\Lambda\ge \Lambda_0 := 3 .
	\end{align}
Then there exist a finite $K_3$, depending only on $d,m,A_m$, and $B$ such that  $\mathrm{Hyp3}_{\widetilde{\Lambda}}(K_3)$ is true.
\end{lemma}
\begin{proof}
 	Fix $0<t\le T_0$, a direction $\eta\in \R^d\setminus\{0\}$, and define the function
	\begin{align*}
		z\mapsto F(z):= \hat{f}(t,\eta^- + z\tfrac{\eta}{|\eta|})
	\end{align*}
	of the single real variable $z$, where we think of $\eta^-$ as given in the $\eta$-parametrization \eqref{eq:etapar} for some $\theta$ and $\omega\in \S^{d-2}(\eta)$, and where we drop, for simplicity, the dependence on the time $t$ in our notation for $F$ and $f$.
	Then, since $\|f(t,\cdot)\|_{L^1_m}\le A_m$ one has $\hat{f}(t,\cdot)\in \mathcal{C}^m(\R^d)$  and thus also  $F\in \mathcal{C}^m(\R)$ with $\|F\|_{L^{\infty}} \leq A_m$ $\|\partial^m_z F\|_{L^{\infty}} \leq (2\pi)^m A_m$,  and Corollary \ref{cor:ddlemma} applied to $F$ now gives
	\begin{align*}
		|\hat{f}(\eta^-)| \le L_{m,1} \left( \int_0^2|\hat{f}(\eta^- + z\tfrac{\eta}{|\eta|})|^2\, \mathrm{d}z\right)^{\frac{m}{2m+2}} .
	\end{align*}
	We multiply  this with the radially increasing weight $G$ to get
	\begin{align*}
		G(\eta^-)^{\frac{2m}{2m+1}} |\hat{f}(\eta^-)|
		\leq
		  L_{m,1} \left( \int_0^2|G(\eta^- +z\tfrac{\eta}{|\eta|})\hat{f}(\eta^- + z\tfrac{\eta}{|\eta|})|^2\, \mathrm{d}z\right)^{\frac{m}{2m+2}} .
	\end{align*}
	Integrating this with respect to $\omega$ and $\theta$, where we think of $\eta^-= \eta^-(\eta,\theta,\omega)$ in the parametrization \eqref{eq:etapar}, and using Jensen's inequality for concave functions, one gets
	\begin{align}
		\int_{\theta_0}^{\tfrac{\pi}{2}} &\int_{\S^{d-2}(\eta)}  G(\eta^-)^{\frac{2m}{2m+1}} |\hat{f}(\eta^-)| \, \mathrm{d}\theta\, \mathrm{d}\omega \nonumber \\
		& \leq
		 L_{m,1} (\tfrac{\pi}{2})^{\frac{m+1}{2m+1}} |\S^{d-2}|^{\frac{m+1}{2m+1}}
		 \left( \int_{\theta_0}^{\tfrac{\pi}{2}} \int_{\S^{d-2}(\eta)} \int_0^2|G(\eta^- +z\tfrac{\eta}{|\eta|})\hat{f}(\eta^- + z\tfrac{\eta}{|\eta|})|^2\,
		 \mathrm{d}z \, \mathrm{d}\theta\, \mathrm{d}\omega \right)^{\frac{m}{2m+1}}
		 \label{eq:step1} .
	\end{align}
	Now assume that $|\eta|\ge \Lambda_0$. Because of the first part of Lemma \ref{lem:change of variables}, we can further bound
	\begin{align*}
		\eqref{eq:step1} &\le
		L_{m,1} (\tfrac{\pi}{2})^{\frac{m+1}{2m+1}} |\S^{d-2}|^{\frac{m+1}{2m+1}}
		\left(\frac{2}{\Lambda_0\cos\theta_0} \right)^{\frac{m}{2m+1}} \\
		 &\qquad\quad
		\left( \int_{\S^{d-2}(\eta)} \int_{\Lambda_0\sin^2\tfrac{\theta_0}{2}}^{\tfrac{|\eta|}{2}+2}
		\int_{\Lambda_0\sin\theta_0}^{\tfrac{|\eta|}{2}}
			| G(x\tfrac{\eta}{|\eta|}-y\omega)\hat{f}(x\tfrac{\eta}{|\eta|}-y\omega)|^2\,
		 \mathrm{d}y \, \mathrm{d}x\, \mathrm{d}\omega \right)^{\frac{m}{2m+1}} \\
		 & \le
		L_{m,1} (\tfrac{\pi}{2})^{\frac{m+1}{2m+1}} |\S^{d-2}|^{\frac{m+1}{2m+1}}
		\left(\frac{2}{\Lambda_0\cos\theta_0} \right)^{\frac{m}{2m+1}}
		(\Lambda_0\sin\theta_0)^{2-d} \\
		 &\qquad\quad
		\left( \int_{\S^{d-2}(\eta)} \int_{\Lambda_0\sin^2\tfrac{\theta_0}{2}}^{\tfrac{|\eta|}{2}+2}
		\int_{\Lambda_0\sin\theta_0}^{\tfrac{|\eta|}{2}}
			| G(x\tfrac{\eta}{|\eta|}-y\omega)\hat{f}(x\tfrac{\eta}{|\eta|}-y\omega)|^2
			\, y^{d-2}\mathrm{d}y\, \mathrm{d}x\, \mathrm{d}\omega \right)^{\frac{m}{2m+1}}
	\end{align*}
	Again, the integration measure $ y^{d-2}\mathrm{d}y\, \mathrm{d}x\, \mathrm{d}\omega$ is $d$-dimensional Lebesgue measure in the cylindrical coordinates  $(x,y\omega)$ with respect to the cylinder in the $\eta$ direction. One checks that the condition $\Lambda\ge\Lambda_0\ge 3$ ensures that
	\begin{align*}
		(\widetilde{\Lambda}/2+2)^2 +(\widetilde{\Lambda}/2) \le (\sqrt{2}\Lambda)^2
	\end{align*}
	so since $|\eta|\leq \widetilde{\Lambda}$, we can extend the integration above to a ball of radius $\sqrt{2}\Lambda$ to get
	\begin{align}
		\eqref{eq:step1}
		&\le
		L_{m,1} (\tfrac{\pi}{2})^{\frac{m+1}{2m+1}} |\S^{d-2}|^{\frac{m+1}{2m+1}}
		\left(\frac{2}{\Lambda_0\cos\theta_0} \right)^{\frac{m}{2m+1}}
		(\Lambda_0\sin\theta_0)^{2-d}
			\| G_{\sqrt{2}\Lambda}f \|_{L^2(\R^d)}^{\frac{2m}{2m+1}} \nonumber\\
		&\le
			 L_{m,1} (\tfrac{\pi}{2})^{\frac{m+1}{2m+1}} |\S^{d-2}|^{\frac{m+1}{2m+1}}
		\left(\frac{2}{\Lambda_0\cos\theta_0} \right)^{\frac{m}{2m+1}}
		(\Lambda_0\sin\theta_0)^{2-d} B^{\frac{2m}{2m+1}} .
			\label{eq:step2}
	\end{align}
 If $|\eta|\le \Lambda_0$ we simply bound
 \begin{align}\label{eq:step3}
 	\int_{\theta_0}^{\tfrac{\pi}{2}} &\int_{\S^{d-2}(\eta)}  G(\eta^-)^{\frac{2m}{2m+1}} |\hat{f}(\eta^-)| \, \mathrm{d}\theta\, \mathrm{d}\omega
 	\le \|\hat{f}\|_{L^\infty} \frac{\pi}{2} |\S^{d-2}| e^{\beta T_0(1+\Lambda_0^2/2)}
 	\le A_m \frac{\pi}{2}|\S^{d-2}| e^{1+\Lambda_0^2/2}.
 \end{align}
 Concerning the bound in the second half of $\mathrm{Hyp}_{\widetilde{\Lambda}}$, a completely analogous calculation as the one above, using the second halft of Lemma \ref{lem:change of variables} gives for $\lambda_0\le |\eta^+|\le \widetilde{\Lambda}$,

 \begin{align}
 	&\int_{\vartheta_0}^{\frac{\pi}{2}} \int_{\S^{d-2}(\eta^+)} G\left(t,\eta^-(\eta^+,\vartheta,\omega)\right)^{\frac{2m}{2m+1}} \left|\hat{f}\left(\eta^-(\eta^+,\vartheta,\omega)\right)\right|\1_{\tfrac{\Lambda}{\sqrt{2}}}(|\eta^-(\eta^+,\vartheta,\omega)|)\,\mathrm{d}\omega\,\mathrm{d}\vartheta \nonumber\\
	&\le
	 L_{m,1} (\tfrac{\pi}{2})^{\frac{m+1}{2m+1}} |\S^{d-2}|^{\frac{m+1}{2m+1}}
		\left(\frac{1}{2\Lambda_0} \right)^{\frac{m}{2m+1}} (\Lambda_0\tan\vartheta_0)^{2-d}
		\nonumber\\
		 &\qquad\quad
		\left( \int_{\S^{d-2}(\eta^+)} \int_{0}^{2}
		\int_{0}^{\frac{\widetilde{\Lambda}}{\sqrt{2}}}
			| G(x\tfrac{\eta}{|\eta|}-y\omega)\hat{f}(x\tfrac{\eta}{|\eta|}-y\omega)|^2
			\, y^{d-2}\mathrm{d}y\, \mathrm{d}x\, \mathrm{d}\omega \right)^{\frac{m}{2m+1}}	
			\label{eq:step4}	
 \end{align}
 By our choice of $\widetilde{\Lambda}$ and $\Lambda_0$, we always have $2^2+ (\widetilde{\Lambda}/2)^2\le (\sqrt{2}\Lambda)^2$, so we can extend the integration above to the whole ball $|\eta^+|\le \sqrt{2}\Lambda$ to see
  \begin{align}
  		\eqref{eq:step4}
  		&	\le
		L_{m,1} (\tfrac{\pi}{2})^{\frac{m+1}{2m+1}} |\S^{d-2}|^{\frac{m+1}{2m+1}}
		\left(\frac{1}{2\Lambda_0} \right)^{\frac{m}{2m+1}} (\Lambda_0\tan\vartheta_0)^{2-d}
		\| G_{\sqrt{2}\Lambda}f \|_{L^2(\R^d)}^{\frac{2m}{2m+1}} \nonumber\\
		&	\le
		L_{m,1} (\tfrac{\pi}{2})^{\frac{m+1}{2m+1}} |\S^{d-2}|^{\frac{m+1}{2m+1}}
		\left(\frac{1}{2\Lambda_0} \right)^{\frac{m}{2m+1}} (\Lambda_0\tan\vartheta_0)^{2-d}
		B^{\frac{2m}{2m+1}}
		\label{eq:step5}
  \end{align}
  If $|\eta^+|\le \Lambda_0$ we simply bound as above
 \begin{align}\label{eq:step6}
 	\int_{\vartheta_0}^{\tfrac{\pi}{4}} \int_{\S^{d-2}(\eta^+)}  G(\eta^-)^{\frac{2m}{2m+1}} |\hat{f}(\eta^-)| \, \mathrm{d}\vartheta\, \mathrm{d}\omega
  	\le A_m \frac{\pi}{4}|\S^{d-2}| e^{1+\Lambda_0^2}.
 \end{align}
 Now we set $K_3$ equal to the maximum of the constants in \eqref{eq:step2}, \eqref{eq:step3}, \eqref{eq:step5}, \eqref{eq:step6}. with this choice, $K_3$ depends only on $d,m,A_m$, and $B$ and   $\mathrm{Hyp3}_{\widetilde{\Lambda}}(K_3)$ is true.
\end{proof}

\begin{proof}[Proof of Theorem \ref{thm:gevrey3}]
By Lemmata \ref{lem:induction1-2} and \ref{lem:induction2-2} a suitable choice for $A_m$ and  $B$ is
 \begin{align*}
    B&:= \|f_0\|_{L^2(\R^d)} e^{C_{f_0}T_0}, \\
    A_m&:= \sup_{t\geq 0} \|f(t, \cdot)\|_{L^1_m(\R^d)} < \infty.
   \end{align*}
   Note that the finiteness of $A_m$ is guaranteed since $f_0 \in L^1_m(\R^d)$, see Remark \ref{rem:ind2}. Again choose the length scales $\Lambda_N$ to be
  \begin{equation*}
    \Lambda_N := \frac{\Lambda_{N-1}+\sqrt{2}\Lambda_{N-1}}{2} = \frac{1+\sqrt{2}}{2}\Lambda_{N-1} =
    \left( \frac{1+\sqrt{2}}{2} \right)^N \Lambda_0
  \end{equation*}
  with  $\Lambda_0=3$, see \eqref{eq:ind2-3-assumption2}, and we set
  \begin{equation*}
    M_3:= \max \left\{ 2 |\S^{d-2}|A_m+1, K_3 \right\}
  \end{equation*}	
  with the constant $K_3$ from Lemma \ref{lem:induction2-3}.

 For the start of the induction, we need $\mathrm{Hyp3}_{\Lambda_0}(M_3)$ to be true.
 Since
 \begin{align*}
 	\sup_{0\le t\le T_0}
 	\sup_{\zeta\in\R^d\setminus\{0\}} \sup_{(z,\rho)\in A_{\Lambda_0}}  &\int_{\S^{d-2}(\zeta)} G\left(t,z\tfrac{\zeta}{|\zeta|} - \rho \omega \right)^{\epsilon(\alpha, 1)} \left|\hat{f}\left(t,z\tfrac{\zeta}{|\zeta|} - \rho \omega \right)\right|\,\mathrm{d}\omega \\
 	&\le |\S^{d-2}|e^{\beta T_0 (1+\Lambda_0^2)} A_m
  \end{align*}
  and from our choice of $M_2$ there exists $\beta_0 >0$ such that $\mathrm{Hyp2}_{\Lambda_0}(M_2)$ is true for all $0\le\beta\le \beta_0$.

  Now, we choose
  \begin{equation*}
  	\beta=\min \left( \beta_0, T_0^{-1},\frac{\tilde{C}_{f_0}}{2^d c_{b,d,2} \, \alpha T_0 M_2 +1} \right).
  \end{equation*}
  With this choice, the conditions of Lemma \ref{lem:induction1-2} and \ref{lem:induction2-2} are fulfilled and $\mathrm{Hyp2}_{\Lambda_0}(M_2)$ is true.

  For the induction step assume that $\mathrm{Hyp2}_{\Lambda_N}(M_2)$ is true. Then Lemma \ref{lem:induction1-2} gives
  \begin{equation*}
  	\| G_{\sqrt{2}\Lambda_N} f\|_{L^2(\R^d)} \le \|\1_{\sqrt{2}\Lambda}(D_v) f_0\|_{L^2(\R^d)} \, e^{C_{f_0}T_0} = B
  \end{equation*}
  and then, since $\epsilon(\alpha,1)\le \frac{2m}{2m+2}$ by our choice of $\alpha$, and
  $\Lambda_{N+1}= \widetilde{\Lambda}_N$, Lemma \ref{lem:induction2-2} shows that  $\mathrm{Hyp2}_{\Lambda_{N+1}}(M_2)$ is true, so by induction, it is true for all $N\in\N$.
  Invoking Lemma \ref{lem:induction1-2} again, we also have
  \begin{equation*}
  	\|G_{\sqrt{2}\Lambda_N}f\|_{L^2(\R^d)} \leq B
  \end{equation*}
  for all $N\in\N$ and letting $N\to\infty$, we see
  $\|Gf\|_{L^2(\R^d)} \leq B $, which concludes the proof of Theorem \ref{thm:gevrey2}.
\end{proof}

\section[Gevrey regularity and (ultra-)analyticity of weak solutions]{Removing the $L^2$ constraint: Gevrey regularity and (ultra-)analyticity of weak solutions}\label{sec:great-results}
In this section we will give the proofs of Theorem \ref{thm:gevrey-main1}, \ref{thm:gevrey-main2}, and \ref{thm:gevrey-main3} in a slightly more general form. More precisely, we will prove

\begin{theorem}[Gevrey smoothing I]\label{thm:gevrey-main1-m}
		Assume that the cross-section $b$ satisfies the \emph{singularity condition} \eqref{eq:cross-section} and the \emph{integrability condition} \eqref{eq:cross-section2} for $d\geq 2$, and for $d=1$, $b_1$ satisfies the \emph{singularity condition} \eqref{eq:cross-section-kac} and the \emph{integrability condition} \eqref{eq:cross-section-kac2} for some $0<\nu<1$. Let $f$ be a weak solution of the Cauchy problem \eqref{eq:cauchyproblem} with initial datum $f_0\ge 0$ and $f_0\in L^1_m(\R^d)\cap L\log L(\R^d)$ for some integer $m\ge 2$. Then, for all $0<\alpha\leq \min\left\{\alpha_{m,d}, \nu\right\}$,
	\begin{align}
		f(t,\cdot)\in G^{\tfrac{1}{2\alpha}}(\R^d)
	\end{align}
	for all $t>0$, where $\alpha_{m,d} = \frac{\log[(4m+d)/(2m+d)]}{\log 2}$.
\end{theorem}

\begin{theorem}[Gevrey smoothing II]\label{thm:gevrey-main2-m}
	Let $d\ge 2$. Assume that the cross-section $b$ satisfies the conditions of Theorem \ref{thm:gevrey-main1}. Let $f$ be a weak solution of the Cauchy problem \eqref{eq:cauchyproblem} with initial datum $f_0\ge 0$ and $f_0\in L^1_m(\R^d)\cap L\log L(\R^d)$ for some integer $m\ge 2$. Then, for all $0<\alpha\leq \min\left\{\alpha_{m,2}, \nu\right\}$,
	\begin{align}
		f(t,\cdot)\in G^{\tfrac{1}{2\alpha}}(\R^d)
	\end{align}
	for all $t>0$, where $\alpha_{m,2} = \frac{\log[(4m+2)/(2m+2)]}{\log 2}$. In particular, the weak solution is real analytic if $\nu=\frac{1}{2}$ and ultra-analytic if $\nu>\frac{1}{2}$ in \emph{any dimension}.
\end{theorem}

If the integrability conditions \eqref{eq:cross-section2} is replaced by the slightly stronger condition \eqref{eq:cross-section-bdd}, which is true in all physically relevant cases, we can prove the  stronger result

\begin{theorem}[Gevrey smoothing III]\label{thm:gevrey-main3-m}
	Let $d\ge 2$. Assume that the cross-section $b$ satisfies the conditions of Theorem \ref{thm:gevrey-main1} and the condition \eqref{eq:cross-section-bdd},  that is, they are bounded away from the singularity.  Let $f$ be a weak solution of the Cauchy problem \eqref{eq:cauchyproblem} with initial datum $f_0\ge 0$ and $f_0\in L^1_m(\R^d)\cap L\log L(\R^d)$ for some integer $m\ge 2$. Then, for all $0<\alpha\leq \min\left\{\alpha_{m,1}, \nu\right\}$,
	\begin{align}
		f(t,\cdot)\in G^{\tfrac{1}{2\alpha}}(\R^d)
	\end{align}
	for all $t>0$, where $\alpha_{m,1} = \frac{\log[(4m+1)/(2m+1)]}{\log 2}$.
\end{theorem}

\begin{remark}
		
\end{remark}

We even have the uniform bound
\begin{corollary}\label{cor:gevrey-m}
	Under the same assumptions as in Theorem \ref{thm:gevrey-main1-m} (or \ref{thm:gevrey-main2-m}, respectively \ref{thm:gevrey-main3-m}), for any weak solution $f$ of the Cauchy problem \eqref{eq:cauchyproblem}
	initial datum $f_0\ge 0$ and $f_0\in L^1_m(\R^d)\cap L\log L(\R^d)$ for some integer $m\ge 2$
	and for any $0<\alpha\leq \min\{\alpha_{d,m}, \nu\}$ (or any $0<\alpha\leq \min\{\alpha_{m,2}, \nu\}$, respectively $0<\alpha\leq \min\{\alpha_{m,1}, \nu\}$) there exist constants $0<K,C<\infty$ such that
	\begin{align}\label{eq:Fourier-uniform}
		\sup_{0\le t<\infty}\sup_{\eta\in\R^d}
		e^{K \min(t,1) \, \langle \eta \rangle^{2\alpha}} \,|\hat{f}(t,\eta)|
		\leq C .
	\end{align}
\end{corollary}

\begin{proof}[Proof of Theorems \ref{thm:gevrey-main1-m} through \ref{thm:gevrey-main3-m}]
	In the case where the initial condition $f_0$ obeys $f_0\ge 0$ and $f_0\in L^1_m(\R^d)\cap L\log L(\R^d)$ for some integer $m\ge 2$, but
	is not necessarily in $L^2(\R^d)$, we use the known $H^\infty$ smoothing of
	the Boltzmann \cite{DW10, AE05, MUXY09} and Kac equation\footnote{A $H^{\infty}$ smoothing effect for the homogeneous non-cutoff Kac equation was first proved by \textsc{L. Desvillettes} \cite{Des95}, but under the stronger assumption that all polynomial moments of the initial datum $f_0$ are bounded, i.e. $f_0\in L^1_k(\R)\cap L\log L(\R)$ for all $k\in\N$.} \cite{LX09} in a mild way (see also Appendix \ref{sec:appendix-hinfty}):
	for $\tau>0$ one has $f(\tau,\cdot)\in L^2(\R^d)$
	and using this as a new initial condition in Theorems \ref{thm:gevrey-main1} through \ref{thm:gevrey-main3}, and noting that $T_0$ in those theorems is arbitrary, this implies that $f(t,\cdot) \in G^{\frac{1}{2\alpha}}(\R^d)$
	for $t>0$.
\end{proof}

\begin{proof}[Proof of Corollary \ref{cor:gevrey-m} ]
	Using known results about propagation of Gevrey regularity by \textsc{Desvillettes}, \textsc{Furioli}, and \textsc{Terraneo} \cite{DFT09} for the non-cutoff homogeneous Boltzmann and Kac equation for Maxwellian molecules, the bounds from Corollary \ref{cor:gevrey1} through \ref{cor:gevrey3}  extend to all times.
\end{proof}

\appendix

\section{$L^2$ type reformulation of the Boltzmann and Kac equations}\label{sec:appendix-reformulation}

A reformulation of the weak form \eqref{eq:weakformulation} of the Boltzmann and Kac equations is derived. We want to choose a suitable test function $\varphi$ in terms of the weak solution $f$ itself in the weak formulation of the Cauchy problem \eqref{eq:cauchyproblem}. We use $\varphi(t,\cdot):= G_\Lambda^2(t,D_v)f(t,\cdot)$ and since this involves a hard cut-off in Fourier space, we automatically have high regularity  of $\varphi(t,v)$ in the velocity variable, the question is to have $\mathcal{C}^1$ regularity in the time variable. For this we follow the strategy by \textsc{Morimoto} \textit{et al.} \cite{MUXY09}.

\begin{proposition}[]\label{prop:L2reformA}
	Let $f$ be a weak solution of the Cauchy problem \eqref{eq:cauchyproblem} with initial datum $f_0$ satisfying \eqref{eq:initialdata}, and let $T_0>0$. Then for all $t\in(0, T_0]$, $\beta>0$, $\alpha\in(0,1)$, and $\Lambda>0$ we have
$G_{\Lambda}f \in \mathcal{C}\left([0, T_0]; L^2(\R^d)\right)$ and
	\begin{align}\label{eq:reformulationappendix}
	\begin{split}
	&\frac{1}{2} \|G_{\Lambda}(t,D_v)f(t,\cdot)\|_{L^2(\R^d)}^2 - \frac{1}{2} \int_0^t \left\langle f(\tau, \cdot), \left( \partial_t G_{\Lambda}^2(\tau, D_v) \right) f(\tau,\cdot)\right\rangle \,\mathrm{d}\tau \\
	&= \frac{1}{2} \|\1_{\Lambda}(D_v)f_0\|_{L^2(\R^d)}^2 + \int_0^t \left\langle Q(f,f)(\tau, \cdot), G_{\Lambda}^2(\tau, D_v)f(\tau, \cdot)\right\rangle \, \mathrm{d}\tau.
	\end{split}
	\end{align}
\end{proposition}

To ensure that we can use $G_{\Lambda}^2f$ as a test function in the weak formulation of the Boltzmann equation, we need the following bilinear estimate on $Q(g,f)$, which is a special case of a larger class of functional inequalities by \textsc{Alexandre} \cite{Ale06,Ale09,AH08}.

\begin{lemma}[Functional Estimate on Collision Operator] \label{lem:functional}
	Assume that the angular collision cross-section $b$ satisfies assumptions \eqref{eq:cross-section}-\eqref{eq:cross-section2} or \eqref{eq:cross-section-kac}-\eqref{eq:cross-section-kac2}, respectively. Then for any $k>\frac{d+4}{2}$ there exists a constant $C>0$ such that
	\begin{align}\label{eq:functionalestimate}
		\| Q(g,f)\|_{H^{-k}(\R^d)} \leq C \|g\|_{L^1_{2}(\R^d)} \|f\|_{L^1_{2}(\R^d)}.
	\end{align}
\end{lemma}

\begin{proof}
This is a direct consequence\footnote{This result is proved in \cite{Ale09} for $d=3$, but the proof depends only on assumption \eqref{eq:cross-section} and general properties of Littlewood-Paley decompositions and holds in any dimension $d\ge 1$.} of Theorem 7.4 in \textsc{Alexandre}'s review \cite{Ale09}: under the assumptions on $b$, for any $m\in\R$ there exists a constant $\widetilde{C}>0$ such that
\begin{align*}
	\|Q(g,f)\|_{H^{-m}(\R^d)} \leq \widetilde{C} \|g\|_{L^1_{2\nu}(\R^d)} \| f\|_{H^{-m+2\nu}_{2\nu}(\R^d)}.
\end{align*}
Since $L^1(\R^d)\subset H^{-s}(\R^d)$ for any $s>\frac{d}{2}$, we obtain for $k>\frac{d+4}{2}$ and $\nu\in(0,1)$,
\begin{align*}
\|f\|_{H^{-k+2\nu}_{2\nu}(\R^d)} = \|\langle \cdot\rangle^{2\nu} f \|_{H^{-k+2\nu}(\R^d)} \leq C \|\langle \cdot\rangle^{2\nu} f\|_{L^1(\R^d)} \leq c \|\langle \cdot\rangle^{2} f\|_{L^1(\R^d)} = c\|f\|_{L^1_2(\R^d)}.
\end{align*}
i.e. $L^1_{2}(\R^d) \subset H^{-k+2\nu}_{2\nu}(\R^d)$ for any $k>\tfrac{d+4}{2}$ and $\nu\in (0,1)$.
Therefore,
\begin{align*}
	\|Q(g,f)\|_{H^{-k}(\R^d)} &\leq \widetilde{C} \|g\|_{L^1_{2\nu}(\R^d)} \|f\|_{H^{-k+2\nu}_{2\nu}(\R^d)} \leq C \|g\|_{L^1_{2}(\R^d)} \|f\|_{L^1_2(\R^d)}. 
\end{align*}
\end{proof}

Lemma \ref{lem:functional} implies that for $f,g\in L^1_2(\R^d)$, $\langle Q(g,f), h \rangle$ is well-defined for all $h\in H^k(\R^d)$, $k>\frac{d+4}{2}$, and one has $\langle Q(g,f), h \rangle  = \langle \widehat{Q(g,f)}, \widehat{h} \,\rangle_{L^2}$.

\begin{proof}[Proof of Proposition \ref{prop:L2reformA}]
Choosing a constant in time test function $\varphi(t,\cdot) = \psi \in \mathcal{C}^{\infty}_0(\R^d)$ in the weak formulation \eqref{eq:weakformulation} yields
\begin{align*}
	\int_{\R^d} f(t,v)\psi(v) \,\mathrm{d}v - \int_{\R^d} f(s,v)\psi(v)\,\mathrm{d}v = \int_s^t \langle Q(f,f)(\tau, \cdot), \psi\rangle \,\mathrm{d}\tau, \quad \text{for}\quad 0\leq s \leq t \leq T_0
\end{align*}
for all $\psi\in\mathcal{C}^{\infty}_0(\R^d)$ (this was already remarked by \textsc{Villani} \cite{Vil98} as an equivalent formulation of \eqref{eq:weakformulation}). By means of \eqref{eq:functionalestimate} this equality can be extended to test functions $\psi \in H^{k}$ for $k>\frac{d+4}{2}$, in particular one can choose $\psi = G_{\Lambda}^2 f(t, \cdot)$ and $\psi=G_{\Lambda}^2 f(s, \cdot)$ which, taking the sum of both resulting equations, yields
\begin{equation} \label{eq:Lip1}
  \begin{split}
	&\|G_{\Lambda}f(t, \cdot)\|_{L^2(\R^d)}^2 - \|G_{\Lambda}f(s,\cdot)\|_{L^2(\R^d)}^2
	   =  \left\langle f(t,\cdot), G_{\Lambda}^2 f(t,\cdot) \right\rangle - \left\langle f(s,\cdot), G_{\Lambda}^2 f(s,\cdot) \right\rangle \\
	 &=
	   \left\langle f(t,\cdot), \left( G_{\Lambda}^2(t,D_v) - G_{\Lambda}^2(s,D_v) \right) f(s,\cdot) \right\rangle
	    + \int_{s}^t \left\langle Q(f,f)(\tau, \cdot), G_{\Lambda}^2 f(t, \cdot) + G_{\Lambda}^2 f(s, \cdot) \right\rangle \,\mathrm{d}\tau.
  \end{split}
\end{equation}
Using Plancherel, the first term on the right hand side of \eqref{eq:Lip1} can be estimated by
\begin{align*}
	& \left| \left\langle f(t,\cdot),   \left( G_{\Lambda}^2(t,D_v) - G_{\Lambda}^2(s,D_v) \right) f(s,\cdot) \right\rangle \right|
	   = \left| \left\langle \hat{f}(t,\cdot),  \left( G_{\Lambda}^2(t,\cdot) - G_{\Lambda}^2(s,\cdot) \right) \hat{f}(s,\cdot) \right\rangle \right|  \\
	&\qquad  \leq \int_{\R^d} |\hat{f}(t,\eta)| \, |G_{\Lambda}^2(t,\eta) - G_{\Lambda}^2(s,\eta)|\, |\hat{f}(s,\eta)| \,\mathrm{d}\eta \\
	& \qquad \leq |t-s| \int_{\R^d} 2\beta \langle\eta\rangle^{2\alpha} G_{\Lambda}^2(t,\eta)\, \mathrm{d}\eta \,\|f(t,\cdot)\|_{L^1(\R^d)} \|f(s,\cdot)\|_{L^1(\R^d)} \leq C_{\Lambda,T_0} |t-s| \, \|f_0\|_{L^1(\R^d)}^2,
\end{align*}
and, using that the terms involving the collision operator can, for any $k>\frac{d+4}{2}$ (compare \eqref{eq:functionalestimate}), be bounded by
\begin{align*}
	|\langle Q(f,f)(\tau, \cdot), G_{\Lambda}^2f(t,\cdot)\rangle| &\leq \|Q(f,f)(\tau, \cdot)\|_{H^{-k}(\R^d)} \| G_{\Lambda}^2 f(t, \cdot)\|_{H^k(\R^d)} \\
	&\leq C \|f\|_{L^1_2(\R^d)}^2 \left( \int_{\R^d} \langle\eta\rangle^{2k} G_{\Lambda}^4(t,\eta) |\hat{f}(t,\eta)|^2 \, \mathrm{d}\eta \right)^{1/2} \\
	&\leq C \|f\|_{L^1_2(\R^d)}^2 \|f(t,\cdot)\|_{L^1(\R^d)} \left(\int_{\R^d} \langle\eta\rangle^{2k}G_{\Lambda}^4(T_0,\eta)\,\mathrm{d}\eta\right)^{1/2} \\
	&\leq C_{\Lambda,T_0}' \|f_0\|_{L^1_2(\R^d)}^2 \|f_0\|_{L^1(\R^d)}
\end{align*}
for any $t\in[0,T_0]$, yields
\begin{align*}
	\left| \int_{s}^t \langle Q(f,f)(\tau, \cdot), G_{\Lambda}^2f(t,\cdot) + G_{\Lambda}^2f(s,\cdot) \rangle \,\mathrm{d}\tau \right| \leq 2 C_{\Lambda, T_0}' |t-s| \, \|f_0\|_{L^1_2(\R^d)}^2 \|f_0\|_{L^1(\R^d)}.
\end{align*}
Plugging the latter two bounds into  \eqref{eq:Lip1} shows that $G_{\Lambda}f\in \mathcal{C}([0,T_0]; L^2(\R^d))$,
in fact, the map $[0,T_0]\ni t\mapsto \|G_\Lambda f(t,\cdot)\|_{L^2(\R^d)}$ is even Lipschitz continuous.
\bigskip

For any test function $\varphi\in\mathcal{C}^1(\R^+; \mathcal{C}_0^{\infty}(\R^d))$ the term involving the partial derivative $\partial_t\varphi$ in the weak formulation \eqref{eq:weakformulation} can be rewritten as
\begin{align*}
	\int_0^t \left\langle f(\tau, \cdot), \partial_{\tau}\varphi(\tau,\cdot)\right\rangle \,\mathrm{d}\tau = \lim_{h\to 0} \int_0^t \left\langle f(\tau, \cdot) + f(\tau+h,\cdot), \frac{\varphi(\tau+h,\cdot)-\varphi(\tau,\cdot)}{2h}\right\rangle\,\mathrm{d}\tau,
\end{align*}
since $f\in\mathcal{C}(\R^+; \mathcal{D}'(\R^d))$. The integral on the right hand side is well-defined even for $\varphi\in L^{\infty}([0,T_0]; W^{2,\infty}(\R^d))$, in particular for $\varphi = G_{\Lambda}^2f$, yielding
\begin{align*}
	&\int_0^t \left\langle f(\tau, \cdot) + f(\tau+h,\cdot) ,\frac{\varphi(\tau+h,\cdot)-\varphi(\tau,\cdot)}{2h}\right\rangle \,\mathrm{d}\tau \\
	&= \int_0^t \left\langle f(\tau, \cdot) + f(\tau+h,\cdot), \frac{G_{\Lambda}^2f(\tau+h,\cdot)-G_{\Lambda}^2f(\tau,\cdot)}{2h}\right\rangle \,\mathrm{d}\tau \\
	&= \frac{1}{2h} \int_0^t \left(\|G_{\Lambda}f(\tau+h,\cdot)\|_{L^2}^2-\|G_{\Lambda}f(\tau,\cdot)\|_{L^2}^2 \right)\,\mathrm{d}\tau \\ &\qquad + \int_0^t \left\langle f(\tau, \cdot), \frac{G_{\Lambda}^2(\tau+h, D_v) - G_{\Lambda}^2(\tau, D_v)}{2h} f(\tau+h,\cdot) \right\rangle \,\mathrm{d}\tau.
\end{align*}
Using $G_{\Lambda}f\in\mathcal{C}([0,T_0]; L^2(\R^d))$ it follows that
\begin{align*}
	&\frac{1}{2h} \int_0^t \left(\|G_{\Lambda}f(\tau + h, \cdot)\|_{L^2(\R^d)}^2 - \|G_{\Lambda}f(\tau, \cdot)\|_{L^2(\R^d)}^2 \right) \,\mathrm{d}\tau \\
	&= \frac{1}{2h} \int_t^{t+h} \|G_{\Lambda}f(\tau, \cdot)\|_{L^2(\R^d)}^2 \,\mathrm{d}\tau - \frac{1}{2h} \int_0^h \|G_{\Lambda}f(\tau, \cdot)\|_{L^2(\R^d)}^2\,\mathrm{d}\tau \\
	&\stackrel{h\to 0}{\longrightarrow} \frac{1}{2} \|G_{\Lambda}f(t,\cdot)\|_{L^2(\R^d)}^2 - \frac{1}{2}\|G_{\Lambda}f(0, \cdot)\|_{L^2(\R^d)}^2.
\end{align*}
where $\|G_{\Lambda}f(0, \cdot)\|_{L^2(\R^d)} = \|\1_{\Lambda}(D_v)f_0\|_{L^2(\R^d)}$.
For the second integral, an application of dominated convergence gives
\begin{align*}
	&\lim_{h\to 0} \int_0^t \left\langle f(\tau, \cdot), \frac{G_{\Lambda}^2(\tau+h, D_v) - G_{\Lambda}^2(\tau, D_v)}{2h} f(\tau+h,\cdot) \right\rangle \,\mathrm{d}\tau \\
	&\qquad = \frac{1}{2} \int_0^t \left\langle f(\tau, \cdot), \left(\partial_{\tau}G_{\Lambda}^2\right)(\tau, D_v) f(\tau,\cdot)\right\rangle \,\mathrm{d}\tau.
\end{align*}
Putting everything together, we thus have proved equation \eqref{eq:reformulationappendix}, i.e.
\begin{align*}
\begin{split}
	\frac{1}{2} \|G_{\Lambda} f\|_{L^2(\R^d)}^2 = \frac{1}{2} \|\1_{\Lambda}(D_v)f_0\|_{L^2(\R^d)}^2 &+ \frac{1}{2} \int_0^t \left\langle f(\tau, \cdot), \left(\partial_{\tau}G_{\Lambda}^2\right)(\tau, D_v) f(\tau,\cdot)\right\rangle \,\mathrm{d}\tau  \\ &+ \int_0^t \left\langle Q(f,f), G_{\Lambda}^2f \right\rangle \,\mathrm{d}\tau.
	\end{split}
\end{align*}
\end{proof}

\section{$H^{\infty}$ smoothing of the Boltzmann an Kac equations}\label{sec:appendix-hinfty}
We follow the strategy as in our proof of Gevrey regularity, with several simplifications. Of course, we \emph{do not} assume that $f_0$ is square integrable! We have
\begin{theorem}[$H^\infty$ smoothing for the homogeneous Boltzmann and Kac equation]\label{thm:Hinftysmoothing}
	Assume that the cross-section $b$ satisfies \eqref{eq:cross-section}-\eqref{eq:cross-section2} for $d\geq 2$, respectively \eqref{eq:cross-section-kac}-\eqref{eq:cross-section-kac2} for $d=1$, with $0<\nu<1$. Let $f$ be a weak solution of the Cauchy problem \eqref{eq:cauchyproblem} with initial datum satisfying conditions \eqref{eq:initialdata}. Then
	\begin{align}
		f(t,\cdot)\in H^\infty(\R^d)
	\end{align}
	for all $t>0$.
\end{theorem}

The proof is known, at least for the three dimensional Boltzmann equation see  \cite{MUXY09}, we give a proof for the convenience of the reader.  Again, one has to use suitable time-dependent Fourier multipliers.
Note that for $f_0\in L^1(\R^d)$ one has
	\begin{align*}
		\|f_0\|_{H^{-\gamma}(\R^d)} \le C_{d,\gamma} \|f_0\|_{L^1(\R^d)}
	\end{align*}
	with $C_{d,\gamma}= \left(\int_{\R^d}\langle \eta \rangle^{-\gamma}\,
						\mathrm{d}\eta \right)^{1/2}$
	which is finite for all $\gamma>d/2$. We choose $\gamma = d$, for convenience, and
	\begin{align*}
		M_\Lambda(t,\eta):= \langle \eta \rangle ^{-d} e^{\beta t\log\langle\eta\rangle} \1_\Lambda(|\eta|)
	\end{align*}
	as a multiplier. Then
	\begin{align*}
		\sup_{\Lambda > 0} \| M_\Lambda(0,D_v) f_0 \|_{L^2(\R^d)}
			= \| M_\infty(0,\cdot) \hat{f}_0\|_{L^2(\R^d)}
			= \|f_0\|_{H^{-d}(\R^d)}
			\le C_{d,d} \| f_0 \|_{L^1(\R^d)} 	
	\end{align*}
	
The proof of Proposition \ref{prop:L2reformA} carries over and we have
\begin{align}
	\begin{split}\label{eq:Hinfty1}
	&\frac{1}{2} \|M_{\Lambda}(t,D_v)f(t,\cdot)\|_{L^2(\R^d)}^2 - \frac{1}{2} \int_0^t \left\langle f(\tau, \cdot), \left( \partial_\tau M_{\Lambda}^2(\tau, D_v) \right) f(\tau,\cdot)\right\rangle \,\mathrm{d}\tau \\
	&= \frac{1}{2} \|M_\Lambda(0,D_v)f_0\|_{L^2(\R^d)}^2 + \int_0^t \left\langle Q(f,f)(\tau, \cdot), M_{\Lambda}^2(\tau, D_v)f(\tau, \cdot)\right\rangle \, \mathrm{d}\tau. \\
	\end{split}
\end{align}
and as in the proof of Corollary \ref{cor:gronwallbound}, we have
\begin{align}
\begin{split}\label{eq:Hinfty2}
	\langle Q(f,f), M_{\Lambda}^2f \rangle
	&=
	\langle Q(f,M_\Lambda f), M_{\Lambda}f \rangle
	+
	\langle M_\Lambda Q(f,f) - Q(f,M_\Lambda f), M_{\Lambda}f \rangle \\
	&\leq - \widetilde{C}_{f_0} \|M_{\Lambda}f\|_{H^{\nu}}^2
		+ C_{f_0} \|M_{\Lambda}f\|_{L^2}^2
		+ \langle M_{\Lambda}Q(f,f) - Q(f,M_{\Lambda}f), M_{\Lambda}f\rangle
\end{split}
\end{align}
The replacement of Proposition \ref{prop:ce} is
\begin{proposition}\label{prop:Hinftyce}
	The commutation error is bounded by
	\begin{align}\label{eq:ce-polynomial}
		\left|\langle  M_{\Lambda}Q(f,f) - Q(f,M_{\Lambda}f), M_{\Lambda}f\rangle\right|
		\leq  (1+2^{d-1})c_{b,d} \|f\|_{L^1}\left(\frac{d}{2} + \frac{\beta t}{2}2^{\beta t/2} \right) \|M_{\Lambda} f\|_{L^2}^2
	\end{align}	
	with the constant $c_{b,d}$ from Lemma \ref{lem:induction1}.
\end{proposition}

\begin{remark} Of course, for any weak solution $f$ of the Boltzmann and Kac equations, $\|f\|_{L^1}=\|f(t,\cdot)\|_{L^1} = \|f_0\|_{L^1}$.
 The fact that the commutator is bounded in terms of the $L^2$ norm of
 $M_\Lambda f$ makes the proof of $H^\infty$ smoothing for the Boltzmann and Kac equations much simpler than the proof of Gevrey regularity.
\end{remark}
\begin{proof}
  As in the proof of Proposition \ref{prop:ce}, Bobylev's formula shows
  \begin{align}
  \begin{split}\label{eq:easyce1}
  	&|\langle  M_{\Lambda}Q(f,f) - Q(f,M_{\Lambda}f), M_{\Lambda}f\rangle|
  	\le  \\
  	&\quad \le
  	 \int_{\R^d} \int_{\S^{d-1}} b\left(\frac{\eta}{|\eta|}\cdot\sigma\right) M_{\Lambda}(\eta) |\hat{f}(\eta)|
  	 |\hat{f}(\eta^-)|  |\hat{f}(\eta^+)|
  	 |M_\Lambda(t,\eta)- M_\Lambda(t,\eta^+)|
  	  \,\mathrm{d}\sigma \,\mathrm{d}\eta \\
  	&\quad \le
  	  \|\hat{f}\|_{L^\infty} \int_{\R^d} \int_{\S^{d-1}} b\left(\frac{\eta}{|\eta|}\cdot\sigma\right) M_{\Lambda}(\eta) |\hat{f}(\eta)|
     |\hat{f}(\eta^+)|
  	 |M_\Lambda(t,\eta)- M_\Lambda(t,\eta^+)|
  	  \,\mathrm{d}\sigma \,\mathrm{d}\eta
  \end{split}
  \end{align}	
    where, as before, $\eta^\pm = \frac{1}{2}(\eta \pm |\eta|\sigma)$.
  To bound $|M_\Lambda(\eta)- M_\Lambda(\eta^+)|$, we let $s:=|\eta|^2$ and $s^{+} = |\eta^{+}|^2$. Recall that $|\eta^+|^2= \frac{|\eta|^2}{2}(1+ \frac{\eta}{|\eta|}\cdot\sigma)$ and
  \begin{align*}
  	1-\frac{s^+}{s} = 1-\frac{|\eta^+|^2}{|\eta|^2} = \frac{1}{2}\left(1-\frac{\eta}{|\eta|}\cdot\sigma\right)
  \end{align*}
  Again, because of the support condition on the collision kernel $b(\cos\theta)$, we have $\frac{s}{2} \leq s^+ \leq s$. Set
  $\widetilde{M}(s):= (1+s)^{-d/2}e^{\frac{\beta t}{2}\log(1+s)}$.
  Then, for $|\eta|\le \Lambda$,
  \begin{align}
  \begin{split}\label{eq:easyce2}
  	M_\Lambda(\eta) & - M_\Lambda(\eta^+)
  	 = \widetilde{M}(s)- \widetilde{M}(s^+)
  	  = (1+s)^{-d/2}e^{\frac{\beta t}{2}\log(1+s)} - (1+s^+)^{-d/2}e^{\frac{\beta t}{2}\log(1+s^+)} \\
  	&=
  		(1+s)^{-d/2}\left(e^{\frac{\beta t}{2}\log(1+s)} - e^{\frac{\beta t}{2}\log(1+s^+)}\right)
  		+\left( (1+s)^{-d/2} - (1+s^+)^{-d/2}\right)e^{\frac{\beta t}{2}\log(1+s^+)} .
  \end{split}
  \end{align}
  Since $s\le 2s^+$, we have $ (1+s^+)^{-1}\le 2(1+s)^{-1}$. Hence
  \begin{align*}
  	\left| (1+s)^{-d/2} - (1+s^+)^{-d/2}\right|
  	& = \frac{d}{2}\int_{s^+}^{s} (1+r)^{-d/2-1}\, \mathrm{d}r
  	  \le  \frac{d}{2} (1+s^+)^{-d/2-1} (s-s^+) \\
  	& \le d (1+s^+)^{-d/2} \Big( 1-\frac{s^+}{s} \Big)
  \end{align*}
  In addition, $ \log(1+s)\le \log(2(1+s^+)) = \log2 + \log(1+s^+)$. So
  \begin{align*}
  	\left| e^{\frac{\beta t}{2}\log(1+s)} - e^{\frac{\beta t}{2}\log(1+s^+)} \right|
  	&\le \frac{\beta t}{2} \int_{s^+}^{s} \frac{1}{1+r}e^{\frac{\beta t}{2}\log(1+r)} \, \mathrm{d}r
  	  \le \frac{\beta t}{2} \frac{s}{1+s^+} e^{\frac{\beta t}{2}\log(1+s)}\Big(1-\frac{s^+}{s}\Big) \\
  	&\le \beta t 2^{\frac{\beta t}{2}}e^{\frac{\beta t}{2}\log(1+s^+)}\Big( 1-\frac{s^+}{s} \Big).
  \end{align*}
  Also   $ \log(1+s)\le \log(2(1+s^+)) = \log2 + \log(1+s^+)$. These bounds together with \eqref{eq:easyce2} show
  \begin{align*}
  	\left| M_\Lambda(\eta)  - M_\Lambda(\eta^+) \right|
  	\le \left( d + \beta t\, 2^{\frac{\beta t}{2}} \right)
  	     \left( 1-\frac{|\eta^+|^2}{|\eta|^2} \right)M_\Lambda(\eta^+)
  \end{align*}
  for all $|\eta|\le \Lambda$.
  Since the integration in \eqref{eq:easyce1} is only over $|\eta|\le \Lambda$, plugging this together with $\|\hat{f}\|_{L^\infty}\leq \|f\|_{L^1}$
  into \eqref{eq:easyce1} yields
  \begin{align*}
  \begin{split}
  	&|\langle  M_{\Lambda}Q(f,f) - Q(f,M_{\Lambda}f), M_{\Lambda}f\rangle|
  	  \\
  	&\quad \le
  	 \|f\|_{L^1} \left( d + \beta t \, 2^{\frac{\beta t}{2}} \right) \int_{\R^d}
  	 \int_{\S^{d-1}} b\left(\frac{\eta}{|\eta|}\cdot\sigma\right)
  	    \left( 1-\frac{|\eta^+|^2}{|\eta|^2} \right)
  	    M_{\Lambda}(\eta) |\hat{f}(\eta)| \, 
  	    M_\Lambda (\eta^+) |\hat{f}(\eta^+)|
  	  \,\mathrm{d}\sigma \,\mathrm{d}\eta.
  \end{split}
  \end{align*}
   Noting again 
   \begin{align*}
   	 	M_{\Lambda}(\eta) |\hat{f}(\eta)| \, M_\Lambda (\eta^+) |\hat{f}(\eta^+)| \leq \frac{1}{2} \left( (M_{\Lambda}(\eta) |\hat{f}(\eta)|)^2 + (M_\Lambda (\eta^+) |\hat{f}(\eta^+)|)^2 \right)
   \end{align*}
   and performing the same change of variables for the integral containing $\eta^+$ as in the proof of Lemma \ref{lem:ce} finishes the proof of equation \eqref{eq:ce-polynomial}.
\end{proof}

Now we can finish the
\begin{proof}[Proof of Theorem \ref{thm:Hinftysmoothing}]
	Using \eqref{eq:Hinfty1}, \eqref{eq:Hinfty2}, Proposition \ref{prop:Hinftyce}, and
	\begin{align*}
		\partial_\tau M_\Lambda(\tau,\eta)^2
		  = 2\beta \log\langle \eta \rangle \, M_\Lambda(\tau,\eta)^2
	\end{align*}
	 one sees
	\begin{align*}
	 \|M_{\Lambda}(t,D_v)f(t,\cdot)\|_{L^2}^2
	 &\le
	   \|f_0\|_{H^{-d}}^2
	   + 2 C_{f_0} \int_0^t \|M_{\Lambda}(\tau,D_v)f(\tau,\cdot)\|_{L^2}^2 \, \mathrm{d}\tau \\
	 &\quad +
	     \int_0^t \left\langle M_{\Lambda}(\tau,D_v)f(\tau, \cdot), \Big(\beta\log\langle D_v \rangle -2\widetilde{C}_{f_0}\langle D_v\rangle^{2\nu}\Big)M_{\Lambda}(\tau,D_v)f(\tau,\cdot)\right\rangle \,\mathrm{d}\tau \\
	 & \quad+ (1+2^{d-1}) c_{b,d} \|f_0\|_{L^1}
	 	\int_0^t\left(\frac{d}{2} + \frac{\beta \tau}{2}2^{\frac{\beta \tau}{2}} \right) \|M_{\Lambda}(\tau,D_v) f(\tau,\cdot)\|_{L^2}^2
    \end{align*}
	Setting
	\begin{align*}
		A(\beta,\tau) &:=
		 \sup_{\eta\in\R^d} \left(\beta\log\langle \eta \rangle - 2\widetilde{C}_{f_0}\langle \eta \rangle^{2\nu}\right)
		 	+ 2 C_{f_0}
		 	+ (1+2^{d-1})c_{b,d} \|f_0\|_{L^1} \left(\frac{d}{2} + \frac{\beta \tau}{2}2^{\frac{\beta \tau}{2}} \right) \\
		 	&= \frac{\beta}{2\nu} \left[ \log\left(\frac{\beta}{4\nu \widetilde{C}_{f_0}}\right) - 1\right] + 2 C_{f_0}
		 	+ (1+2^{d-1})c_{b,d} \|f_0\|_{L^1} \left(\frac{d}{2} + \frac{\beta \tau}{2}2^{\frac{\beta \tau}{2}} \right)
	\end{align*}
	the above can be bounded by
	\begin{align*}
		\|M_{\Lambda}(t,D_v)f(t,\cdot)\|_{L^2}^2
	 	\le
	      \|f_0\|_{H^{-d}}^2 + \int_0^t A(\beta,\tau) \|M_{\Lambda}(\tau,D_v)f(\tau,\cdot)\|_{L^2}^2 \, \mathrm{d}\tau
	\end{align*}
	and from Gronwall's lemma we get
	\begin{align*}
		\|M_{\Lambda}(t,D_v)f(t,\cdot)\|_{L^2}^2
		\le
			\|f_0\|_{H^{-d}}^2 \exp\left( \int_0^t A(\beta,\tau)\, \mathrm{d}\tau \right) .
	\end{align*}
	Letting $\Lambda\to\infty$ one sees
	\begin{align*}
		\|f(t,\cdot)\|_{H^{\beta t-d}}^2 = \|M_\infty(t,D_v)f(t,\cdot)\|_{L^2}^2
		\le
			\|f_0\|_{H^{-d}}^2 \exp\left( \int_0^t A(\beta,\tau)\, \mathrm{d}\tau \right) .
	\end{align*}
	that is, $f(t,\cdot)\in H^{\beta t-d}(\R^d)$. Now let $\beta\to\infty$ to see that $f(t,\cdot)\in H^{\infty}(\R^d)$ for any $t>0$.
\end{proof}

\begin{remark}
	Setting $\beta=\frac{\gamma+d}{t}$, one sees that $\|f(t, \cdot)\|_{H^{\gamma}(\R^d)} \lesssim t^{-\frac{\gamma+d}{4\nu}}$, so the $H^{\gamma}$ norms, in particular the $L^2$ norm, of $f(t, \cdot)$ blow up at most polynomially as $t\to 0$.
\end{remark}

\section{The Kolmogorov-Landau inequality}\label{sec:app-landau}
In this section we give a short proof of
\begin{lemma}[Kolmogorov-Landau inequality on the unit interval]\label{lem:landau-app}
	Let $m\geq 2$ be an integer. There exists a constant $C_m>0$ such that for all $w\in W^{m, \infty}([0,1])$,
	\begin{align*}
		\|w^{(k)}\|_{L^{\infty}([0,1])} \leq C_m \left( \frac{\|w\|_{L^{\infty}([0,1])}}{u^k} + u^{m-k} \|w^{(m)}\|_{L^{\infty}([0,1])} \right), \quad k=1, \dots, m-1,
	\end{align*}
	for all $0< u \leq 1$.
\end{lemma}
For the convenience of the reader, we give a short proof. The following argument is in part borrowed from \textsc{R. A. DeVore} and \textsc{G. G. Lorentz}'s book \cite{DL93} (pp.37--39).
\begin{proof}
	Since $w \in W^{m,\infty}([0,1])$, it has absolutely continuous derivatives of order up to $m-1$ and essentially bounded $m^{\mathrm{th}}$ derivative.
	
	Let $x\in [0, \tfrac{1}{2}]$ and $h\in[0, \tfrac{1}{2}]$. Then, by Taylor's theorem,
	\begin{align*}
		w(x+h) = w(x) + \sum_{j=1}^{m-1} \frac{h^j}{j!} w^{(j)}(x) + R_m(x,h)
	\end{align*}
	with remainder $R_m(x,h) = \int_0^h \frac{(h-t)^{m-1}}{(m-1)!} w^{(m)}(x+t)\,\mathrm{d}t$, which can be bounded by
	\begin{align*}
		|R_m(x,h)| \leq \|w^{(m)}\|_{L^{\infty}([0,1])} \int_0^h \frac{(h-t)^{m-1}}{(m-1)!}\, \mathrm{d}t = \frac{h^m}{m!} \|w^{(m)}\|_{L^{\infty}([0,1])}.
	\end{align*}
	
	Choosing $m-1$ real numbers $0<\lambda_1 < \lambda_2 < \cdots < \lambda_{m-1} \leq 1$ we obtain for $h\in [0, \frac{1}{2}]$ the system of equations
	\begin{align}\label{eq:kl-system}
		\sum_{j=1}^{m-1} \lambda_s^j \frac{h^j}{j!} w^{(j)}(x) = w(x+\lambda_s h) - w(x) - R_m(x, \lambda_s h) \quad \text{for } s = 1, \cdots, m-1.
	\end{align}
	Setting
	\begin{align*}
		V &= \begin{pmatrix} 	
		\lambda_1 & \lambda_1^2 & \cdots & \lambda_1^{m-1} \\
 		\lambda_2 & \lambda_2^2 & \cdots & \lambda_2^{m-1} \\
 		\vdots    &      &    \ddots    & \vdots \\
 		\lambda_{m-1} & \lambda_{m-1}^2& \cdots & \lambda_{m-1}^{m-1}
 			\end{pmatrix},
 		\quad \mathbf{w}(x) = \begin{pmatrix}
 			h w'(x) \\ \frac{h^2}{2} w''(x) \\ \vdots \\ \frac{h^{m-1}}{(m-1)!} w^{(m-1)}(x)
			 \end{pmatrix},\\
  		\mathbf{b}(x) &= \begin{pmatrix} w(x+\lambda_1 h) - w(x) - R_m(x, \lambda_1 h) \\ w(x+\lambda_2 h) - w(x) - R_m(x, \lambda_2 h) \\ \vdots \\ w(x+\lambda_{m-1} h) - w(x) - R_m(x, \lambda_{m-1} h)
  			\end{pmatrix},
	\end{align*}
	we have $V \mathbf{w}(x) = \mathbf{b}(x)$.
	Since the Vandermonde determinant
	\begin{align*}
	\det V = \prod_{i=1}^{m-1} \lambda_i \prod_{1\leq j<l \leq m-1} (\lambda_l - \lambda_j) \neq 0,
	\end{align*}
	$V$ is invertible and we obtain $\mathbf{w}(x) = V^{-1} \mathbf{b}(x)$ and therefore
	\begin{align}\label{eq:KL-normestimate}
		\left|\frac{h^{k}}{k!} w^{(k)}(x)\right| \leq \|\mathbf{w}(x)\| \leq \|V^{-1}\| \, \|\mathbf{b}(x)\|.
	\end{align}
	where $\|\cdot\|$ is any norm on $\R^{m-1}$, respectively the induced operator norm on the space of $(m-1)\times (m-1)$ real matrices. Choosing for concreteness the $\ell^{1}$ norm on $\R^{m-1}$, we have
	\begin{align*}
		\|\mathbf{b}(x)\| = \sum_{s=1}^{m-1} |w(x+\lambda_s h) - w(x) - R_m(x, \lambda_s h)| \leq (m-1) \left(2\|w\|_{L^{\infty}([0,1])} + \frac{h^m}{m!} \|w^{(m)}\|_{L^{\infty}([0,1])}\right).
	\end{align*}	
	While for our application the size of $\|V^{-1}\|$ is of no importance, one can even explicitly calculate it: The inverse of the Vandermonde matrix $V$ is explicitly known (see for instance \cite{Gau62}),
	\begin{align*}
			\left(V^{-1}\right)_{\alpha\beta} = (-1)^{\alpha-1} \frac{\sigma_{m-1-\alpha}^{\beta}}{\lambda_{\beta} \prod_{\nu\neq\beta} (\lambda_{\nu} - \lambda_{\beta})}, \quad \alpha, \beta = 1, \dots, m-1,
	\end{align*}
	where $\sigma_{i}^{j}$, $i,j=1, \dots,m-2$ is the $i^{\mathrm{th}}$ elementary symmetric function in the $(m-2)$ variables $\lambda_1, \dots, \lambda_{j-1}, \lambda_{j+1}, \dots, \lambda_{m-1}$,
	\begin{align*}
		\sigma_{i}^{j} = \sum_{\substack{1\leq \nu_1 < \cdots < \nu_i \leq m-1 \\ \nu_1, \dots, \nu_i \neq j}} \lambda_{\nu_1}\cdots \lambda_{\nu_i}, \qquad\sigma_0^{j} := 1.
	\end{align*}
	By means of the identity (Lemma 1 in \cite{Gau62})
	\begin{align*}
		\sum_{i=0}^{m-2} \sigma_{i}^j = \prod_{\substack{\nu=1\\\nu\neq j}}^{m-1}(1+\lambda_{\nu})
	\end{align*}
	which holds since the $\lambda_{\nu}$ are all positive, we get
	\begin{align*}
		\|V^{-1}\| &= \max_{1\leq \beta\leq m-1} \sum_{\alpha=1}^{m-1} \left| \left( V^{-1} \right)_{\alpha\beta}\right| = \max_{1\leq \beta\leq m-1} \frac{1}{\lambda_{\beta} \prod_{\nu\neq\beta} |\lambda_{\nu} - \lambda_{\beta}|} \sum_{\alpha=1}^{m-1} \sigma_{m-1-\alpha}^{\beta} \\
		&= \max_{1\leq \beta\leq m-1} \frac{1}{\lambda_{\beta}} \prod_{\substack{\nu=1 \\\nu\neq\beta}}^{m-1} \frac{1+\lambda_{\nu}}{|\lambda_{\nu} - \lambda_{\beta}|}.
	\end{align*}
	
	Going back to inequality \eqref{eq:KL-normestimate}, we have so far proved that
	\begin{align*}
		\frac{h^k}{k!} \left| w^{(k)}(x) \right| \leq (m-1) \|V^{-1}\| \left(2\|w\|_{L^{\infty}([0,1])} + \frac{h^m}{m!} \|w^{(m)}\|_{L^{\infty}([0,1])}\right),
	\end{align*}
	which yields
	\begin{align}\label{eq:KL-1}
	\begin{split}
		\left| w^{(k)}(x) \right| &\leq (m-1) \|V^{-1}\| \left(\frac{2 k!}{h^k} \|w\|_{L^{\infty}([0,1])} + h^{m-k} \frac{k!}{m!} \|w^{(m)}\|_{L^{\infty}([0,1])}\right) \\
		&\leq (m-1) \|V^{-1}\| \left(\frac{2 m!}{h^k} \|w\|_{L^{\infty}([0,1])} +  h^{m-k} \|w^{(m)}\|_{L^{\infty}([0,1])}\right)
		\end{split}
	\end{align}
	For $x\in[\tfrac{1}{2},1]$ the same calculations with $h$ replaced by $-h$ prove inequality \eqref{eq:KL-1} also in this case, so
	\begin{align}\label{eq:KL-2}
		\|w^{(k)}\|_{L^{\infty}([0,1])} \leq (m-1) \|V^{-1}\| \left(\frac{2 m!}{h^k} \|w\|_{L^{\infty}([0,1])} +  h^{m-k} \|w^{(m)}\|_{L^{\infty}([0,1])}\right)
	\end{align}
	for all $h\in[0, \tfrac{1}{2}]$. Taking an arbitrary $u\in[0,1]$, inequality \eqref{eq:KL-2} implies with $h=\frac{u}{2} \in [0, \tfrac{1}{2}]$,
	\begin{align*}
		\|w^{(k)}\|_{L^{\infty}([0,1])} &\leq 2^m m! (m-1) \|V^{-1}\| \left(\frac{1}{u^k} \|w\|_{L^{\infty}([0,1])} +  u^{m-k} \|w^{(m)}\|_{L^{\infty}([0,1])}\right),
	\end{align*}
	which is the claimed inequality with
	\begin{align}\label{eq:KL-constant}
		C_m = 2^m m! (m-1) \|V^{-1}\| = 2^m m! (m-1)  \max_{1\leq \beta\leq m-1} \frac{1}{\lambda_{\beta}} \prod_{\substack{\nu=1 \\\nu\neq\beta}}^{m-1} \frac{1+\lambda_{\nu}}{|\lambda_{\nu} - \lambda_{\beta}|}.
	\end{align}
\end{proof}

\begin{remark}
The constant $C_m$ in equality \eqref{eq:KL-constant} is far from optimal, but can be made small by minimising in the choice of the points $0<\lambda_1< \cdots< \lambda_{m-1}\leq 1$, suggesting that the optimal constant might be obtained by  methods from approximation theory.

Indeed, by a more refined argument making use of numerical differentiation formulas, the minimisers of the associated multiplicative Kolmogorov-Landau inequality, i.e., extremisers of
	\begin{align*}
		M_k(\sigma):=\sup\{ \|w^{(k)}\|_{L^{\infty}([0,1])}: w\in W^{m, \infty}([0,1]), \|w\|_{L^{\infty}([0,1])} \leq 1, \|w^{m}\|_{L^{\infty}([0,1])} \leq \sigma\}
	\end{align*}
	are explicitly known (at least for a wide range of parameters $m\in\N$ and $\sigma\geq 0$). The optimal Kolmogorov-Landau constants in these cases are given by the end-point values of certain Chebyshev type perfect splines. We refer to the papers by \textsc{A. Pinkus} \cite{Pin78} and \textsc{S. Karlin} \cite{Kar75}, as well as the recent article by \textsc{A. Shadrin} \cite{Sha14} and references therein.
\end{remark}

\section{Proof of Lemma \ref{lem:entropy}}\label{sec:appendix-LlogL}
\begin{proof}
	Let $f\in L^1_2(\R^d)\cap L\log L(\R^d)$ Then
	\begin{equation*}
		|H(f)| = \int_{\R^d} f \log_+ f \, \mathrm{d}v + \int_{\R^d} f\log_- f \,\mathrm{d}v.
	\end{equation*}
	The positive part is bounded by $\int f\log(1+f)\,\mathrm{d}v = \|f\|_{L\log L}$. The negative part can be controlled by
	\begin{align*}
		\int_{\R^d} f\log_- f\,\mathrm{d}v = \int_{\{f\leq 1\}}f\log \frac{1}{f}\,\mathrm{d}v \leq C_{\delta} \int_{\{f\leq 1\}} f^{1-\delta} \,\mathrm{d}v \leq C_{\delta} \left(\int_{\R^d} (1+|v|^2)^{-\frac{1-\delta}{\delta}} \,\mathrm{d}v\right)^{\delta} \|f\|_{L^1_2}^{1-\delta}
	\end{align*}
	which is finite for $0<\delta<\frac{2}{d+2}$, having used that for any $\delta>0$ there exists a constant $C_{\delta}$ such that $\log t \leq C_{\delta}t^{\delta}$ for all $t\geq 1$.

	Conversely, let $f\in L^1_2(\R^d)$ with finite entropy $H(f)$. Then
	\begin{align*}
		\int_{\R^d} f \log(1+f) \,\mathrm{d}v
		&=
		\int_{\{f\le 1\}} f \log(1+f) \,\mathrm{d}v
		+\int_{\{f>1\}} f \log(1+f) \,\mathrm{d}v
	\end{align*}
	On where $f\le 1$, we replace $f$ by $1$ and where $f>1$, we bound $1+f$ by $2f$
	leading to
	 	\begin{align*}
		\int_{\R^d} f \log(1+f) \,\mathrm{d}v
		&\leq
		\log2 \int_{\R^d} f \,\mathrm{d}v
		+\int_{\R^d} f \log f  \,\mathrm{d}v
		+  \int_{\R^d} f \log_- f  \,\mathrm{d}v
	\end{align*}
	As above, we conclude
	 \begin{align}\label{eq:LlogL}
		\int_{\R^d} f \log(1+f) \,\mathrm{d}v
		&\leq
		\log2 ||f\|_{L^1(\R^d)} + H(f) +  C_{\delta,d} \|f\|_{L^1_2(\R^d)}^{1-\delta}.
	\end{align}
	with a finite constant $C_{\delta,d}$ for $0<\delta<\frac{2}{d+2}$.
\end{proof}

\bigskip
\noindent\textbf{Acknowledgements.} We would like to thank Radjesvarane Alexandre for a discussion emphasising the question of smoothing properties of the  Boltzmann equation. It is a pleasure to  thank the REB program of CIRM for giving us the opportunity to start this research.
	J.-M.~B.\ was partially supported by the project SQFT ANR-12-JS01-0008-01. D.~H., T.~R., and S.~V.\ gratefully acknowledge financial support by the Deutsche Forschungsgemeinschaft (DFG) through CRC 1173. D.~H.\ also thanks the Alfried Krupp von Bohlen und Halbach Foundation for financial support.
	Furthermore, we thank the University of Toulon and the Karlsruhe Institute of Technology for their hospitality.

\bibliographystyle{aomplain}
\bibliography{boltzmann}

\vfill \noindent
\textbf{Jean-Marie Barbaroux}\\
\textsc{Aix-Marseille Universit\'e, CNRS, CPT, UMR 7332, 13288 Marseille, France}\\
et \textsc{Universit\'e de Toulon, CNRS, CPT, UMR 7332, 83957 La Garde, France}\\
\textit{E-mail}: \href{mailto:barbarou@univ-tln.fr}{barbarou@univ-tln.fr}\\
\textbf{Dirk Hundertmark}\\
\textsc{Karlsruhe Institute of Technology, Englerstra{\ss}e 2, 76131 Karlsruhe, Germany}\\
\textit{E-mail}: \href{mailto:dirk.hundertmark@kit.edu}{dirk.hundertmark@kit.edu}\\
\textbf{Tobias Ried}\\
\textsc{Karlsruhe Institute of Technology, Englerstra{\ss}e 2, 76131 Karlsruhe, Germany}\\
\textit{E-mail}: \href{mailto:tobias.ried@kit.edu}{tobias.ried@kit.edu}\\
\textbf{Semjon Vugalter}\\
\textsc{Karlsruhe Institute of Technology, Englerstra{\ss}e 2, 76131 Karlsruhe, Germany}\\
\textit{E-mail}: \href{mailto:semjon.wugalter@kit.edu}{semjon.wugalter@kit.edu}

\end{document}